\documentclass[11pt,a4paper]{article}

\usepackage{algorithm,algpseudocode,placeins}
\usepackage{enumitem,booktabs,tabularx,array}
\usepackage[T1]{fontenc}
\usepackage{lmodern}
\usepackage[margin=27mm,headheight=14pt]{geometry}
\usepackage{amsmath,amssymb,amsthm,mathtools,bm}
\usepackage{microtype}
\usepackage{enumitem,booktabs,tabularx,array}
\usepackage[numbers,sort&compress]{natbib}
\usepackage{xurl}
\usepackage{xcolor}
\definecolor{linkblue}{RGB}{26,69,109}
\usepackage[colorlinks=true,linkcolor=linkblue,citecolor=linkblue,urlcolor=linkblue,
 pdfauthor={},pdftitle={Tight Sampling Complexity in Every Fixed Dimension at All Noise Levels}]{hyperref}
\usepackage{fancyhdr}
\setlist{itemsep=3pt,topsep=5pt}
\numberwithin{equation}{section}
\newtheorem{theorem}{Theorem}[section]
\newtheorem{proposition}[theorem]{Proposition}
\newtheorem{lemma}[theorem]{Lemma}
\newtheorem{corollary}[theorem]{Corollary}
\theoremstyle{definition}
\newtheorem{definition}[theorem]{Definition}
\theoremstyle{remark}
\newtheorem{remark}[theorem]{Remark}
\theoremstyle{remark*}
\newtheorem{remark*}[theorem]{Remark}
\newcommand{\E}{\mathbb E}
\renewcommand{\P}{\mathbb P}
\newcommand{\R}{\mathbb R}
\newcommand{\TV}{\mathrm{TV}}
\newcommand{\ind}{\mathbf 1}
\newcommand{\eps}{\varepsilon}
\newcommand{\cH}{\mathcal H}

\renewcommand{\P}{\mathbb P}

\newcommand{\Shape}{\mathsf{T}}
\newcommand{\norm}[1]{\lVert#1\rVert}

\title{\textbf{Tight Sampling Complexity with Stochastic Gradient Oracles in Fixed Dimensions}}
\author{Weiming Ou\footnote{Shanghai University of Finance and Economics, email: ouweiming@163.sufe.edu.cn}\qquad Xiao Wang\footnote{MoE Key Laboratory of Interdisciplinary Research of Computation and Economics, Shanghai University of Finance and Economics, email: wangxiao@sufe.edu.cn}}
\begin{document}\maketitle
\begin{abstract}
We establish the stochastic-gradient query complexity of sampling smooth
strongly log-concave distributions in every fixed dimension $d\geq1$.
For $\mu$-strongly convex, $L$-smooth potentials with unknown minimizers $x_f^\star$
in the ball $\mathbb{B}(0,\mu^{-1/2})$, and unbiased gradient oracles with variance at most $\sigma^2$, the
minimax worst-case expected query complexity is $\Theta\left(\log(1+\kappa)+\frac{\sigma^2}{\mu\varepsilon}\right)$,
jointly optimal for the condition number $\kappa:=L/\mu$, variance $\sigma^2\ge 0$, and TV accuracy $0<\varepsilon\leq 1/10$. 

In the noiseless setting $\sigma=0$, this tight complexity $\Theta\left(\log(1+\kappa)\right)$ is independent of $\varepsilon$. Moreover, a gradient-only sampler generates an exact sample with ${O}(\log(1+\kappa))$ worst-case expected queries. Without previous ball $\mathbb{B}(0,\mu^{-1/2})$, exact sampling from any initial point $x_0$ can be implemented with expected cost ${O}\left(\log(1+\kappa)+\log(1+\sqrt{\mu}\|x_0-x_f^\star\|)\right)$, without knowing the initial distance. We also show that dependence on initial distance is generally unavoidable.
\end{abstract}
% \noindent\textbf{Keywords:} log-concave sampling; fixed dimension;
% stochastic gradient oracle; total variation; minimax complexity;
% shallow-cut ellipsoid method.
\clearpage\tableofcontents\clearpage

\section{Introduction}\label{sec:introduction}
Sampling from a density proportional to $e^{-f}$ is central to Bayesian
inference and uncertainty quantification. When full gradients are costly,
stochastic gradients offer a natural source of information about the target
\citep{welling2011sgld}. For $\mu$-strongly convex, $L$-smooth potentials,
a basic question is how many gradient queries are needed to attain
total-variation (TV) accuracy $\eps$. Both $\eps$ and the condition number
$\kappa=L/\mu$ enter nonasymptotic sampling guarantees
\citep{dalalyan2017lmc}.

Recent work has clarified the optimal accuracy dependence. In
\emph{High-accuracy log-concave sampling with stochastic queries},
\citet{chen2026stochastic} establish polylogarithmic accuracy dependence
under sub-exponential gradient noise. Under only a variance ceiling
$\sigma^2$, their strongly convex upper bound in $\R^d$ has order
$\widetilde O(\kappa\sqrt d\,[1+\sigma^2/(\mu\eps)])$, with logarithmic
and initialization factors suppressed. At fixed positive variance, their
lower bound matches the $1/\eps$ exponent but not the upper bound's
multiplicative $\kappa$ dependence. Thus accuracy optimality, up to
logarithmic factors, does not give joint optimality in $\eps$ and $\kappa$.

This distinction matters for ill-conditioned targets. An accuracy-optimal
bound can still overestimate the cost by a large factor when $\kappa$
grows, and it does not determine whether curvature must multiply the
statistical cost. The joint $\eps$--$\kappa$ dependence is left open by
these bounds. Resolving it requires upper and lower bounds that match as
both parameters vary, not merely the correct power of $1/\eps$.

Noise adaptivity is equally important. Gradient estimators can have very
different variances, and averaging reduces variance only by spending
additional queries. A useful complexity characterization should improve
as the variance ceiling $\sigma^2$ decreases, distinguish noise-dominated
from nearly noiseless regimes, and recover the exact-gradient benchmark
at $\sigma^2=0$. Fixing a positive noise level in an accuracy asymptotic
can hide this transition. Merely removing a stochastic term from a lower
bound can also leave no meaningful noiseless guarantee.

We therefore seek matching expected query bounds under unbiased,
variance-bounded gradient access without potential values. Here
\emph{noise-adaptive} means dependence on the supplied variance ceiling,
including zero; it does not mean that this ceiling is unknown to the
algorithm. This leads to our central question:
\begin{center}
\emph{Can we obtain tight sampling complexity jointly in $\eps$ and
$\kappa$, with a guarantee that also adapts to the noise level?}
\end{center}

\paragraph{Our contributions.}
We resolve this question in the \emph{fixed-dimensional setting}, under
the oracle, localization, and expected-cost conventions of Section~\ref{sec:model}.
\begin{enumerate}[leftmargin=*]
\item For every fixed $d\geq1$, we establish the minimax rate
$\Theta(\log(1+\kappa)+\sigma^2/(\mu\eps))$, jointly in $\eps$, $\kappa$,
and all $\sigma^2\geq0$, with constants depending only on $d$
(Theorems~\ref{thm:upper} and~\ref{thm:lower}).
\item We construct a gradient-only sampler with fully charged geometric
preprocessing and variance-controlled Poisson acceptance, including all
failed histories and transcript markers
(Algorithms~\ref{alg:full-sampler}--\ref{alg:one-trial} and
Section~\ref{sec:upper-sketch}).
\item We give a zero-noise exact sampler, its radius-free extension, and
a fixed-curvature lower bound isolating logarithmic worst-case location
cost (Section~\ref{sec:endpoints}).
\end{enumerate}

\subsection{Closest related work}\label{sec:related-main}
The bounded-variance comparison above uses \citet{chen2026stochastic},
Theorem~3.4(i) and Remark~5, with their lower bound in Proposition~4.1
and Remark~6. Their bounds already track noise strength; neither the
$1/\eps$ exponent nor the rare-informative-reply mechanism is new here.
Our contribution is joint fixed-dimensional optimality, without
multiplicative curvature or logarithmic losses, under the conditional
oracle and stopping conventions below.

For exact information, \citet{chewi2022scalar} obtain a
$\Theta(\log\log\kappa)$ scalar benchmark with a known mode and potential
values. \citet{chewi2024lower} give value-and-gradient query bounds and
fixed-dimensional ellipsoidal preprocessing. Their geometry also permits
accuracy-independent expected value-based exact rejection, as derived in
Appendix~\ref{app:value-gradient-benchmark}. Our noiseless contribution is
therefore gradient-only exact sampling with a coarsely localized unknown
mode, not accuracy-independent expected rejection cost alone.

Poisson rejection and randomized gradient integration have precedents in
the first-order framework of \citet{chen2026fors} and the exact-diffusion
constructions of \citet{beskos2006retrospective}. We analyze these principles
with noisy, history-dependent replies and explicitly charge geometric
preprocessing and random stopping. Appendix~\ref{app:related} gives the
extended comparison, including recent exact-gradient methods and the
relevant oracle, initialization, and error-metric distinctions.

\section{Model and main minimax results}\label{sec:model}\label{sec:results}
Fix an integer $d\geq1$ and public parameters $0<\mu\leq L<\infty$,
$\sigma^2\geq0$, and $0<\eps\leq1/10$. Norms are Euclidean and logarithms
are natural unless a base is displayed. With $d$ fixed, constants hidden
in $O,\Omega,\Theta$ may depend on $d$ but are uniform in the other
parameters unless stated otherwise. The notation $C_d$ denotes a finite
constant depending only on $d$ and may change between bounds. Write
$\mathbb{B}(a,r)=\{x\in\R^d:\norm{x-a}\leq r\}$ and
$\mathbb{B}_d=\mathbb{B}(0,1)$ for the closed Euclidean balls.

\begin{definition}[Potential class]\label{def:class}
The class $\mathcal C_d(L,\mu)$ consists of all $f\in C^1(\R^d)$ such that
\begin{equation}\label{eq:class}
 \frac\mu2\norm{x-y}^2
 \leq f(x)-f(y)-\langle\nabla f(y),x-y\rangle
 \leq\frac L2\norm{x-y}^2\qquad(x,y\in\R^d),
\end{equation}
and whose unique minimizer satisfies
$\norm{x_f^\star}\leq\mu^{-1/2}$. The target is
$\nu_f(dx)=Z_f^{-1}e^{-f(x)}dx$, where $Z_f=\int_{\R^d}e^{-f(x)}dx$.
\end{definition}
Strong convexity ensures $0<Z_f<\infty$. The curvature endpoints are
public bounds, not a requirement that every function attain them;
no second derivative is assumed. The ball localizes the unknown mode,
not the target's support. Additive constants in $f$ are irrelevant.
For initialization extensions, $\mathcal U_d(L,\mu)$ denotes the same
curvature class without localization, and
$\mathcal C_d(L,\mu;x_0,R)$ restricts its modes to $\mathbb{B}(x_0,R)$.

\paragraph{Oracle and stopping conventions.}\label{sec:oracle}
A physical query at $x_t$ returns a full vector $g_t$. Let
$\mathcal H_t^{\rm in}$ be the complete information $\sigma$-field after
$x_t$ is selected and immediately before its reply is generated. It
contains $x_t$, all previous queries and replies, and all internal
randomness already generated, but no future random draws.
Every admissible conditional kernel satisfies, almost surely,
\begin{equation}\label{eq:oracle}
 \E[g_t\mid\mathcal H_t^{\rm in}]=\nabla f(x_t),\qquad
 \E[\norm{g_t-\nabla f(x_t)}^2\mid\mathcal H_t^{\rm in}]\leq\sigma^2.
\end{equation}
This is a \emph{total vector} variance bound; successive replies need not
be independent or identically distributed. Potential values, Hessians,
proximal oracles, and convex-body sampling oracles are unavailable.
Exact real arithmetic and internal random generation are free.

The call count is a stopping time for the physical filtration
\begin{equation}\label{eq:physical}
 \mathcal F_t=\sigma(x_1,g_1,\ldots,x_t,g_t),
 \qquad\mathcal F_0=\{\varnothing,\Omega\}.
\end{equation}
The output may use private randomness. Since that randomness is not
itself in $\mathcal F_t$, we encode private branch decisions in charged
query locations before stopping. This stricter convention is explicit:
it accounts for the exact one-call lower bound below, and costs only a
constant-factor overhead in our construction.

\begin{remark*}
Conditioning on $\mathcal H_t^{\rm in}$ fixes the current query as well as
the randomness already used to select it. This matters because a fresh
random query need not be $\mathcal F_{t-1}$-measurable. For example, an
exact oracle for $f(x)=x^2/2$, queried first at either $1$ or $-1$ with equal
probability, satisfies~\eqref{eq:oracle}, whereas
$\E[g_1\mid\mathcal F_0]=0\ne\nabla f(x_1)$.
Fresh-data oracles satisfy~\eqref{eq:oracle} when their pointwise moments
obey these bounds and their new data are independent of the incoming
history. The assumption also permits history-dependent reply laws;
it does not require independent replies. The physical filtration $(\mathcal F_t)$
serves the separate stopping requirement.
\end{remark*}

\paragraph{Minimax quantity and normalization.}
Write $\norm{P-Q}_{\TV}=\sup_D|P(D)-Q(D)|$ for any two probability distributions $P$ and $Q$. The complexity is
\begin{equation}\label{eq:minimax}
 N^\star_{\TV}(d,L,\mu,\sigma,\eps)
 =\inf_{\mathcal A}\sup_{f\in\mathcal C_d(L,\mu),\,\mathcal O}
    \E_{f,\mathcal O,\mathcal A}[T],
\end{equation}
where the infimum is over measurable, almost surely terminating algorithms
obeying the information and stopping rules, with TV error at most $\eps$
for every $f$ and every oracle $\mathcal O$ satisfying~\eqref{eq:oracle}.
An empty infimum is $+\infty$. Normalize without changing query count by
\begin{equation}\label{up:normalization}
 z=\sqrt\mu\,x,\quad F(z)=f(z/\sqrt\mu),\quad G=g/\sqrt\mu,
 \quad\kappa=L/\mu,\quad A=\sigma^2/\mu.
\end{equation}
Then $F$ has curvature bounds $1,\kappa$, its mode $x_\star=\sqrt\mu x_f^\star$ is in $\mathbb{B}(0,1)$,
and normalized noise has total variance at most $A$. Inverse scaling
preserves TV and physical stopping. Throughout, put
\begin{equation}\label{eq:parameters}
 X=A/\eps,\qquad S=1+\log(1+\kappa)+X.
\end{equation}

\subsection{Sampling procedure and guarantees}\label{sec:algorithm-overview}
In this section, we describe the algorithms achieving the required upper bound, including Algorithms~\ref{alg:full-sampler}
and~\ref{alg:one-trial}. A call
$\mathsf{Oracle}(y)$ queries $y/\sqrt\mu$ and divides the reply by $\sqrt\mu$.
All calls are sequential and charged, including repeated locations.
Fresh random draws are generated only when needed, after preceding replies. The algorithms mainly consist of several ingredients below:

\paragraph{Geometric preprocessing, step 1: $\mathsf{CenterPilot}(\kappa,A,B,\delta_m)$.}
Given the center tolerance $B$ and failure budget $\delta_m$, initialize
$E=c_E+Q_E\mathbb{B}_d=2\mathbb{B}_d$.
At each round inspect the symmetric vertices $z_i^\pm=c_E\pm\gamma Q_Ee_i$,
where $\gamma>0$ is a fixed geometric parameter depending only on $d$.
A vertex outside $2\mathbb{B}_d$ gives the separating normal
$v=z_i^\pm$ without an oracle call. Otherwise form fresh means
$\bar g_{i,b}^\pm=b^{-1}\sum_{\ell=1}^b\mathsf{Oracle}(z_i^\pm)$
at all $2d$ vertices and the scores
\[
 \widehat s_{i,b}^\pm=\langle\bar g_{i,b}^\pm,z_i^\pm-c_E\rangle,
 \qquad \widehat M_b=\max\{1,\max_{i,\pm}\widehat s_{i,b}^\pm/m_0\}.
\]
Here $m_0>0$ is a score floor determined by $B$ and $d$.
Double the fresh batch size $b=1,2,4,\ldots$ until
$b\widehat M_b\geq\Lambda$, for a variance--confidence threshold $\Lambda$.
If every returned score is at most $2m_0$, return $m=c_E$.
Otherwise use the mean at a maximizing vertex as a cut normal $v$.
For either type of cut, replace $E$ by the prescribed enclosing ellipsoid of
$E\cap\{y:\langle v,y-c_E\rangle\leq2\gamma\|Q_E^\top v\|\}$ and repeat.
Fixed tie-breaking makes the choice unambiguous. Return $m=0$ if the
next full batch would exceed the call cap or the update cap is reached.
Symmetry lets small scores certify the center's potential gap, while a
large score selects a region to discard. Thus this step seeks only
$F(m)-F(x_\star)\leq B$ with failure probability at most $\delta_m$,
not optimization to accuracy $\eps$.

\paragraph{Geometric preprocessing, step 2: $\mathsf{RoundingPilot}(m,\kappa,A,\delta_q)$.}
Starting from $E=c_E+Q_E\mathbb{B}_d=\mathbb{B}(m,2)$, inspect the same
vertices. A vertex outside $\mathbb{B}(m,2)$ gives the free cut normal
$z_i^\pm-m$. For the remaining rounds, query instead the midpoints
$y_i^\pm=(m+z_i^\pm)/2$ and form scores
$\widehat s_{i,b}^\pm=\langle\bar g_{i,b}^\pm,y_i^\pm-m\rangle$.
Double fresh batches until $b\max\{1,\max_{i,\pm}\widehat s_{i,b}^\pm\}
\geq\lambda$. If all scores meet the prescribed small-score threshold,
contract the inscribed ellipsoid
$(m+c_E)/2+[\gamma/(2\sqrt d)]Q_E\mathbb{B}_d$ toward $m$
by the prescribed factor. Return its center $z\in\R^d$ and invertible
shape matrix $\Shape\in\R^{d\times d}$. Otherwise, use a maximizing
midpoint's mean as the shallow-cut normal at its original vertex and repeat.
Before a call exceeding the global budget, or at the round cap, return
$(z,\Shape)=(m,I_d)$.
The midpoint test is designed to enclose a whole level set, not just find
a low-value point. When the center and rounding succeed, the returned
ellipsoid satisfies
\[
 z+\Shape\mathbb{B}_d\subseteq K_1:=\{x:F(x)-F(m)\leq1\}
 \subseteq z+\rho\Shape\mathbb{B}_d.
\]
Thus $z$ locates the ellipsoid and $\Shape$ determines its shape;
$z$ need not equal $m$. The failure budget $\delta_q$ controls this
sandwich conditional on a valid center; no potential value is tested.

\paragraph{Proposal generation: $\mathsf{RadialProposal}(z,\Shape,\rho)$.}
From the returned geometry define
\[
 U(x)=\tfrac14\bigl(\|\Shape^{-1}(x-z)\|/\rho-1\bigr)_+,
 \qquad q(x)\propto e^{-U(x)}.
\]
Draw a uniform unit direction $\Theta$ and, independently, a radius $R$
with density proportional to $v^{d-1}e^{-(v-1)_+/4}$ for $v>0$;
return $x=z+\rho\Shape R\Theta$. A finite mixture of a unit-ball radial
law and shifted gamma laws generates $R$ without oracle calls.
The flat ellipsoidal region adapts the proposal to the level set, and
its tails give full support rather than discarding target mass.

\paragraph{Directional envelopes: $\mathsf{TwoGrids}(m,s,r,\kappa,A,\delta)$.}
For a proposal $x\ne m$, set $r=\|x-m\|$ and $s=(x-m)/r$.
Write $w(t)=\langle s,\nabla F(m+st)\rangle$ only to describe the
quantities being estimated. Starting at $t_0=0$, use $t_j=r(1-2^{-j})$ for $1\leq j\leq J$,
where $J\geq0$ is the smallest integer with $r2^{-J}\leq\kappa^{-1/2}$,
and add the terminal endpoint $r$. The inputs determine batch sizes $\beta_j\geq1$ and
paddings $e_j>0$; form the signed averages
\[
 \widehat w_{+,j}=\left\langle s,\frac1{\beta_j}
       \sum_{i=1}^{\beta_j}\mathsf{Oracle}(m+st_j)\right\rangle,
 \qquad
 \widehat w_{-,j}=\left\langle-s,\frac1{\beta_j}
       \sum_{i=1}^{\beta_j}\mathsf{Oracle}(m+s(r-t_j))\right\rangle.
\]
Use separate fresh batches: complete the positive grid and its terminal
node before the reflected negative grid. On the interval ending at
$t_j$, place height $\max\{0,\widehat w_{\pm,j}+e_j\}$, with the
terminal batch and padding treated likewise. At a shared boundary use the
larger adjacent height, and at an endpoint its one-sided value. Reflect the negative
rectangles back to $[0,r]$ and return $g_+,g_-$ and their areas
$Z_g=\int_0^r(g_++g_-)\,dt$, $Z_-=\int_0^r g_-\,dt$.
Padding protects against underestimation; reflection makes both signed
parts nondecreasing before rectangles are formed. With the calibrated
schedules, the envelopes dominate the corresponding derivative parts
except on an event of conditional probability at most $\delta$.

\paragraph{Trial control: $\mathsf{FiniteCategory}(Q,H)$ and $\mathsf{Position}(a_0/Q)$.}
Given the grids, a margin $c>0$, an offset $b_{\rm acc}>0$, and an
integer cap $H\geq1$, set $a_0=g_++g_-+2c/r$, $Q=Z_g+2c$, and
$v=\min\{1,e^{Z_-+c-b_{\rm acc}}\}$. A failed Bernoulli$(v)$ coin
rejects the proposal immediately. Otherwise $\mathsf{FiniteCategory}$
returns $N_c$ by assigning Poisson$(Q)$ masses to $0,\ldots,H$ and
the entire tail mass to $H+1$. Overflow rejects; zero accepts without
nodes. For each remaining node, $\mathsf{Position}$ selects the uniform
background with weight $2c/Q$ or an envelope rectangle with weight
its area divided by $Q$, then draws uniformly in the selected interval.
It returns $t\in[0,r]$ with density $a_0/Q$; this density remains fixed
throughout the trial. The positive background prevents a zero denominator,
and the overflow category keeps the number of processed nodes finite.

Given a correction batch size $n\geq1$, average $n$ new directional
gradients at $m+st$ as $\bar w$ and draw a
rejection bit with probability
$\operatorname{clip}_{[0,1]}((\bar w+g_-(t)-u'(t)+c/r)/a_0(t))$,
where $u(t)=U(m+st)$ and $\operatorname{clip}_{[0,1]}(a)=\min\{1,\max\{0,a\}\}$.
Compute $u'$ from $U$ outside its flat region and use zero on that region
and its boundary. Compensation by $g_-$ offsets negative directional
values, subtraction of $u'$ corrects the proposal, and clipping makes
the probability legal on every history. Reject at the first marked node;
accept after $N_c$ unmarked nodes. If $x=m$, bypass all segment operations
and use the coin $\min\{1,e^{U(m)-b_{\rm acc}}\}$ instead.

\paragraph{Recording and returning the result.}
$\mathsf{Marker}(y)$ makes a charged oracle call at $y$ in normalized
coordinates, waits, and discards the reply for estimation while retaining
it in the history. The markers record the proposal, category, and decisions, so stopping is
visible in the physical transcript. Bernoulli and category selections use
fresh open-interval uniforms. $\mathsf{OneTrial}$ returns $(a,x)$ with acceptance
bit $a$; Algorithm~\ref{alg:full-sampler} returns the first accepted
$x/\sqrt\mu$ within $K$ trials, or $m/\sqrt\mu$ after all trials reject.

\begin{algorithm}[!ht]
\caption{Gradient-only sampler with finite caps}
\label{alg:full-sampler}
\begin{algorithmic}[1]
\Require Public $d,L,\mu,\sigma^2,\eps$ and the stochastic gradient oracle
\Require Center tolerance $0<B\leq1/16$, expansion $\rho\geq1$,
         margin $c>0$, offset $b_{\rm acc}>0$
\Require Failure budgets $\delta_m,\delta_q,\delta\in(0,1)$;
         batch size $n\geq1$ and caps $K,H\geq1$ (integers)
\Ensure A sample $Y$ in the original coordinates
\State $\kappa\gets L/\mu$, $A\gets\sigma^2/\mu$; form $\mathsf{Oracle}$ as above.
\State $m\gets\Call{CenterPilot}{\kappa,A,B,\delta_m}$
\State $(z,\Shape)\gets\Call{RoundingPilot}{m,\kappa,A,\delta_q}$
\State Form the displayed $U,q$; fix $\Pi=(m,z,\Shape,q,U)$.
\For{$k=1,\ldots,K$}
  \State $(a,x)\gets\Call{OneTrial}{\Pi}$
         \Comment{Shared parameters remain fixed}
  \If{$a=1$}
    \State \Return $Y=x/\sqrt\mu$
  \EndIf
\EndFor
\State \Return $Y=m/\sqrt\mu$
       \Comment{After the final marker of trial $K$}
\end{algorithmic}
\end{algorithm}

\begin{algorithm}[!ht]
\caption{One trial, including every transcript marker}
\label{alg:one-trial}
\small
\begin{algorithmic}[1]
\Require Shared $\rho,\kappa,A,\delta,c,b_{\rm acc},n,H$ from Algorithm~\ref{alg:full-sampler}
\Function{OneTrial}{$\Pi=(m,z,\Shape,q,U)$}
  \State $x\gets\Call{RadialProposal}{z,\Shape,\rho}$;
         $\Call{Marker}{x}$
         \Comment{Reveal the proposal; wait for the reply}
  \State $r\gets\norm{x-m}$; $a\gets0$
  \If{$r=0$}
    \State $a\gets\operatorname{Bernoulli}(\min\{1,e^{U(m)-b_{\rm acc}}\})$
  \Else
    \State $s\gets(x-m)/r$
    \State $(g_+,g_-,Z_g,Z_-)\gets\Call{TwoGrids}{m,s,r,\kappa,A,\delta}$
    \State $a_0(t)\gets g_+(t)+g_-(t)+2c/r$;
           $Q\gets Z_g+2c$
    \State $v\gets\min\{1,e^{Z_-+c-b_{\rm acc}}\}$;
           $b_{\rm pre}\gets\operatorname{Bernoulli}(v)$
    \If{$b_{\rm pre}=0$}
      \State $\Call{Marker}{0e_1}$
             \Comment{Control marker: failed prefactor}
    \Else
      \State $N_c\gets\Call{FiniteCategory}{Q,H}$
             \Comment{$H+1$ denotes overflow}
      \State $\Call{Marker}{(N_c+1)e_1}$
             \Comment{Control marker: passed prefactor and category}
      \If{$N_c\le H$}
        \State $a\gets1$
               \Comment{Category $0$ has an empty node loop}
        \For{$j=1,\ldots,N_c$}
          \State $t\gets\Call{Position}{a_0/Q}$
          \State $\overline w\gets
                 \left\langle s,\frac1n\sum_{i=1}^n
                   \mathsf{Oracle}(m+st)\right\rangle$
                 \Comment{Exactly $n$ fresh calls}
          \State $b\gets\operatorname{Bernoulli}\!\left(
                 \operatorname{clip}_{[0,1]}
                 \frac{\overline w+g_-(t)-u'(t)+c/r}{a_0(t)}
                 \right)$
          \State $\Call{Marker}{be_1}$
                 \Comment{Node mark bit; wait for its reply}
          \If{$b=1$}
            \State $a\gets0$; \textbf{break}
          \EndIf
        \EndFor
      \Else
        \State $a\gets0$
               \Comment{Overflow: no correction calls}
      \EndIf
    \EndIf
  \EndIf
  \State $\Call{Marker}{ae_1}$
         \Comment{Final acceptance marker; wait for its reply}
  \State \Return $(a,x)$
\EndFunction
\end{algorithmic}
\end{algorithm}

\FloatBarrier
Stage order distinguishes coincident marker and data-query locations. Complete implementation is shown in Algorithms~\ref{alg:full-sampler} and~\ref{alg:one-trial}. The detailed setting of parameters in the two algorithms are all stated in the proof of Theorem~\ref{thm:upper}.

\begin{theorem}[Constructive upper bound]\label{thm:upper}
For every fixed $d\geq1$, a finite $C_d$ exists such that, for every
public tuple above, Algorithms~\ref{alg:full-sampler} and~\ref{alg:one-trial}
with the public parameter choices specified in the proof return $Y$ with
\begin{equation}\label{eq:upper-main}
 \norm{\mathcal L(Y)-\nu_f}_{\TV}\leq7\eps/16,
 \qquad\sup_{f,\mathcal O}\E[T]\leq C_dS.
\end{equation}
The supremum includes all admissible conditional oracles. The algorithm
terminates on every finite-real-reply path with its specified open-interval
uniforms, and $T$ is a physical-transcript stopping time. In particular,
$N^\star_{\TV}\leq C_dS$, including when $\sigma^2=0$.
\end{theorem}

\begin{theorem}[Minimax lower bound]\label{thm:lower}
For every fixed $d\geq1$ and every public tuple above,
\begin{equation}\label{eq:lower-main}
 N^\star_{\TV}(d,L,\mu,\sigma,\eps)
 \geq\max\{1,X/24,S/62\}.
\end{equation}
The bound permits arbitrary adaptive full-vector queries and random
stopping under the worst-case expected physical-query convention.
\end{theorem}

Together these theorems imply
\begin{equation}\label{eq:main}
 N^\star_{\TV}(d,L,\mu,\sigma,\eps)
 =\Theta\!\left(\log(1+\kappa)+\frac{\sigma^2}{\mu\eps}\right).
\end{equation}
The constants are jointly uniform in $\kappa$, $\sigma^2\geq0$, and
$0<\eps\leq1/10$. The ratio of the displayed upper and lower benchmarks
is at most $62C_d$. We prove the theorems in
Appendices~\ref{sec:upper} and~\ref{sec:lower}; the next two sections
explain their mechanisms. These are information bounds in fixed dimension,
not polynomial-time implementations or dimension-uniform optimal rates.

\section{Upper bound: geometric amortization and variance-controlled acceptance}
\label{sec:upper-sketch}
We sketch Theorem~\ref{thm:upper} in this section. The above Algorithms~\ref{alg:full-sampler} and~\ref{alg:one-trial} must avoid multiplying geometric search by statistical precision, control acceptance
error without potential values, and charge all failed histories. The next three subsections address
these requirements.
% The
% operations were specified in Section~\ref{sec:algorithm-overview};
% we now explain why geometry and noise have additive costs, and why
% randomized acceptance retains no integration bias at zero noise.

\subsection{Constant-energy geometry at additive cost}
\paragraph{Upper bound: idea.}
The pilots need only constant energy accuracy. A high-score round can
therefore use a less accurate gradient than a low-score round. What
matters is the sum of reciprocal energy scores, not a uniform accuracy
requirement at every geometric cut. Let $\mathcal M\geq1$ denote the
normalized true score: $\widetilde M$ for centering and $M$ for rounding
in Appendix~\ref{sec:upper}. With failure budget $\delta$, a successful
round uses $O(1+A/(\delta\mathcal M))$ calls and incurs failure
probability $O(\delta/\mathcal M)$
(Propositions~\ref{fx:center} and~\ref{fx:geometry}).

\paragraph{Proof sketch.}
A valid shallow cut contracts volume by a dimension-dependent constant.
At a round with score $\mathcal M$, convexity bounds the current
ellipsoid's volume by $C_d\mathcal M^d$ times that of the protected
sublevel set. After the first visit to $[2^j,2^{j+1})$, only $O(j+1)$
further valid cuts can preserve that set, regardless of subsequent
score increases. Hence
\begin{equation}\label{main:energy-budget}
 \sum_{\text{valid rounds}}\frac1{\mathcal M}
 \leq C_d\sum_{j\geq0}(j+1)2^{-j}<\infty.
\end{equation}
The total number of cuts is $O(\log(1+\kappa))$, so summing costs gives
$O(\log(1+\kappa)+A/\delta)$. The batching test does not know the true
score: an early stop requires a large score error, while inaccurate
later stages have summable dyadic probabilities. Martingale orthogonality
gives mean-square error at most $A/b$ for a fresh mean of $b$ replies.
The early-stop probability bound grows with $b$, whereas the late-error
bound decays as $1/b$. Summing over the respective dyadic ranges avoids
a logarithmic union penalty. First-failure events are summed only
along good prefixes, never by conditioning on future success. Public
caps extend the cost bound to failed histories.

Both pilots succeed with probability at least $1-\eps/4$.
For their proposal, put $I(x)=F(x)-F(m)$.
Proposition~\ref{fx:proposal} gives
\begin{equation}\label{main:ideal-acceptance}
 \alpha_*(x)=e^{-I(x)+U(x)-b_{\rm acc}}\leq1,
 \qquad \int q(x)\alpha_*(x)\,dx\geq p_0>0,
\end{equation}
where $p_0$ depends only on $d$. The inner ellipsoid supplies target
mass; convexity outside the outer ellipsoid controls the proposal's
tails. Since $q\alpha_*$ is proportional to $e^{-F}$, this is the desired
ideal rejection law. The same construction also bounds proposal
moments on failed pilot histories, which will matter when counting work.

\subsection{A compensated Poisson experiment instead of exponentiating an estimate}
\paragraph{Upper bound: idea.}
Exponentiating an unbiased estimate of $I$ introduces bias. Instead,
the absence of Poisson marks represents the exponential exactly;
random positions belong to the experiment rather than approximating
an integral. The negative derivative is compensated to keep the
marking intensity nonnegative.

\paragraph{Proof sketch.}
On a good pilot, the directional derivative $w$ introduced above satisfies
$\int_0^r|w(t)|\,dt\leq I(x)+2B$ and $r^2\leq4(I(x)+2B)$.
Right-endpoint rectangles on a dyadic grid are charged to neighboring
larger-value intervals, with only a constant terminal remainder.
The padding adds a constant mass and costs
$O(1+\log(1+\kappa r^2)+Ar^2/\delta)$ calls on every history
(Lemma~\ref{up:noisy-envelope}). This controls the positive and negative
envelope masses separately, rather than paying for the largest endpoint
slope along the whole segment.

For the computable $a_0,Q$ above, define only for analysis
$b_0=w+g_--u'+c/r$ and $\ell=\int_0^r b_0(t)\,dt$.
The proposal is flat while the target may decrease; after the ray
leaves that region its target derivative dominates the proposal
derivative (Lemma~\ref{up:proposal-ray-derivative}). On good pilot and
envelope histories this yields
\begin{equation}\label{main:poisson-slack}
 c/r\leq b_0(t)\leq a_0(t)-c/r,\qquad
 \ell\geq c,\quad r^2\leq8\ell,\quad Q\leq68\ell.
\end{equation}
Ideal marks have mean probability $p=\ell/Q\geq1/68$.
For $N\sim\operatorname{Poisson}(Q)$ independent ideal positions and
marks give no-mark probability $\E[(1-p)^N]=e^{-\ell}$.
Since $\ell=I+Z_--U+c$, the prefactor cancels the compensation:
\begin{equation}\label{main:poisson-identity}
 v e^{-\ell}=e^{-I+U-b_{\rm acc}}=\alpha_*(x).
\end{equation}
Here clipping of $v$ is inactive on the good event.

For the actual noisy marks, a two-sided margin bounds clipping bias by
variance divided by margin (Lemma~\ref{up:clipping-lemma}).
After integrating over the position density, the conditional mark error
is at most $\eta_0=Ar^2/(ncQ)$, with $\eta_0/p\leq32A/n$.
Comparing survival recursively with $(1-p)^j$ gives total error at most
$\eta_0\sum_{j\geq0}(1-p)^j=\eta_0/p$, not a factor proportional to $Q$.
This uses conditional bounds at reached histories, not independent oracle
replies. Thus $n=O(1+A/\eps)$ controls noise-induced bias, and this bias
vanishes at $A=0$. The survival estimate also controls work: even if
$Q$ is large, nodes after the first mark are never queried. The expected
number actually processed is bounded by a geometric sum, uniformly in
the generated Poisson count. No moment bound on $Q$ is required.
Finite caps incur separate errors.

\subsection{Accuracy, failed histories, and physical stopping}
The calibrated tolerances and caps make each reached good-pilot
trial's acceptance function differ from $\alpha_*$ by at most $\tau$
for proposal-almost every point. Normalization contributes at most
$\tau/(p_0-\tau)\leq\eps/8$ in TV. Mixing over the first acceptance
preserves this bound; the outer-cap fallback contributes $\eps/16$
and pilot failure $\eps/4$, giving $7\eps/16$.

On good pilots the expected number of trials is $O(1)$, and on good
envelopes the expected number of examined nodes is $O(1)$.
Failed envelopes cost at most $H$, offset by their small probability.
Failed pilots cost at most $KH(n+1)$ in correction calls and their
markers; $\eps[1+\log(1/\eps)]^2$ is uniformly bounded, so this term
also has expected order $1+A/\eps$. All-history proposal moments control
grid costs. Including every marker gives $\E[T]\leq C_dS$, not a cost
conditioned on success. The marker sequence makes each stopping decision
measurable in the physical transcript; ignored replies remain in all later
incoming histories (Section~\ref{up:operational-section}).

\section{Lower bound: noise and curvature encode different information}
\label{sec:lower-sketch}
We sketch Theorem~\ref{thm:lower}. Each scalar witness $h$ is embedded as
$h(x_1)+\tfrac12\norm{x_{2:d}}^2$ and supplied with exact gradients in
the added coordinates. Total noise variance is unchanged and projecting
the target onto its first coordinate preserves the scalar law. The
couplings and decision trees below are applied directly to full-vector
algorithms; no projected stopping-filtration assumption is needed.

\subsection{Rare informative replies force the noise term}
\paragraph{Hard instance: idea.}
A small linear tilt changes the target law while producing the same
gradient gap at every query point. Hiding that tilt behind a rare oracle
innovation prevents a remote or adaptive query from amplifying the
signal. Accuracy then requires enough physical queries to encounter
informative replies with appreciable probability.

\paragraph{Proof sketch.}
Let $\tau_0=4\eps$ and
$H_{\rm a}(u)=u^2/2+(\kappa-1)(|u|-8)_+^2/2$.
The two potentials $F_\pm=H_{\rm a}\pm\tau_0u$ attain the curvature
endpoints and have modes $\mp\tau_0\in[-1,1]$.
Their normalized target laws are separated by more than $8\eps/3$ in TV.
For $A>0$, let $p_{\rm r}=\tau_0^2/(A+\tau_0^2)$ and return
\[
 G_\pm(u)=H_{\rm a}'(u)\pm\frac{\tau_0}{p_{\rm r}}\xi,
 \qquad \xi\sim\operatorname{Bernoulli}(p_{\rm r})
\]
with fresh innovations. These replies are unbiased and have variance
exactly $A$. Couple both runs with the same private tape and innovations.
Their transcripts and outputs agree until a queried innovation is one.
Nonnegative summation at a random stopping time gives
$\E\bigl[\sum_{j\leq T}\xi_j\bigr]=p_{\rm r}\E[T]$.
Hence their output-law separation is at most $p_{\rm r}\E[T]$ in each
marginal. The target separation and two accuracy errors imply
$p_{\rm r}\E[T]>2\eps/3$, and therefore
$\E[T]\geq A/(24\eps)$
(Proposition~\ref{prop:noise-lower}). This does not impose a
fixed query horizon or a positive lower bound on $A$.

\paragraph{Where the $1/\eps$ scale comes from.}
The target perturbation is of size $\tau_0\asymp\eps$, so its gradient
gap is also $O(\eps)$. Under a variance ceiling $A$, that gap can be
hidden in innovations of probability on the order of
$\eps^2/(A+\eps^2)$. To separate the output laws by order $\eps$,
the coupled executions must split with probability of that order.
Dividing the required splitting probability by the innovation rate
gives $A/\eps$, rather than $A/\eps^2$. The exact calculation retains
the additive $\eps$ term and works for every positive $A$, even when
$A\ll\eps^2$. The distant curvature anchor ensures that changing
$\kappa$ does not alter this noise argument or let far-away queries
amplify the tilt.

\subsection{A finite-reply family forces logarithmic curvature cost}
\paragraph{Hard instance: idea.}
Exact real replies can normally carry unlimited information, so a packing
of modes alone is insufficient. We choose a family for which any fixed
query has at most three possible gradient replies. A sample nevertheless
identifies one of $\Theta(\sqrt\kappa)$ separated mass concentrations.
Only logarithmically many ternary decisions can reveal that index.

\paragraph{Proof sketch.}
For $s_0=\sqrt\kappa\geq640$, take
$M=\lfloor s_0/80\rfloor$ disjoint narrow ramps of width
$s_0/(s_0^2-1)$ with thresholds
$\theta_j=-1/4+40(j-1)/s_0$.
The derivative of the $j$th potential equals $u-s_0/2$ to the left,
$u+s_0/2$ to the right, and has slope $\kappa$ on its ramp.
Its slope elsewhere is one. Its mode lies in $(-1/4,1/4)$, and its
target gives mass greater than $99/100$ to an interval $D_j$ disjoint
from all the other decoding intervals.

At any real query, at most one ramp can be active: every remaining
reply is one of the two common outer values. Fixing the private tape
therefore produces a ternary transcript tree. Under a uniform family
index, TV accuracy gives decoding probability greater than $89/100$.
Truncating at $h=\lceil4\bar n\rceil$, where $\bar n$ is the
prior-average expected query count, loses at most $1/4$ by Markov's
inequality. At most $3^h$ leaves remain, and each fixed output decodes
at most one index. Thus $3^h/M>16/25$, which yields
\[
 N^\star_{\TV}\geq\tfrac14\bigl[\log_3(16M/25)-1\bigr]
 =\Omega(\log\kappa)
\]
in the large-curvature regime (Proposition~\ref{prop:curvature-lower}).
Exact gradients are admissible for every advertised variance ceiling.

Two separated Gaussian targets additionally force at least one physical
query, since the zero-call event belongs to the trivial $\mathcal F_0$
(Proposition~\ref{prop:one-call}). Combining the three inequalities,
using the constant bound when $\kappa<640^2$, gives $N^\star_{\TV}\geq S/62$.
This combines worst-case lower bounds, not costs on a single instance;
Appendix~\ref{sec:completion} records the constants.

\section{Exact sampling and the role of initialization}\label{sec:endpoints}
We first examine zero noise and then the quadratic endpoint $L=\mu$.
Appendix~\ref{up:operational-section} gives both arguments.

\begin{corollary}[Exact gradient-only sampling]\label{up:noiseless-exact}
For every fixed $d\geq1$, if $\sigma^2=0$ and
$f\in\mathcal C_d(L,\mu)$, an accuracy-free gradient-only sampler returns
an exact sample from $\nu_f$, terminates almost surely, and satisfies
$\sup_f\E[T]\leq C_d\log(1+\kappa)$ with physical-transcript stopping.
\end{corollary}

The exact sampler is uncapped; it is not obtained by setting $\eps=0$
in the finite theorem. The quadratic case admits a sharper bound.

\begin{corollary}[Quadratic endpoint]\label{lo:quadratic-case}
When $\kappa=1$, with $X=\sigma^2/(\mu\eps)$,
\[
 \max\{1,X/24\}\leq N^\star_{\TV}
 \leq\max\{1,\lceil X/2\rceil\}\leq1+X/2.
\]
At $\sigma^2=0$, exactly one gradient query suffices for exact sampling.
\end{corollary}

We next allow an arbitrary supplied ball, then remove the radius input
at zero noise; Appendix~\ref{app:initialization} gives both reductions.

\begin{corollary}[An arbitrary supplied ball]\label{cor:known-radius}
For $f\in\mathcal C_d(L,\mu;x_0,R)$ with $x_0,R$ supplied, let
$\Gamma_R=\max\{1,\mu R^2\}$. There is an $\eps$-TV sampler with
\begin{equation}\label{eq:known-radius-noisy}
 \sup_{f,\mathcal O}\E[T]
 \leq C_d\!\left[\log(1+\kappa\Gamma_R)
            +\frac{\Gamma_R\sigma^2}{\mu\eps}\right].
\end{equation}
At zero noise an almost surely terminating exact sampler satisfies
\begin{equation}\label{eq:known-radius-exact}
 \sup_f\E[T]\leq C_d\log(1+\kappa\Gamma_R)
 \leq C'_d\!\left[\log(1+\kappa)+\log(1+\sqrt\mu R)\right].
\end{equation}
Both algorithms use physical-transcript stopping.
\end{corollary}

No joint radius--noise optimality is claimed. With exact gradients,
a supplied radius is unnecessary for an instance-wise guarantee.

\begin{corollary}[No supplied radius at zero noise]\label{cor:no-radius-exact}
For fixed $x_0$ and every $f\in\mathcal U_d(L,\mu)$, exact gradients
suffice for exact sampling, without a radius or accuracy input, at
instance-wise expected cost
\begin{equation}\label{eq:unlocalized-exact-cost}
 \E_f[T]\leq C_d\!\left[\log(1+\kappa)
       +\log\bigl(1+\sqrt\mu\norm{x_0-x_f^\star}\bigr)\right].
\end{equation}
The sampler terminates almost surely and its query count is a physical
stopping time.
\end{corollary}

Finite instance-wise cost need not yield a finite uniform bound over
locations. A fixed-curvature family isolates this obstruction.

\begin{proposition}[Fixed-curvature location obstruction]\label{main:location}
At $\mu=1,L=2$, every $R\geq64$, fixed $d\geq1$, and
$0<\eps\leq1/10$, the noiseless expected minimax cost on
$\mathcal C_d(2,1;0,R)$ is $\Theta(\log(1+R))$.
On $\mathcal U_d(2,1)$ the uniformly accurate minimax expected cost is
infinite, even at fixed positive accuracy.
\end{proposition}

Appendix~\ref{app:localization-lower} proves this worst-case-over-locations
obstruction; it is not per-instance and does not apply when $L=\mu$.

\section{Conclusion}\label{sec:conclusion}
We characterized fixed-dimensional gradient-only sampling by the joint
expected query bound $\Theta(\log(1+\kappa)+\sigma^2/(\mu\eps))$.
At zero noise, an uncapped variant gives exact sampling with almost-sure
termination. The bounds allow dimension-dependent constants and do not
assert an accuracy-independent deterministic cap. Sharper radius--noise
dependence and finite-precision implementations remain natural directions.
\label{main:end}

\clearpage
\begingroup
\interlinepenalty=10000
\renewcommand{\bibfont}{\small}
\setlength{\bibsep}{4pt plus 1pt}
\renewcommand{\bibsection}{\section*{\refname}\addcontentsline{toc}{section}{\refname}}
\bibliographystyle{plainnat}
\bibliography{all_noise_references}
\appendix
\section{Extended related work}\label{app:related}

The complexity of sampling depends on the information supplied by the oracle, the approximation metric, the initialization, and the convention for counting queries. These distinctions are essential when dimension is fixed and other parameters vary. A value oracle can evaluate an acceptance ratio directly, a known minimizer removes a localization problem, and a fixed query budget is different from a worst-case expected stopping cost. We organize the comparison around these distinctions. Throughout this discussion, $\kappa=L/\mu$ and $A=\sigma^2/\mu$; logarithmic factors are suppressed only when reporting a bound from the literature using $\widetilde O$ notation.

The geometric approach of \citet{lovasz2007geometry} analyzes ball-walk and hit-and-run sampling for log-concave densities without local smoothness assumptions. Its information model supplies density values up to a common multiplicative constant, together with geometric bounds, and its analysis includes preprocessing toward isotropic position. This is foundational background for using geometry to construct efficient sampling procedures. Direct density evaluation and geometric initialization differ from the noisy-gradient information available here; their cost guarantees therefore cannot be transferred merely by fixing the ambient dimension.

\paragraph{The closest stochastic-oracle result.}
\citet{chen2026stochastic} already establish the inverse-accuracy exponent under bounded variance. Their Theorem~3.4(i) and Remark~5 give the smooth log-Sobolev upper bound, specializing under strong convexity to
\[
\widetilde O\!\left(\kappa\sqrt d\left(1+\frac{A}{\varepsilon}\right)\right).
\]
Proposition~4.1 and Remark~6 supply the scalar Gaussian rare-reply lower bound with a fixed horizon. Their upper bound allows small variance; neither its inverse-accuracy exponent nor that lower-bound mechanism is a contribution here.

Finite initial divergence is already available under our localization
promise: the query-free law $q_0=N(0,L^{-1}I_d)$ satisfies
\[
 1+\chi^2(q_0\Vert\nu_f)\le\kappa^{d/2}e^\kappa.
\]
Indeed, $Z_fe^{f(x_f^\star)}\le(2\pi/\mu)^{d/2}$ and
$f(x)-f(x_f^\star)\le L\norm{x-x_f^\star}^2/2$ give this bound by
completing the square in $\int q_0^2/\nu_f$ and using
$L\norm{x_f^\star}^2\le\kappa$. Thus initialization alone is not a
barrier to applying that fixed-kernel benchmark here; the joint rate
and the oracle and stopping conventions are the relevant distinctions.

Our comparison concerns the joint minimax rate and the operational model. We prove matching constant-factor bounds for $1+\log(1+\kappa)+A/\varepsilon$, for every $A\ge0$, with constants depending only on the fixed dimension. The mode is unknown, with only a prescribed localization ball. The response distribution may change with the entire history, provided that its conditional mean and total variance obey the oracle conditions. The TV algorithm has worst expected physical-call cost of this order and terminates on every real-valued execution path. At $A=0$, an uncapped variant additionally produces an exact sample with $O(1+\log(1+\kappa))$ expected calls and almost-sure termination. These are distinct stopping statements. Restricting an existing theorem to fixed dimension does not supply either the joint rate or its implementation in this model.

Noise tails are a separate assumption. The sub-exponential branch of \citet[Remark~5]{chen2026stochastic} has only polylogarithmic accuracy dependence, with a separate tail proxy. A small variance bound alone does not imply such a tail bound: at every positive $A$, the admissible class still contains rare, large replies. The transition to $A=0$ therefore has to be analyzed within the variance-only model. Our acceptance construction retains a noise error proportional to $A$, without treating randomness in the integration locations as additional oracle noise.

\paragraph{Scalar sampling and query lower bounds.}
\citet{chewi2022scalar} obtain the known-mode scalar $\Theta(\log\log\kappa)$ rate. Theorem~3 uses values to construct a rejection envelope, followed by constant expected additional cost for an exact sample; Theorem~2 permits values, gradients, and second derivatives in its lower bound. Their Appendix~C gives bisection-based localization for a conditional sampling subproblem with an unknown mode. Its upper construction still uses values. Thus the double-logarithmic rate does not contradict the localization cost proved here. Likewise, the logarithmic accuracy cost of their capped rejection procedure is a maximum-budget statement; it does not replace their constant expected additional cost.

\paragraph{Exact-information fixed-dimensional sampling.}
For fixed $d\geq2$, \citet[arXiv v2, Theorem~49]{chewi2024lower}
give a known-mode value-and-gradient algorithm with a query budget
$O(\log(1+\kappa)+\log^{O(d)}(1/\varepsilon))$.
Their proof truncates the rejection procedure. In contrast, the
quantity~\eqref{eq:minimax} measures expected cost.
Appendix~\ref{app:value-gradient-benchmark} derives a full-support
rejection sampler from their ellipsoidal sublevel-set approximation.
This gives exact sampling with $O(\log(1+\kappa))$ expected
value-and-gradient queries; truncation yields a deterministic budget
$O(\log(1+\kappa)+\log(1/\varepsilon))$.
These are consequences proved here for comparison, not statements
attributed to their Theorem~49.

Our additional noiseless guarantee is gradient-only exact sampling:
Corollary~\ref{up:noiseless-exact} attains the optimal expected order with a coarsely
localized unknown mode, and Corollary~\ref{cor:no-radius-exact}
requires no supplied radius, at logarithmic instance-wise
initialization cost. Positive-noise guarantees additionally cover
history-dependent, conditionally unbiased variance-bounded replies.
The ellipsoid geometry is classical; the gradient-only pilots, their
summable stochastic costs, and value-free acceptance are the
additional components analyzed here. We neither assume sublevel
membership access nor simulate every geometric query at a common
noise precision.

The lower-bound oracle comparison has the opposite direction.
\citet[arXiv v2, Theorem~4]{chewi2024lower} establish
an $\Omega(\log\kappa)$ lower bound at fixed constant TV accuracy
already in dimension two, with known mode and even when values and
gradients are available. Our scalar curvature family encodes an
unknown mode, and its proof treats expected gradient-query cost.
It proves optimality for our stated model, including $d=1$, rather
than a stronger lower bound for their more informative model.
Appendix~\ref{app:localization-lower} separately isolates location
hardness at fixed curvature. Scalar witnesses are completed with
Gaussian coordinates and analyzed directly for full-vector algorithms.

\citet{chatterji2022lower} develop a stochastic-gradient lower-bound
framework using statistical decision theory. In our total-variance
convention, their journal Theorem~4.1 (Theorem~1 in arXiv v3)
gives TV error at least
$\sigma/(16\sqrt n)$ when $n\ge\sigma^2/4$ and
$\alpha\le\sigma^2/(256n)$, over $\alpha$-smooth,
$\alpha/2$-strongly convex potentials whose target mean has norm at
most $1/\alpha$. Their notation instead writes the total variance as
$d\sigma_{\rm paper}^2$. The inverse-square accuracy dependence must
be read with the curvature restriction: it does not give such a
lower bound at arbitrarily small accuracy with a fixed positive
strong-convexity parameter, and so does not contradict the present
inverse-linear rate.

\paragraph{Rejection sampling with limited information.}
Adaptive rejection sampling has long exploited the geometry of univariate log-concave densities. \citet{gilks1992ars} use tangent bounds on the log-density to construct piecewise exponential proposals, refining the envelope as sampling proceeds. This establishes the classical role of shape-adapted proposals and acceptance bounds in scalar sampling. In that setting, the evaluations used to construct and test the envelope are exact. With the oracle studied here, a reported gradient cannot be treated as a deterministic supporting slope. Our preprocessing therefore controls its own failure probability, and our acceptance analysis separately accounts for clipping noisy marking probabilities.

The first-order rejection sampling framework, FORS, of \citet{chen2026fors} is a direct methodological predecessor. Section~3.1 and Theorem~3.1 use bounded log-ratio estimators and a Poisson product to sample their exponential tilt exactly; randomized gradient integration supplies potential differences. Their Gaussian-tilt instantiation clips a generally unbounded estimator and approximates the original target. The smooth-LSI application in the extended version
(arXiv v2, Theorem~G.1(i)) assumes exact gradients and proximal access, fixed R\'enyi order $\lambda\ge2$, and finite initial $R_\lambda$. For $0<\varepsilon\le1/2$, it attains $R_\lambda\le\varepsilon^2$; its strongly convex specialization costs $\widetilde O(\kappa\sqrt d)$ expected calls to those oracles, retaining accuracy and initialization logarithms. Primitive exactness is therefore different from an exact-target complexity theorem.

The Poisson acceptance identity also has classical precedents. \citet{beskos2006retrospective} simulate an exponential of a nonnegative path integral by testing whether a Poisson process places any points under its graph. Their exact diffusion algorithms reveal only the required path skeleton, under drift regularity and implementable path-envelope assumptions. This is an exact-simulation principle, not a strongly convex gradient-query bound. Our proof applies the same no-point identity on a deterministic segment, with envelopes constructed from gradient replies and with their full query costs included.

The additional acceptance analysis here is quantitative. Compensating the negative directional gradient produces a nonnegative intensity. A fixed slack keeps its normalized marking probabilities inside $[0,1]$, so clipping an averaged noisy mark introduces bias controlled by the oracle variance. A survival-probability recursion bounds the accumulated error under history-dependent replies. Early rejection controls the expected number of processed nodes, including when the proposal-dependent Poisson intensity is large. At zero variance the uncapped experiment with exact marks recovers exact acceptance. We claim neither Poisson rejection, gradient-path integration, nor exact simulation from incomplete numerical information as new principles.

Bernoulli factories provide another perspective on acceptance decisions from incomplete numerical information. The existence theory of \citet{keane1994factory} and the quantitative constructions of \citet{nacu2005coins} study simulation of a coin with probability $h(p)$ from independent coins of unknown probability $p$. The latter obtain fast simulation for analytic functions on suitable compact intervals. These results distinguish exact simulation of a transformed probability from applying a nonlinear function to an estimator. Their bounded independent input coins are not automatically supplied by a real-valued conditional variance oracle. Our proof uses a direct Poisson marking experiment and specifies its bias and stopping costs, without assuming an additional Bernoulli-factory oracle.

A related distinction appears in the pseudo-marginal approach of \citet{andrieu2009pseudo}: appropriate nonnegative unbiased estimates of an unnormalized density can be incorporated into an augmented Markov chain while preserving the desired marginal stationary distribution. An unbiased estimate of a potential difference is a different object, since exponentiation does not preserve unbiasedness. Moreover, invariance of a Markov chain and a finite-query total-variation guarantee for its output are separate requirements. These observations explain why the integral estimator alone is insufficient here. The acceptance bias, the probability of reaching each trial, and the number of calls before termination all enter the complexity proof.

\paragraph{Langevin methods and stochastic gradients.}
The nonasymptotic analysis of Langevin algorithms forms a major part of the broader sampling literature. \citet{dalalyan2017lmc} derive explicit approximation guarantees for Langevin Monte Carlo under smooth log-concavity, using comparisons with continuous-time diffusion. \citet{durmus2017ula} analyze the unadjusted Langevin algorithm in total variation, including constant and decreasing step sizes. These results show how smoothness, curvature, dimension, initialization, and discretization interact in concrete sampling schemes. Their algorithm-specific bounds are not lower bounds on all possible gradient-query algorithms. In the present fixed-dimensional problem, rejection from a directly constructed proposal replaces the task of following a discretized diffusion for a prescribed number of steps.

\citet{welling2011sgld} introduce stochastic gradient Langevin dynamics for Bayesian learning from minibatches. Their construction combines stochastic optimization updates with injected Gaussian noise and a decreasing step size. It motivates the computational setting in which full gradients are expensive and random gradient estimates are available. The minimax question considered here fixes the dimension and is more abstract in its oracle specification: there need not be a dataset, a finite-sum decomposition, or access to individual likelihood terms. The theorem concerns the marginal distribution of one output after a finite, explicitly accounted sampling procedure, rather than an asymptotic annealing statement.

\citet{dalalyan2019inaccurate} give quantitative Wasserstein guarantees for Langevin Monte Carlo with inaccurate gradients, encompassing stochastic and deterministic errors. \citet{durmus2019convex} develop a convex-optimization analysis of Langevin Monte Carlo and extensions involving stochastic gradients. These works make clear that gradient error must be incorporated into the sampling analysis instead of being treated as an implementation detail. The metric remains important in comparisons: a Wasserstein estimate does not, without a further argument, give the total-variation accuracy required here. Likewise, a result for a particular noisy Langevin recursion does not determine the optimum over adaptive algorithms that can choose batches, integration locations, proposals, and stopping decisions.

Variance reduction can substantially change the information available to a sampler. \citet{chatterji2018variance} analyze SAGA, SVRG, and control-variate variants of stochastic-gradient Monte Carlo in Wasserstein distance under smoothness, strong convexity, and additional regularity assumptions. Their constructions exploit finite-sum structure and correlated evaluations of component gradients. Such access is useful, but is not supplied by an arbitrary conditional variance bound. Our oracle may alter its response law after every call and need not permit the same random component to be evaluated at two locations. We therefore cannot assume the cancellation identities underlying a finite-sum control variate. Conversely, our worst-case guarantee does not describe the additional gains possible when this structure is available.

\paragraph{Proximal sampling, warm starts, and recent exact-gradient methods.}
\citet{lee2021rgo} introduce the restricted Gaussian oracle as a building block for structured log-concave sampling. It samples from the target after adding a quadratic localization term and supports reductions for composite and finite-sum potentials. Their convention for a first-order oracle includes both a potential value and its gradient, which is stronger than the gradient-only access assumed here. \citet{chen2022proximal} extend the analysis of the resulting proximal sampler to weak log-concavity and functional-inequality assumptions, and relate it to optimization in Wasserstein space. These reductions explain the importance of accurately sampling localized subproblems, but implementing the subproblem remains part of an oracle-complexity theorem.

\citet{fan2023proximal} improve the dimension dependence of proximal sampling by refining the implementation of the restricted Gaussian oracle, obtaining high-accuracy guarantees under several regularity regimes. Their strongly log-concave results achieve an $\widetilde O(\kappa\sqrt d)$ dependence, with the corresponding accuracy factors, and analyze initialization as well as approximate rejection. This is a multivariate algorithmic benchmark, not a joint fixed-dimensional stochastic-oracle minimax theorem. Our proof also constructs a region on which rejection is efficient, but does so once for the original target. It does not run an outer proximal chain or assume a pre-existing implementation of a restricted Gaussian oracle.

Initialization can be a substantive computational cost even for exact-gradient methods. \citet{altschuler2024warm} construct algorithmic warm starts for high-accuracy log-concave sampling, combining R\'enyi-divergence control for underdamped Langevin discretizations with subsequent high-accuracy sampling. This addresses the gap between a mixing result from a favorable initial law and an implementable algorithm that must obtain such a law. Our initialization requirement is instead a deterministic ball containing the unknown minimizer. The pilot converts this information into a point with a constant potential gap and then into a proposal with constant overlap. Its cost is included in the displayed minimax bound.

More recently, \citet{chen2026diffusions} use rejection on diffusion path space with Girsanov-based ratio estimators. Their Theorem~3.2(ii) assumes a $C^2$ convex, smooth potential satisfying an LSI, and controls a fixed-order R\'enyi divergence from a supplied phase-space initialization. With suitable initial divergence, its exact-gradient complexity is $\widetilde O(\kappa^{2/3}d^{1/3}\operatorname{polylog}(1/\varepsilon))$; additional derivative assumptions yield further improvements. The target-law guarantee retains a positive accuracy parameter. This multivariate result does not implement an arbitrary conditional variance oracle or supply the zero-noise fixed-dimensional exact-target statement proved here.

The September 2026 preprints of
\citet{chen2026bps,chen2026picard} further improve exact-gradient
sampling. The Proximal Bouncy Particle Sampler gives expected
$\widetilde O(\sqrt\kappa\,d^{1/4})$ queries from a supplied bounded
order-2 R\'enyi warm start. The companion Smoothed Picard HMC theorem,
including implemented proximal computations, gives
$\widetilde O(\kappa^{7/6}d^{1/6}+\sqrt\kappa\,d^{1/4})$ from a
reference point with gradient norm at most $\sqrt{\mu d}$.
These statements concern exact gradients and $C^2$ curvature bounds;
they do not implement an arbitrary conditional variance-only oracle.

The result established here is the joint fixed-dimensional rate over
the full range $A\ge0$, together with its fully charged gradient-only
construction and exact zero-noise variant. The constant may depend on
$d$, but not on $\kappa,A,\varepsilon$. Classical ellipsoid geometry,
Poisson and gradient-path rejection, and rare-message lower bounds
remain their respective precedents. The proofs below supply the energy
amortization, marking-error analysis, and stopping arguments needed to
combine these ingredients in the specified oracle model.

\section{Subroutine and implementation guide}\label{app:algorithms}
Algorithms~\ref{alg:full-sampler} and~\ref{alg:one-trial} state the
sampling rules with symbolic parameters in Section~\ref{sec:algorithm-overview}.
The proof of Theorem~\ref{thm:upper} in Appendix~\ref{up:operational-section}
gives their public calibration in
\eqref{alg:dimension-parameters}--\eqref{alg:public-parameters} and the
finite-category weights in \eqref{alg:category-weights}.
The named routines abbreviate the finite constructions in
Appendix~\ref{sec:upper}; they are not additional oracles.
This guide locates their definitions and guarantees.
The working-ellipsoid symbols $c_E,Q_E$ in Section~\ref{sec:algorithm-overview}
are $c,Q$ in the two pilot proofs; its temporary batch size $b$ is denoted
$n$ there, not the fixed correction batch size. The midpoint probes $y_i^\pm$ are written $z_{i,\pm}$ in the rounding
proof. These local names separate geometric matrices from the scalar
correction normalizer $Q$.

The center routine is the procedure in Proposition~\ref{fx:center},
including its update cap, physical-call cap, and fallback. The rounding
routine is Proposition~\ref{fx:geometry}, with its own caps and
nonsingular fallback geometry. Its terminal score threshold is $4$;
contracting the displayed midpoint ellipsoid toward $m$ by $1/5$ yields
exactly $z=m+(c_E-m)/10$ and $\Shape=\gamma Q_E/(10\sqrt d)$ in
\eqref{fx:geom-returned-inner}. The returned parameters determine
$U,q$ through \eqref{fx:proposal-potential}--\eqref{fx:proposal-density};
Proposition~\ref{fx:proposal} gives the finite radial sampling rule.
Neither pilot queries a potential value or a sublevel-set membership
oracle.

For each proposed point, the two directional grids are defined by
\eqref{up:grid-definition}--\eqref{up:estimated-envelope}.
Lemma~\ref{up:proposal-ray-derivative} supplies the derivative convention
used in the mark probability, while the finite rectangle mixture after
\eqref{up:integration-density} samples a correction position. Every
repeated location in a new batch still incurs a new physical call.
The overflow branch and all marker calls are explicit in the main-body
specification; \eqref{alg:category-weights} gives the category weights
used in its implementation.

Section~\ref{up:operational-section} proves measurability, physical-transcript
stopping, pathwise finite execution of the capped procedure, and its
unconditional expected cost. Its proof includes failed pilots, failed
envelopes, ignored marker replies, and the zero-radius branch. The
zero-noise exact modification is specified separately in the proof of
Corollary~\ref{up:noiseless-exact}; it removes both rejection caps and
asserts almost-sure, rather than every-path, termination.

\section{Proof of the Upper Bound}\label{sec:upper}\label{sec:proofs}
All algorithms below operate in normalized coordinates. Their outputs
are divided by $\sqrt\mu$ before being returned.

For the global sampler, take $B=3/64$ and $b_{\rm acc}=10$.
Local pilot statements that allow a variable $B$ retain their stated
ranges. Thus the numerical offset $10$ in the acceptance calculations
below is the specialization of $b_{\rm acc}$ in Algorithm~\ref{alg:one-trial}.
We collect the full public parameter recipe when completing the proof of
Theorem~\ref{thm:upper}; the preceding estimates justify that choice.

\subsection{Conditional averages and a constant-gap center}
\label{fx:center-section}
The statistical input to both geometric pilots is a conditional
mean-square estimate. We establish it first because every later batch
is entered at a random, history-dependent time.
\begin{lemma}[Vector averages under adaptive kernels]\label{up:mean-lemma}
At a point fixed by an incoming history, the mean $\overline G$ of $n$ new
vector replies satisfies
\[
 \E[\overline G\mid\mathcal H]=\nabla F(z),\qquad
 \E[\norm{\overline G-\nabla F(z)}^2\mid\mathcal H]\le A/n.
\]
Consequently $\P(\norm{\overline G-\nabla F(z)}>t\mid\mathcal H)
\le A/(nt^2)$. Projection onto any unit vector fixed before the
calls has the same scalar variance bound.
\end{lemma}
\begin{proof}
The centered replies form martingale differences. If $i<j$, the
$i$th difference is measurable before reply $j$, whose conditional
mean is zero. Their conditional inner-product expectation is therefore
zero, by the tower property. Expanding the squared norm of the sum
and using the individual variance bounds proves the assertion.
Chebyshev and norm contraction under projection give the consequences.
At an adaptively reached batch, condition on its incoming history;
the event that it is reached is already determined.
\end{proof}

This estimate will be used to find a single point with a constant
potential gap, rather than to optimize to the final sampling accuracy.

We work in the normalized model: $F$ is $1$-strongly convex and
$\kappa$-smooth, its minimizer $x_\star$ satisfies $\norm{x_\star}\le1$,
and the total conditional squared noise of a vector reply is at most
$A\ge0$. All conditional statements below refer to the complete incoming
history. We write $\mathbb{B}_d=\{x\in\R^d:\norm{x}\le1\}$ and
\[
 K_v=\{x\in\R^d:F(x)-F(x_\star)\le v\},\qquad v>0.
\]
These level sets are used in the analysis and are not supplied to the
algorithm.

\begin{proposition}[Center construction]\label{fx:center}
Fix $d\ge1$, $0<B\le1/16$, and $0<\delta<1$. There is an algorithm
which returns a point $m$ such that
\begin{equation}\label{fx:center-guarantee}
 \P\{F(m)-F(x_\star)>B\}\le\delta.
\end{equation}
Every actual gradient query and every possible output belongs to
$2\mathbb{B}_d$. The algorithm terminates on every real-valued reply
history, and its number of physical calls is a stopping time for the
physical transcript. On every such history its call count obeys
\begin{equation}\label{fx:center-cost}
 T_{\rm center}\le C_{d,B}
       \left(\log(1+\kappa)+\frac A\delta\right),
\end{equation}
where $C_{d,B}<\infty$ is independent of $\kappa,A,\delta$ and of the
admissible conditional oracle. In particular, fixing $B=3/64$ gives a
constant depending only on $d$.
\end{proposition}

The proof gives explicit public budgets. A symmetric collection of
vertices certifies the center's potential gap, and also determines the
accuracy needed for the next cut. The inverse energy scales of the
estimated iterations have a bounded sum. This allows both the sample
costs and the conditional failure probabilities to be summed without a
factor of $\log\kappa$ multiplying $A/\delta$.

We begin with the deterministic operation that every subsequent
geometric update will use. Its explicit formula fixes both the amount
of retained slack and the contraction charged in the volume argument.

Set
\begin{equation}\label{fx:shallow-parameters}
 \gamma=\frac1{8d},\qquad \alpha=2\gamma=\frac1{4d},\qquad
 a_d=\frac9{32(d+1)}.
\end{equation}
An ellipsoid is represented as $E=c+Q\mathbb{B}_d$, with $Q$ invertible;
$Q$ need not be symmetric.

\begin{lemma}[Fixed shallow-cut update]\label{fx:shallow}
For any nonzero $v\in\R^d$, one can obtain an ellipsoid $E'$ containing
\begin{equation}\label{fx:shallow-retained}
 E\cap\{x:\langle v,x-c\rangle\le
                    \alpha\norm{Q^Tv}\}
\end{equation}
by finitely many arithmetic operations and square roots. The update
preserves invertibility of the shape matrix and satisfies
\begin{equation}\label{fx:shallow-volume}
 \operatorname{vol}(E')\le e^{-a_d}\operatorname{vol}(E)
                 \le e^{-1/(8d)}\operatorname{vol}(E).
\end{equation}
\end{lemma}

\begin{proof}
Put $u=Q^Tv/\norm{Q^Tv}$. For $d\ge2$, define
\begin{equation}\label{fx:shallow-formula}
 \begin{split}
  \beta&=\frac{1-d\alpha}{d+1},\qquad
  a=\frac{d(1+\alpha)}{d+1},\qquad
  h=\frac{d\sqrt{1-\alpha^2}}{\sqrt{d^2-1}},\\
  c'&=c-\beta Qu,\qquad
  Q'=Q\bigl(a uu^T+h(I-uu^T)\bigr).
 \end{split}
\end{equation}
Both eigenvalues $a,h$ in the last factor are positive, so $Q'$ is
invertible. In unit coordinates, write $s=\langle u,z\rangle$ for a
point $z\in\mathbb{B}_d$ with $s\le\alpha$. Its squared norm in the new
ellipsoid is at most
\[
 \frac{(s+\beta)^2}{a^2}+\frac{1-s^2}{h^2},
        \qquad -1\le s\le\alpha.
\]
Since $\alpha<1/d$, one has $a<h$, so this expression is a convex
quadratic in $s$. It equals one at $s=-1$ and at $s=\alpha$ and hence is
at most one throughout the interval. This proves inclusion of
\eqref{fx:shallow-retained}.

The volume ratio is $a h^{d-1}$. At $\alpha=1/(4d)$,
\[
 a=1-\frac3{4(d+1)},\qquad
 h^2=1+\frac{15}{16(d^2-1)}.
\]
Using $\log(1-x)\le-x$ and $\log(1+x)\le x$ gives
\[
 \log a\le-\frac3{4(d+1)},\qquad
 (d-1)\log h\le\frac{15}{32(d+1)}.
\]
Their sum is at most $-a_d$.

For $d=1$, use the interval update
\begin{equation}\label{fx:shallow-interval}
 \beta=\frac{1-\alpha}{2},\qquad a=\frac{1+\alpha}{2},\qquad
 c'=c-\beta Qu,\qquad Q'=aQ.
\end{equation}
In unit coordinates this is precisely the interval $[-1,\alpha]$.
Here $\alpha=1/4$ and $a=5/8$, and
$\log a\le-3/8\le-a_1$. Finally,
$9/[32(d+1)]\ge1/(8d)$ for every integer $d\ge1$, which proves the
second bound in \eqref{fx:shallow-volume}.
\end{proof}

The shallow-cut update is deterministic. To use it with noisy normals,
we next identify the quantities that a bounded set of test vertices can
certify and then estimate those quantities with fresh batches.

Consider an ellipsoid $E=c+Q\mathbb{B}_d$ whose symmetric vertices
\begin{equation}\label{fx:center-vertices}
 z_i^+=c+\gamma Qe_i,\qquad z_i^-=c-\gamma Qe_i,
                         \qquad 1\le i\le d,
\end{equation}
all belong to $2\mathbb{B}_d$. There are $D=2d$ vertices; here $D$
counts points, not a diameter. Let $w_i^\pm=z_i^\pm-c$. The midpoint
and difference of each opposite pair give the bounds below, for which
we set $W=2$ and $R_Q=2\sqrt d/\gamma$:
\begin{equation}\label{fx:center-bounded-geometry}
 \norm c\le2,\qquad \norm{w_i^\pm}\le W=2,\qquad
 \norm Q_{\rm op}\le R_Q.
\end{equation}
Symmetry will turn a bound on the vertex scores into a center-gap
certificate with factor $C_0=\sqrt d/\gamma$. Accordingly, introduce
the score floor
\begin{equation}\label{fx:center-constants}
 m_0=\frac{B}{4C_0}.
\end{equation}
This floor is fixed by the requested center accuracy, not by the
unknown potential values. Define the unobserved scores and energy scale
\begin{equation}\label{fx:center-energy}
 s_i^\pm=\langle\nabla F(z_i^\pm),w_i^\pm\rangle,
 \qquad
 M=\max\{m_0,s_i^\pm:1\le i\le d\},
 \qquad \widetilde M=\frac M{m_0}\ge1.
\end{equation}
The algorithm estimates this scale without being given $M$.

For this estimation subroutine, fix a confidence coefficient
$0<\zeta<1$; its value for the complete pilot will be chosen after the
reciprocal-energy bound is proved. To keep noisy cut normals within the
geometric tolerance, set
\begin{equation}\label{fx:center-estimation-scale}
 t_0=\frac1{100R_Q},\qquad V=\frac{A}{m_0^2}.
\end{equation}
The second quantity is the variance ceiling after dividing replies by
$m_0$. The stopping threshold is
\begin{equation}\label{fx:center-batch-threshold}
 \Lambda=\max\left\{1,\frac{8DV(W^2+t_0^{-2})}{\zeta}\right\}.
\end{equation}
The unit floor keeps the batch schedule nonempty even at $A=0$;
only the other term is needed for the probability estimates.

At fresh stages $n=1,2,4,\ldots$, take $n$ new replies at each of the
$D$ vertices. Divide each resulting mean by $m_0$ and call the
normalized vector $\widehat g_i^\pm$. Form
\begin{equation}\label{fx:center-batch-rule}
 \widehat M=\max\{1,\langle\widehat g_i^\pm,w_i^\pm\rangle:
                                     1\le i\le d\}.
\end{equation}
End the subroutine at the first stage with $n\widehat M\ge\Lambda$.
Return the unnormalized means $G_i^\pm=m_0\widehat g_i^\pm$ and scores
$\widehat s_i^\pm=\langle G_i^\pm,w_i^\pm\rangle$. Earlier stages are
discarded. Since $\widehat M\ge1$, the first dyadic $n\ge\Lambda$
always ends the subroutine on every reply history.

\begin{lemma}[Adaptive batches for the center]\label{fx:center-batch}
Conditional on the complete incoming history fixing the bounded
vertices \eqref{fx:center-vertices}, the subroutine has an event
$\mathcal E_{\rm est}$ satisfying
\begin{equation}\label{fx:center-batch-failure}
 \P(\mathcal E_{\rm est}^{c}\mid\mathcal H)
                      \le\frac{\zeta}{\widetilde M}.
\end{equation}
On $\mathcal E_{\rm est}$ all returned vectors obey
\begin{equation}\label{fx:center-batch-accuracy}
 \norm{G_i^\pm-\nabla F(z_i^\pm)}\le t_0 M,
\end{equation}
and the total number of calls in all its fresh stages is at most
\begin{equation}\label{fx:center-batch-cost}
 2D+\frac{8D\Lambda}{\widetilde M}.
\end{equation}
The event can be chosen measurable at the end of the subroutine.
\end{lemma}

\begin{proof}
By Lemma~\ref{up:mean-lemma}, the normalized mean at any fixed
vertex and reached stage has conditional squared error at most $V/n$.
This remains true under all intervening adaptation of the oracle.

First consider a reached stage with $n<\Lambda/(2\widetilde M)$.
Stopping there requires $\widehat M\ge\Lambda/n>2\widetilde M$.
The floor value one in \eqref{fx:center-batch-rule} cannot cause this.
Some estimated scalar score must therefore have error at least
$\Lambda/(2n)$. By \eqref{fx:center-bounded-geometry}, the corresponding
vector error is at least $\Lambda/(2nW)$. Conditional Markov inequality
and a union over the vertices give the upper bound
\[
 \frac{4DVW^2n}{\Lambda^2}
\]
for stopping at this early stage. The sum of all early dyadic sizes is
at most $\Lambda/\widetilde M$, so the probability of any early stop is
at most
\begin{equation}\label{fx:center-early-error}
 \frac{4DVW^2}{\Lambda\widetilde M}.
\end{equation}

At each reached late stage $n\ge\Lambda/(2\widetilde M)$, the
probability that any normalized vector error exceeds
$t_0\widetilde M$ is at most
\[
 \frac{DV}{n t_0^2\widetilde M^2}.
\]
The sum of reciprocal late dyadic sizes is at most
$4\widetilde M/\Lambda$, including when the first such size is one.
The probability of any inaccurate reached late stage is thus at most
\begin{equation}\label{fx:center-late-error}
 \frac{4DV}{t_0^2\Lambda\widetilde M}.
\end{equation}
Define $\mathcal E_{\rm est}$ by excluding the early-stop event and
these inaccurate reached late stages. Only stages actually executed
enter this definition, so the event is measurable when the subroutine
ends. Equations \eqref{fx:center-early-error}--\eqref{fx:center-late-error}
and \eqref{fx:center-batch-threshold} give
\eqref{fx:center-batch-failure}. Each union is over fixed dyadic indices;
reaching an index is decided before its new replies. These arguments
therefore require no independence across stages.

On $\mathcal E_{\rm est}$ the returned stage is late and accurate.
At any accurate stage, putting $\theta=Wt_0<1/100$, one has
\[
 \widehat M\ge(1-\theta)\widetilde M\ge\widetilde M/2.
\]
If $\widetilde M=1$, the first inequality also follows from the floor
$\widehat M\ge1$. Consequently the first dyadic
$n\ge2\Lambda/\widetilde M$, if reached, triggers stopping. The final
size obeys $n_{\rm stop}\le\max\{1,4\Lambda/\widetilde M\}$. The sum
of the fresh stage sizes is at most $2n_{\rm stop}$, proving
\eqref{fx:center-batch-cost}. Rescaling the normalized vector errors
proves \eqref{fx:center-batch-accuracy}.
\end{proof}

We now specify the local iteration rules. They either make a free
geometric cut, estimate the vertex scores, or return a certified center.
The global caps are chosen after their preservation and energy bounds
have been established.

\begin{enumerate}
\item Form the vertices \eqref{fx:center-vertices}. If some vertex $z$
lies outside $2\mathbb{B}_d$, choose the first such vertex in a fixed
ordering and apply Lemma~\ref{fx:shallow} with normal $v=z$. Count one
update and continue. This step makes no oracle calls.

\item If all vertices lie in $2\mathbb{B}_d$, execute the adaptive
subroutine \eqref{fx:center-batch-rule}.

\item If the returned scores satisfy
\begin{equation}\label{fx:center-stop}
 \max_{i,\pm}\widehat s_i^\pm\le2m_0,
\end{equation}
return the current $c$ and stop.

\item Otherwise choose the first maximizing vertex $z$ and its returned
mean $G$. Thus $\widehat s=\langle G,z-c\rangle>2m_0$ and $G\ne0$.
Apply Lemma~\ref{fx:shallow} with normal $v=G$, count one update, and
continue.
\end{enumerate}

All choices use a fixed index order to break ties. The next lemma
explains why an accurate estimate makes either the return or the cut
valid. The set to be preserved is $K_b$, where we now set $b=m_0$;
thus the protected level matches the score floor in the stopping rule.

\begin{lemma}[Accurate center iterations]\label{fx:center-validity}
As long as all preceding estimation events $\mathcal E_{\rm est}$
hold, the current ellipsoid contains $K_b$. At an estimated iteration
with $K_b\subset E$, its true scale satisfies
\begin{equation}\label{fx:center-gap-by-energy}
 F(c)-F(x_\star)\le C_0M.
\end{equation}
If its estimation event also holds, a return under
\eqref{fx:center-stop} has gap at most $3B/4$, and any update made
instead preserves $K_b$.
\end{lemma}

\begin{proof}
Strong convexity and smoothness at the minimizer give
\begin{equation}\label{fx:center-level-balls}
 \mathbb{B}\left(x_\star,\sqrt{\frac{2b}{\kappa}}\right)
      \subset K_b
      \subset \mathbb{B}(x_\star,\sqrt{2b})\subset2\mathbb{B}_d.
\end{equation}
Thus the initial ellipsoid contains $K_b$. The final containment in~\eqref{fx:center-level-balls} uses the normalized localization
promise $\|x_\star\|\leq1$; it is not implied by the curvature bounds
alone. The extensions in Section~\ref{sec:endpoints}
restore this promise by translation and a conservative public
strong-convexity bound before invoking the existing sampler. They do
not delete the promise from this pilot or enlarge its initial
ellipsoid without changing its analysis.

For a free update, $z$ is outside $2\mathbb{B}_d$. Every
$y\in2\mathbb{B}_d$ satisfies
\[
 \langle z,y-z\rangle\le2\norm z-\norm z^2<0.
\]
Since $z-c=\pm\gamma Qe_i$ for some $i$,
\[
 \langle z,z-c\rangle\le\gamma\norm{Q^Tz}
                      \le\alpha\norm{Q^Tz}.
\]
The halfspace used by the update therefore contains $2\mathbb{B}_d$,
and the new ellipsoid contains the previously retained $K_b$.

For an estimated iteration, monotonicity of the gradient on each of
the coordinate lines through $c$ gives
\[
 -s_i^-\le\gamma(Q^T\nabla F(c))_i\le s_i^+.
\]
Every coordinate in the middle is therefore bounded in absolute value
by $M$, and
\[
 \norm{Q^T\nabla F(c)}\le\frac{\sqrt d}{\gamma}M=C_0M.
\]
The inclusion $x_\star\in E$ supplies a vector $u_\star$ of norm at
most one with $x_\star=c+Qu_\star$. Convexity now proves
\eqref{fx:center-gap-by-energy}. This is a certificate at the single
point $c$; no function values or identification of an unknown best
iterate are needed.

Suppose the current estimation event holds. With $\theta=Wt_0<1/100$,
\begin{equation}\label{fx:center-score-error}
 |\widehat s_i^\pm-s_i^\pm|\le\theta M.
\end{equation}
If \eqref{fx:center-stop} holds, then $M\le2m_0+\theta M$, whence
$M\le3m_0$. Equation \eqref{fx:center-gap-by-energy} bounds the
returned gap by $3C_0m_0=3B/4$.

Otherwise, $\widehat s>2m_0$. The possibility
$\max_{i,\pm}s_i^\pm\le m_0$ would contradict
\eqref{fx:center-score-error}; hence $M=\max_{i,\pm}s_i^\pm$.
Since the chosen estimated score is maximal,
\[
 M\le\frac{\widehat s}{1-\theta}.
\]
For $y\in K_b$, convexity at its selected vertex $z$ yields
\[
 \langle\nabla F(z),y-z\rangle
       \le F(y)-F(z)\le b.
\]
The points $y\in E$ and $z=c\pm\gamma Qe_i$ satisfy
$\norm{y-z}\le(1+\gamma)\norm Q_{\rm op}\le2R_Q$.
Using \eqref{fx:center-batch-accuracy} and
$2R_Qt_0=1/50$ gives
\[
 \langle G,y-z\rangle\le b+\frac M{50}
                            \le b+\frac{\widehat s}{25}.
\]
Also $\norm{Q^TG}\ge\widehat s/\gamma$ because
$\widehat s=\langle G,z-c\rangle$. Re-centering the preceding
inequality at $c$ gives
\[
 \frac{\langle G,y-c\rangle}{\norm{Q^TG}}
 \le\gamma\left(1+\frac1{25}+\frac b{\widehat s}\right)
 <2\gamma=\alpha,
\]
where $b=m_0$ and $\widehat s>2m_0$ were used. The actual shallow
halfspace therefore contains $K_b$. Lemma~\ref{fx:shallow} proves its
preservation after this update, completing the induction.
\end{proof}

Preservation alone bounds the number of cuts, but does not yet control
their noise cost. To obtain an additive bound, we relate each current
energy to the remaining ellipsoid volume. The vertex gap adds one copy
of $M$ to the center-gap bound, so put
\begin{equation}\label{fx:center-volume-constants}
 C_1=C_0+1,\qquad H_d=(C_0C_1)^d.
\end{equation}
The resulting volume comparison also controls the reciprocal energies,
even when the algorithm revisits an energy range.

\begin{lemma}[Energy and ellipsoid volume]\label{fx:center-energy-volume}
At any estimated iteration whose preceding cuts preserve $K_b$, the
current ellipsoid and its true scale satisfy
\begin{equation}\label{fx:center-energy-volume-bound}
 \operatorname{vol}(E)
       \le H_d\widetilde M^d\operatorname{vol}(K_b).
\end{equation}
Along any prefix of the iteration rules above in which all preceding estimation
events hold, including one final estimated iteration before its own
accuracy is decided,
\begin{equation}\label{fx:center-reciprocal-sum}
 \sum_{\textnormal{estimated iterations in the prefix}}
                  \frac1{\widetilde M}\le C_{\rm sum},\qquad
 C_{\rm sum}:=2+\frac{2\log H_d+4d\log2}{a_d}.
\end{equation}
\end{lemma}

\begin{proof}
Convexity at a vertex gives
\[
 F(z_i^\pm)-F(c)
        \le\langle\nabla F(z_i^\pm),z_i^\pm-c\rangle\le M.
\]
Together with \eqref{fx:center-gap-by-energy}, this places every vertex
in $K_{C_1M}$. Their convex hull contains
$c+(\gamma/\sqrt d)Q\mathbb{B}_d$, since the Euclidean unit ball scaled
by $1/\sqrt d$ is contained in the unit $\ell_1$ ball. It follows that
\[
 \operatorname{vol}(E)
       \le\left(\frac{\sqrt d}{\gamma}\right)^d
                         \operatorname{vol}(K_{C_1M}).
\]
For $v\ge b$, convexity about the minimizer implies
\[
 K_v\subset x_\star+\frac vb(K_b-x_\star).
\]
Indeed, for $x\in K_v$ the point
$x_\star+(b/v)(x-x_\star)$ has gap at most $b$.
Since $M=m_0\widetilde M=b\widetilde M$, the preceding volume bounds
prove \eqref{fx:center-energy-volume-bound} with
$H_d=(C_0C_1)^d$.

Fix an integer $k\ge0$. At the first estimated iteration in the band
$2^k\le\widetilde M<2^{k+1}$, the current ellipsoid has volume at most
\[
 H_d2^{d(k+1)}\operatorname{vol}(K_b).
\]
Every intervening update, free or estimated, reduces volume by at
least $e^{-a_d}$. All ellipsoids before the last iteration of the
prefix still contain $K_b$. Therefore the number of estimated
iterations in this band is at most
\[
 1+\frac{\log H_d+d(k+1)\log2}{a_d}.
\]
This argument does not assume that the energy scales are monotone.
Summing the contribution of each band gives
\[
 \begin{split}
 \sum\frac1{\widetilde M}
 &\le\sum_{k=0}^{\infty}2^{-k}
        \left(1+\frac{\log H_d+d(k+1)\log2}{a_d}\right)\\
 &=2+\frac{2\log H_d+4d\log2}{a_d}=C_{\rm sum}.
 \end{split}
\]
This proves \eqref{fx:center-reciprocal-sum} also for a prefix ending
just before an estimation outcome is declared inaccurate.
\end{proof}

We can now assemble the finite pilot. The reciprocal-energy sum fixes
the confidence allocation, rather than splitting confidence equally
among all possible iterations:
\begin{equation}\label{fx:center-confidence}
 \zeta=\frac{\delta}{2C_{\rm sum}}.
\end{equation}
Use this value in the batch threshold~\eqref{fx:center-batch-threshold}.
The lower volume of $K_b$ and the shallow-cut contraction suggest the
update cap
\begin{equation}\label{fx:center-update-cap}
 N=\left\lceil\frac{d\log(2\kappa/b)}{2a_d}\right\rceil+1.
\end{equation}
Summing the per-round batch bound over at most $N$ estimated iterations
and the reciprocal-energy budget then gives the physical-call cap
\begin{equation}\label{fx:center-budgets}
 Q_{\max}=\left\lceil2DN+8D\Lambda C_{\rm sum}\right\rceil+1.
\end{equation}
All these quantities are public and finite. Initialize $E_0=2\mathbb{B}_d$
with $c=0$ and $Q=2I$, and apply the local rules above, allowing at most
$N$ updates in total, free or estimated. Before every fresh batch stage,
check its entire cost $Dn$; if it would exceed $Q_{\max}$ cumulatively,
return $0$. Also return $0$ immediately after the $N$th update unless
a center has already been returned. There is no further iteration after
that update. The following proof shows that neither cap interrupts an
all-good execution, while both caps remain effective on failed histories.

\begin{proof}[Proof of Proposition~\ref{fx:center}]
First analyze the same procedure with the $N$-update cap but without
the global physical-call cap. This auxiliary procedure is also finite
on every history: each estimation subroutine stops by its first
dyadic $n\ge\Lambda$, and at most $N$ updates are performed.

Condition on the incoming history of each reached estimation round.
Lemma~\ref{fx:center-batch} bounds its failure probability by
$\zeta/\widetilde M$. Let $\mathcal I$ index reached estimation rounds
with no earlier failed estimation event. All preceding cuts preserve
$K_b$, so \eqref{fx:center-reciprocal-sum} gives
\begin{equation}\label{fx:center-total-failure}
 \begin{split}
 \P\{\text{some estimation event fails}\}
 &\le\zeta\,\E\!\left[
      \sum_{i\in\mathcal I}\frac1{\widetilde M_i}\right]\\
 &\le\zeta C_{\rm sum}=\frac\delta2.
 \end{split}
\end{equation}
The summands include at most one failed iteration. The reciprocal
sum is bounded pathwise on just this prefix, so no bound conditional
on future success has been used.

On the event that every estimation event holds, the $N$-update
fallback is impossible. All updates retain $K_b$, which contains the
ball in \eqref{fx:center-level-balls}. But after $N$ updates,
\[
 \operatorname{vol}(E_N)
 \le e^{-a_dN}\operatorname{vol}(2\mathbb{B}_d)
 <\operatorname{vol}\!\left(
        \mathbb{B}\left(x_\star,\sqrt{2b/\kappa}\right)\right)
\]
by the definition of $N$ in \eqref{fx:center-update-cap}. Thus the
auxiliary procedure must have returned a center $c$ under
\eqref{fx:center-stop}, and its gap is at most $3B/4$ by
Lemma~\ref{fx:center-validity}.

On the same event, there are at most $N$ estimated iterations. Their
individual pathwise costs \eqref{fx:center-batch-cost} and the sum
\eqref{fx:center-reciprocal-sum} bound the entire physical call count
by
\begin{equation}\label{fx:center-good-path-cost}
 2DN+8D\Lambda C_{\rm sum}<Q_{\max}.
\end{equation}
Consequently the global physical-call cap cannot alter any such run.
Formally, use the same conditional oracle kernels to couple the
capped and auxiliary procedures up to their first possible
separation. On the successful auxiliary event, no stage would exceed
the global budget, so separation does not occur. Combining this fact
with \eqref{fx:center-total-failure} proves
\eqref{fx:center-guarantee}, with the stronger upper bound $\delta/2$.

On every history, including those with estimation failures, the
capped algorithm uses at most $Q_{\max}$ physical calls. Its formula
\eqref{fx:center-budgets} gives
\[
 Q_{\max}=O\left(\log(1+\kappa)+\frac A\delta\right),
\]
where the hidden constant may also depend on $B$: the quantities
$D,W,R_Q,C_0,m_0,t_0,H_d,C_{\rm sum}$ depend only on $d,B$, while \eqref{fx:center-confidence} gives
$\zeta=\delta/(2C_{\rm sum})$ and
\eqref{fx:center-estimation-scale} gives $V=A/m_0^2$.
The unit floor in $\Lambda$ adds only a constant depending on $d,B$,
which is absorbed by $\log(1+\kappa)$ because $\kappa\ge1$. This proves
the deterministic bound \eqref{fx:center-cost}, and therefore the
same bound for the expected cost under every admissible oracle.
No cost incurred after an erroneous estimate is discarded.

For completeness, the execution is well defined on arbitrary
real-valued reply histories. The shape matrix stays invertible by
Lemma~\ref{fx:shallow}. A free-cut normal is nonzero because it is an
out-of-ball vertex. An estimated-cut normal is nonzero because its
score exceeds $2m_0$. All computed quantities at a finite stage are
finite, and the stage sizes and the number of updates have finite
public caps. Every actual query is made only after all vertices have
passed the test for membership in $2\mathbb{B}_d$. A center returned
under \eqref{fx:center-stop} is the midpoint of an opposite in-ball
pair and hence is also in $2\mathbb{B}_d$; a fallback returns $0$.

Finally, the pilot generates no private random variables. Starting
from its known initial state, its free updates, query locations,
stage lengths, continuation decisions, and possible budget
termination can all be reconstructed from the physical replies
already received. All operations between replies are finite.
Thus the event that no further call will be made after a given
physical reply is measurable from the physical transcript at that
reply. The physical call count is a stopping time, and the completed
pilot state can be parsed from that transcript for use by subsequent
phases. This finishes the proof.
\end{proof}

\subsection{Constructing a proposal in fixed dimension}
\label{fx:geometry-section}

We next construct the global proposal from a center of bounded potential
error. Throughout this subsection, $d\ge1$, the normalized potential $F$
is $1$-strongly convex and $\kappa$-smooth, and the total conditional
gradient-noise variance is at most $A\ge0$. Set
\begin{equation}\label{fx:geom-parameters}
 B=\frac3{64},\qquad
 \alpha=\frac1{4d},\qquad \gamma=\frac1{8d},\qquad
 D=2d,\qquad \eta_d=\frac{\gamma}{2\sqrt d},\qquad
 C_{\rm geom}=\eta_d^{-d}.
\end{equation}
The symbol $D$ counts test points and is unrelated to the radial
normalizing constant introduced below. Define further
\begin{align}
 L_{{\rm geom},d}
 &=4+16d\log C_{\rm geom}+32d^2\log2,
 \label{fx:geom-amort-constant}\\
 H_\kappa
 &=2+\left\lceil8d^2\log(8\sqrt\kappa)\right\rceil.
 \label{fx:geom-round-cap}
\end{align}
All constants with a subscript $d$ depend only on the fixed dimension.

\begin{proposition}[Proposal construction from a constant-gap center]
\label{fx:geometry}
Suppose the public inputs and incoming physical transcript supply a
point $m$ satisfying
$F(m)-\min F\le B$, and let $0<\delta<1$. Put
\begin{equation}\label{fx:geom-budget}
 \zeta=\frac{\delta}{L_{{\rm geom},d}},\qquad
 \lambda=\max\left\{1,\frac{2^{16}DA}{\zeta}\right\},\qquad
 B_{\rm geom}
 =\left\lceil2DH_\kappa+8D\lambda L_{{\rm geom},d}\right\rceil.
\end{equation}
There is a procedure using at most $B_{\rm geom}$ physical calls on
every execution which returns a point $z\in\R^d$ and an invertible
matrix $\Shape$. With conditional probability at least $1-\delta$, these
parameters satisfy
\begin{equation}\label{fx:geom-sandwich}
 z+\Shape \mathbb{B}_d\ \subset\ \{x:F(x)-F(m)\le1\}
 \ \subset\ z+\rho \Shape \mathbb{B}_d,
 \qquad \rho=\frac{20\sqrt d}{\gamma}.
\end{equation}
On every history, including failures and histories for which the
supplied center does not satisfy its promise, the procedure is finite
and its output obeys
\begin{equation}\label{fx:geom-all-path-shape}
 \norm{z-m}\le1,\qquad \norm{\Shape}_{\rm op}\le1.
\end{equation}
Its rules use only public parameters and physical queries and replies.
In particular, its stopping decision and returned parameters are
measurable in the physical transcript. Its call bound is
$O(1+\log\kappa+A/\delta)$.
\end{proposition}

We prove this proposition by first establishing the geometric and
statistical components, then specifying their finite implementation.

We first identify the set that the rounding procedure must preserve.
Its inner and outer radii determine the initial ellipsoid and the
maximum number of valid cuts; neither radius requires a value query.

For this rounding argument, measure levels relative to the supplied
center $m$, rather than to the unknown minimum used in the center pilot:
\begin{equation}\label{fx:geom-levelsets}
 I(x)=F(x)-F(m),\qquad K_t=\{x:I(x)\le t\},\qquad K=K_1.
\end{equation}
Let $x_\star$ be the minimizer of $F$. Strong convexity and
$B\le1/8$ give
\[
 \norm{m-x_\star}\le\sqrt{2B}\le\frac12.
\]
For $x\in K$, the same inequality gives
$\norm{x-x_\star}\le\sqrt{2(1+B)}\le3/2$. Therefore
\begin{equation}\label{fx:geom-outer-ball}
 K\subset \mathbb{B}(m,2).
\end{equation}
The descent inequality
\[
 F\left(m-\frac{\nabla F(m)}\kappa\right)
 \le F(m)-\frac{\norm{\nabla F(m)}^2}{2\kappa}
\]
implies $\norm{\nabla F(m)}\le\sqrt{2\kappa B}\le\sqrt\kappa/2$.
Consequently, if $r_0=1/(4\sqrt\kappa)$ and $\norm{y-m}\le r_0$,
smoothness gives
\[
 I(y)\le\frac{\sqrt\kappa}{2}\norm{y-m}
       +\frac\kappa2\norm{y-m}^2
 \le\frac18+\frac1{32}<1.
\]
We have proved
\begin{equation}\label{fx:geom-inner-ball}
 \mathbb{B}\left(m,\frac1{4\sqrt\kappa}\right)\subset K\subset \mathbb{B}(m,2).
\end{equation}

For $t\ge1$, convexity and $I(m)=0$ imply
\begin{equation}\label{fx:geom-level-growth}
 K_t\subset m+t(K-m),\qquad
 \operatorname{vol}(K_t)\le t^d\operatorname{vol}(K).
\end{equation}
Indeed, if $x\in K_t$, then
$I(m+(x-m)/t)\le I(x)/t\le1$. All these sets are compact, with
positive finite volume, by strong convexity and
\eqref{fx:geom-inner-ball}.

The initial enclosure is now available. We next describe how symmetric
vertices produce a valid cut or certify an inner ellipsoid. Points
outside the known ball can be eliminated without querying the oracle;
only the remaining vertices need statistical tests.

Maintain an ellipsoid $E=c+Q\mathbb{B}_d$, with $Q$ invertible, starting
from $c=m$, $Q=2I_d$. Its $D=2d$ test vertices are
\begin{equation}\label{fx:geom-vertices}
 v_{i,+}=c+\gamma Qe_i,\qquad
 v_{i,-}=c-\gamma Qe_i,\qquad 1\le i\le d.
\end{equation}
Whenever a cut with nonzero normal $g$ at a test vertex $v$ preserves
$K$, its normalized offset is
\[
 \frac{g\cdot(v-c)}{\norm{Q^Tg}}
 =\frac{\pm\gamma\,g\cdot Qe_i}{\norm{Q^Tg}}
 \le\gamma<\alpha.
\]
We weaken it to the shallow cut of offset $\alpha$ and apply
Lemma~\ref{fx:shallow}. The updated ellipsoid contains the retained
part and its volume is at most $e^{-1/(8d)}$ times the previous
volume. The interval update in that lemma covers $d=1$.

First inspect the vertices without making oracle calls. If some
$v$ lies outside $\mathbb{B}(m,2)$, use $g=v-m$ and the cut
$g\cdot(y-v)\le0$. This is a valid cut: for every $y\in \mathbb{B}(m,2)$,
\[
 g\cdot(y-v)
 \le2\norm{g}-\norm{g}^2<0.
\]
It has a nonzero normal, preserves $K$ by
\eqref{fx:geom-outer-ball}, and costs no gradient query. Update the
ellipsoid and begin the next round.

Only if all vertices belong to $\mathbb{B}(m,2)$ do we make a gradient-testing
round. For each vertex define
\begin{equation}\label{fx:geom-testing-points}
 z_{i,\pm}=\frac{m+v_{i,\pm}}2,\qquad
 w_{i,\pm}=z_{i,\pm}-m.
\end{equation}
Then $\norm{w_{i,\pm}}\le1$. Relabel these $D$ pairs as
$(z_i,w_i)_{i=1}^D$, and introduce their true, unknown scores and
energy:
\begin{equation}\label{fx:geom-energy}
 s_i=\nabla F(z_i)\cdot w_i,\qquad
 M=\max\{1,s_1,\ldots,s_D\}.
\end{equation}

\begin{lemma}[A reliable test either stops or cuts safely]
\label{fx:geom-test}
Suppose estimates $\widehat g_i$ satisfy
\begin{equation}\label{fx:geom-good-estimates}
 \norm{\widehat g_i-\nabla F(z_i)}\le\frac{M}{100}
 \qquad(1\le i\le D).
\end{equation}
Set $\widehat s_i=\widehat g_i\cdot w_i$ and
$\widehat M=\max\{1,\widehat s_1,\ldots,\widehat s_D\}$.
If $\max_i\widehat s_i\le4$, then the ellipsoid
\begin{equation}\label{fx:geom-inner-five}
 \frac{m+c}{2}+\eta_d Q\mathbb{B}_d
\end{equation}
is contained in $K_5$. Otherwise, choosing a maximizing index $i$,
the cut with normal $\widehat g_i$ at $v_i=m+2w_i$ preserves $K$.
\end{lemma}

\begin{proof}
The bound on the $w_i$ implies
\begin{equation}\label{fx:geom-estimated-energy}
 |\widehat s_i-s_i|\le M/100,
 \qquad |\widehat M-M|\le M/100.
\end{equation}
If all estimated scores are at most $4$, then
$99M/100\le\widehat M\le4$, so $M\le400/99<5$.
Convexity relative to $m$ gives $I(z_i)\le s_i\le M$.
Thus the convex hull of the points $z_i$ lies in $K_5$. It contains
\eqref{fx:geom-inner-five}, because
\[
 \frac1{\sqrt d}\mathbb{B}_d
 \subset\operatorname{conv}\{e_1,-e_1,\ldots,e_d,-e_d\}
\]
and the test points have center $(m+c)/2$ and generating vectors
$\gamma Qe_i/2$.

For the other case, fix a maximizing index, so
$\widehat s_i=\widehat M>4$. For every $y\in K$, convexity gives
\[
 \nabla F(z_i)\cdot(y-z_i)
 \le I(y)-I(z_i)\le1+B.
\]
Writing $e_i=\widehat g_i-\nabla F(z_i)$ and using
$\norm{y-m}\le2$, $\norm{w_i}\le1$, we obtain
\[
 \widehat g_i\cdot(y-m)
 \le\widehat s_i+1+B+3\norm{e_i}
 \le\widehat s_i+1+B+3M/100.
\]
By \eqref{fx:geom-estimated-energy},
$M\le100\widehat s_i/99$, and hence
\[
 1+B+3M/100
 \le9/8+\widehat s_i/33<\widehat s_i.
\]
The last inequality follows from $\widehat s_i>4$.
Therefore $\widehat g_i\cdot(y-v_i)<0$. The normal is nonzero,
since its inner product with $w_i$ exceeds $4$.
\end{proof}

The preceding lemma establishes correctness of an individual reliable
test. To avoid paying its noise cost once for every possible cut, we
must also count how often low-energy tests can occur. The following
volume argument provides exactly that accounting.

\begin{lemma}[Reciprocal-energy bound]
\label{fx:geom-amort}
Consider the geometric execution up to termination or up to and
including its first gradient-testing round that fails
\eqref{fx:geom-good-estimates}. The same conclusion holds if a
round is declared bad for violating an additional promised batch-cost
bound. Then
\begin{equation}\label{fx:geom-reciprocal-sum}
 \sum_{\text{testing rounds in this prefix}}\frac1M
 \le L_{{\rm geom},d}.
\end{equation}
If every testing round is good, the execution terminates before
$H_\kappa$ outer rounds.
\end{lemma}

\begin{proof}
At the start of every round in the prefix, $E$ contains $K$: all
preceding cuts are valid by Lemma~\ref{fx:geom-test} or by the
free ball-cut calculation. At a testing round, the true inequalities
$I(z_i)\le s_i\le M$ imply
\[
 \frac{m+c}{2}+\eta_d Q\mathbb{B}_d\subset K_M.
\]
This conclusion does not require accuracy of the current estimates.
Taking volumes and using \eqref{fx:geom-level-growth} gives
\begin{equation}\label{fx:geom-volume-energy}
 \operatorname{vol}(E)
 \le C_{\rm geom}\operatorname{vol}(K_M)
 \le C_{\rm geom}M^d\operatorname{vol}(K).
\end{equation}

Fix $k\ge0$. If there is a testing round with
$2^k\le M<2^{k+1}$, consider the first such round in the prefix.
At that time the volume ratio to $K$ is at most
$C_{\rm geom}2^{d(k+1)}$. Each subsequent testing round, except
possibly the last one in the prefix, produces a valid cut if the
algorithm has not stopped. Each such cut decreases volume by a
factor at most $e^{-1/(8d)}$; any intervening free cuts decrease it
as well. The maintained ellipsoid cannot have volume below that of
$K$. Thus the number of testing rounds in this energy bin is at most
\[
 2+8d\log C_{\rm geom}+8d^2(k+1)\log2.
\]
The extra two safely include the initial round and a possible
terminal or first bad round. Summing over bins yields
\begin{align*}
 \sum\frac1M
 &\le\sum_{k\ge0}2^{-k}
   [2+8d\log C_{\rm geom}+8d^2(k+1)\log2]\\
 &=4+16d\log C_{\rm geom}+32d^2\log2
 =L_{{\rm geom},d}.
\end{align*}

Finally, on an all-good execution every cut preserves $K$.
The initial-to-minimum possible volume ratio is at most
$(8\sqrt\kappa)^d$, by \eqref{fx:geom-inner-ball}.
There can therefore be at most $8d^2\log(8\sqrt\kappa)$ cuts.
Every nonterminal outer round performs one cut, so the cap
\eqref{fx:geom-round-cap} cannot be reached on such an execution.
\end{proof}

The reciprocal-energy bound is expressed in terms of the unknown
true scores. We now give the observable batching rule that achieves
the corresponding cost and failure probability without knowing those
scores. Its early-stop and late-error bounds mirror the center pilot.

\begin{lemma}[Adaptive batches at a finite collection of points]
\label{fx:geom-batch}
Condition on the complete incoming history of a testing round, so
the points $z_i$ and vectors $w_i$, $\norm{w_i}\le1$, are fixed.
Let $M$ be given by \eqref{fx:geom-energy}. For $0<\zeta<1$ put
$\lambda=\max\{1,2^{16}DA/\zeta\}$. There is a finite procedure returning
estimates $\widehat g_1,\ldots,\widehat g_D$ such that, with
conditional probability at least $1-\zeta/M$, both
\eqref{fx:geom-good-estimates} and the call bound
\begin{equation}\label{fx:geom-batch-good-cost}
 N_{\rm batch}\le2D+\frac{8D\lambda}{M}
\end{equation}
hold. On all histories it uses fewer than $4D\lambda$ calls.
\end{lemma}

\begin{proof}
For the successive batch sizes
\[
 n=1,2,4,\ldots,2^{\lceil\log_2\lambda\rceil},
\]
take a fresh mean $\widehat g_{i,n}$ of $n$ replies at each of the
$D$ points. All $D$ means are completed before testing this stage.
Set
\[
 \widehat M_n
 =\max\{1,\widehat g_{1,n}\cdot w_1,\ldots,
                  \widehat g_{D,n}\cdot w_D\}.
\]
Stop at the first stage with $n\widehat M_n\ge\lambda$ and
return its means. At the last stage $n\ge\lambda$, stopping is
automatic. Since $\lambda\ge1$, the sum of all possible batch sizes
is less than $4\lambda$, proving the all-history call bound.

By Lemma~\ref{up:mean-lemma}, each vector mean has conditional
squared error at most $A/n$. The query points and their true
gradients remain fixed throughout the testing round, even though
the reply distributions can change with the complete history.

Call a stop early if $n<\lambda/(2M)$. An early stop requires
$\widehat M_n\ge\lambda/n>2M$. Since $M\ge1$ and all true
scores are at most $M$, some mean then has error of norm at least
$\lambda/(2n)$. Conditional Chebyshev and a union bound over the
$D$ means give probability at most $4DAn/\lambda^2$ for this
event at stage $n$. Summing over the early dyadic stages bounds the
probability of any early stop by
\begin{equation}\label{fx:geom-batch-early}
 \sum_{n<\lambda/(2M)}\frac{4DAn}{\lambda^2}
 \le\frac{4DA}{\lambda M}.
\end{equation}

At a stage with $n\ge\lambda/(2M)$, the probability that some
mean has error greater than $M/100$ is at most
$10^4DA/(nM^2)$. The sum of reciprocals of dyadic batch sizes
starting at their first such stage is at most
$4M/\lambda$. This also holds when the first stage is $n=1$,
because then $\lambda\le2M$. Consequently the probability of
any such late-stage error is at most
\begin{equation}\label{fx:geom-batch-late}
 \frac{40000DA}{\lambda M}.
\end{equation}
For adaptive execution these bounds are applied only to stages
actually reached, conditioning at each stage entry. The fixed points
and $M$ do not change, so the same deterministic sums bound the
resulting probabilities. No independence between different stages
or different pointwise batches is needed.

Outside the events in \eqref{fx:geom-batch-early} and
\eqref{fx:geom-batch-late}, the return is at a non-early stage and
all returned means satisfy \eqref{fx:geom-good-estimates}.
Moreover, a stage with $n\ge2\lambda/M$, if reached, has
$\widehat M_n\ge99M/100$ and must stop. The deterministic final
stage may stop sooner. Thus
\[
 n_{\rm stop}\le\max\{1,4\lambda/M\}.
\]
The sum of fresh batch sizes up to this stage is at most
$2n_{\rm stop}$, which proves \eqref{fx:geom-batch-good-cost}.
The total exceptional probability is at most
\[
 \frac{40004DA}{\lambda M}\le\frac\zeta M.
\]
Only $\lambda\ge2^{16}DA/\zeta$ is needed in this estimate. The
additional floor $\lambda\ge1$ makes the same finite schedule valid
at zero noise, where every estimation error is zero almost surely.
\end{proof}

We now combine the test, the reciprocal-energy bound, and the adaptive
batches. The remaining point is to impose finite public budgets without
losing the success probability: an all-good uncapped execution must
fit inside those budgets.

\begin{proof}[Proof of Proposition~\ref{fx:geometry}]
Use the following algorithm, with the constants in
\eqref{fx:geom-budget}. Start at $E=\mathbb{B}(m,2)$ and a zero call count.
Allow at most $H_\kappa$ outer rounds. In each round perform the
vertex inspection and, if needed, the free ball cut described above.
Otherwise apply Lemma~\ref{fx:geom-batch} at the $D$ testing
points. If every estimated score is at most $4$, return the parameters
\begin{equation}\label{fx:geom-returned-inner}
 z=m+\frac{c-m}{10},\qquad
 \Shape=\frac{\gamma}{10\sqrt d}Q.
\end{equation}
If the test does not stop, choose a maximizing index and perform
the gradient cut of Lemma~\ref{fx:geom-test}. Before each physical
call, check the cap $B_{\rm geom}$; if the call would exceed it,
return $z=m$, $\Shape=I_d$. Return this same fallback if the outer-round
cap is reached. Ties in selecting vertices or indices are resolved
by their fixed index order.

For the success analysis, first consider the version with the finite
outer-round and per-round batch caps but without the global call
cap. Declare a batch good if both its accuracy and its call bound
in Lemma~\ref{fx:geom-batch} hold. At any testing round whose
prior batches were good, the conditional probability that the
current batch is bad is at most $\zeta/M$. By
Lemma~\ref{fx:geom-amort}, the sum of $1/M$ along a good prefix,
including a possible first bad round, is at most
$L_{{\rm geom},d}$ on every path. Taking conditional expectations
of the indicators of first failure therefore gives
\begin{equation}\label{fx:geom-first-failure}
 \P(\text{some batch is bad before termination})
 \le\zeta L_{{\rm geom},d}=\delta.
\end{equation}
This is a first-failure union bound under the complete history; it
does not condition future replies on simultaneous success events.

On the complementary event, every cut preserves $K$, and the
algorithm terminates by a successful test before the outer-round
cap. Its physical call count is bounded by
\[
 \sum_{\text{testing rounds}}(2D+8D\lambda/M)
 \le2DH_\kappa+8D\lambda L_{{\rm geom},d}
 \le B_{\rm geom}.
\]
Hence adding the global cap cannot change any all-good execution.
For a formal comparison, couple the capped and uncapped versions
through their common prefix, continuing the oracle kernels on the
counterfactual uncapped history after a cap. Both finite procedures
have identical histories before the cap. Since the uncapped version
cannot reach this cap on its all-good event, the implemented version
has success probability at least $1-\delta$ as well.

At successful termination, Lemma~\ref{fx:geom-test} supplies
\eqref{fx:geom-inner-five}. Contract this ellipsoid toward $m$
by a factor $1/5$. Convexity and $I(m)=0$ imply
$z+\Shape \mathbb{B}_d\subset K_1$, with the parameters in
\eqref{fx:geom-returned-inner}. Also $K_1\subset E$, and
$m\in E$ implies $\norm{Q^{-1}(c-m)}\le1$. For $x\in E$,
\begin{align*}
 \norm{\Shape^{-1}(x-z)}
 &\le\norm{\Shape^{-1}(x-c)}+\norm{\Shape^{-1}(c-z)}\\
 &\le\frac{10\sqrt d}{\gamma}
       +\frac{9\sqrt d}{\gamma}
 <\rho.
\end{align*}
This proves \eqref{fx:geom-sandwich}.

For the all-history conclusions, every accepted testing ellipsoid
has all vertices in $\mathbb{B}(m,2)$, whether or not its estimates are
accurate. Averaging opposite vertices gives $\norm{c-m}\le2$,
and subtracting them gives $\norm{Qe_i}\le2/\gamma$. Therefore
\[
 \norm{Q}_{\rm op}\le\frac{2\sqrt d}{\gamma},\qquad
 \norm{z-m}\le\frac15,\qquad
 \norm{\Shape}_{\rm op}\le\frac15.
\]
The fallback $z=m$, $\Shape=I_d$ obeys
\eqref{fx:geom-all-path-shape}. The matrices remain invertible
by Lemma~\ref{fx:shallow}. Every selected cut has a nonzero
normal even on a bad history: a ball cut has $\norm{v-m}>2$,
and a gradient cut has $\widehat g_i\cdot w_i>4$.
All replies, averages, and geometric operations are finite on the
specified real-valued paths, and both loop and call counts are
capped. All choices and stage boundaries are deterministic functions
of the public inputs and the physical transcript. This completes
the proof.
\end{proof}

Rounding supplies only an ellipsoid sandwich, not a sampler for the
target. We turn it into a full-support proposal whose normalization and
random-generation rule are explicit. The same rule must remain valid
when the rounding procedure returns its fallback.

Let $v_d=\operatorname{vol}(\mathbb{B}_d)$ and define
\begin{equation}\label{fx:radial-normalizer}
 D_d=1+d\sum_{j=0}^{d-1}\binom{d-1}{j}4^{j+1}j!.
\end{equation}
For every output $(z,\Shape)$ of the preceding procedure, including a
fallback, set
\begin{align}
 U(x)&=\frac14
   \left(\frac{\norm{\Shape^{-1}(x-z)}}\rho-1\right)_+,
 \label{fx:proposal-potential}\\
 q(x)&=\frac{e^{-U(x)}}{v_d|\det \Shape|\rho^dD_d}.
 \label{fx:proposal-density}
\end{align}
The success-probability benchmark used subsequently is the public
constant
\begin{equation}\label{up:pzero}
 p_0=\frac{e^{-11}}{\rho^dD_d}>0.
\end{equation}

\begin{proposition}[Domination, acceptance, and all-history moments]
\label{fx:proposal}\label{up:proposal-lemma}
The formula \eqref{fx:proposal-density} is a probability density
on every history and can be sampled with a finite number of internal
random draws and no oracle queries. On the good event of
Proposition~\ref{fx:geometry},
\begin{equation}\label{up:potential-envelope}
 I(x)+B\ge2U(x)\qquad(x\in\R^d).
\end{equation}
In particular, with $h(x)=I(x)-U(x)+10$,
\begin{equation}\label{up:h-geometry}
 h(x)\ge10-B+U(x)>1,
 \qquad 0<1+I(x)\le2(1+h(x)),
\end{equation}
and the ideal acceptance probability satisfies
\begin{equation}\label{fx:proposal-acceptance}
 \int q(x)e^{-h(x)}\,dx\ge p_0.
\end{equation}
On all histories, for $Y_q\sim q$,
\begin{align}
 \E_q[\norm{Y_q-m}^2]
 &\le2+2\rho^2[1+8d+16d(d+1)],\notag\\
 \E_q[\log(1+\norm{Y_q-m})]
 &\le\log\bigl(2+\rho(1+4d)\bigr).
 \label{up:proposal-moments}
\end{align}
\end{proposition}

\begin{proof}
First verify normalization and give a sampling rule. In coordinates
$x=z+\rho \Shape y$, the unnormalized density is
$\exp[-(\norm y-1)_+/4]$. Its integral equals
\begin{align*}
 v_d+d v_d\int_1^\infty r^{d-1}e^{-(r-1)/4}\,dr
 &=v_d+d v_d\sum_{j=0}^{d-1}\binom{d-1}{j}
       \int_0^\infty s^j e^{-s/4}\,ds\\
 &=v_dD_d.
\end{align*}
The last integral is $4^{j+1}j!$, by substitution and repeated
integration by parts. The change of variables supplies the factor
$\rho^d|\det \Shape|$ in \eqref{fx:proposal-density}.

Draw a direction $\Theta$ uniformly on the Euclidean unit sphere.
Independently select a radial component as follows. With probability
$1/D_d$, set $R=U_0^{1/d}$ for
$U_0\sim\operatorname{Unif}(0,1)$. For each
$j\in\{0,\ldots,d-1\}$, with probability
\begin{equation}\label{fx:radial-weights}
 \frac{d\binom{d-1}{j}4^{j+1}j!}{D_d},
\end{equation}
set $R=1+4\sum_{\ell=1}^{j+1}E_\ell$, where the $E_\ell$
are independent unit exponentials. Return $Y_q=z+\rho \Shape R\Theta$.
The interior component has the radial law of the uniform unit ball;
the other components expand $(1+s)^{d-1}e^{-s/4}$ exactly as above.
Their weights sum to one, proving the claimed sampling rule.
For $d=1$, the spherical direction is an independent uniform sign.
For general $d$, normalize a standard Gaussian vector. Generate its
coordinates in independent pairs using
$\sqrt{-2\log U_1}(\cos(2\pi U_2),\sin(2\pi U_2))$, with fresh
open-interval uniforms $U_1,U_2$, discarding one coordinate when needed.
On the probability-zero zero-vector event define a fixed unit direction.
Each exponential is $-\log U$ for another fresh open-interval uniform.
The number of radial exponential draws is at most $d$,
and there is no internal rejection loop or convex-body sampling
oracle.

Now assume the good sandwich \eqref{fx:geom-sandwich}.
Write $E_{\rm out}=z+\rho \Shape \mathbb{B}_d$ and
$t(x)=\norm{\Shape^{-1}(x-z)}/\rho$. For $t(x)>1$, let
$y=m+\theta(x-m)$ be the intersection of the segment from $m$
to $x$ with the boundary of $E_{\rm out}$.
The point $m$ lies in the interior of $K_1$ and hence of
$E_{\rm out}$, so $0<\theta<1$.
Continuity implies $I(y)\ge1$: if $I(y)<1$, an open
neighborhood of $y$ would lie in $K_1\subset E_{\rm out}$,
contradicting its location on the boundary.

Put $u_m=(\rho \Shape)^{-1}(m-z)$, so $\norm{u_m}\le1$.
The reverse triangle inequality gives
\[
 1=\norm{(1-\theta)u_m+\theta(\rho \Shape)^{-1}(x-z)}
 \ge\theta(t(x)+1)-1.
\]
Thus $\theta\le2/(t(x)+1)$. Convexity and $I(m)=0$ yield
$1\le I(y)\le\theta I(x)$, and hence
\[
 I(x)\ge\frac{t(x)+1}{2}\ge\frac{t(x)-1}{2}=2U(x).
\]
Inside $E_{\rm out}$, $U=0$ and $I\ge-B$. This proves
\eqref{up:potential-envelope}. It implies
$h\ge10-B+U>1$. Also $I\ge-B$ gives $1+I>0$, and
$U\le(I+B)/2$ gives
$h\ge I/2+10-B/2$, from which
$1+I\le2(1+h)$ follows. This proves
\eqref{up:h-geometry}.

The inner ellipsoid $z+\Shape \mathbb{B}_d$ lies in $K_1$, so
\[
 \int_{\R^d}e^{-I(x)}\,dx
 \ge e^{-1}v_d|\det \Shape|.
\]
Using the cancellation between $q$ and $e^{-h}$,
\[
 \int q(x)e^{-h(x)}\,dx
 =\frac{e^{-10}}{v_d|\det \Shape|\rho^dD_d}
      \int e^{-I(x)}\,dx
 \ge\frac{e^{-11}}{\rho^dD_d}=p_0.
\]

Finally, the explicit radial mixture gives, on every history,
\[
 \E[R]\le1+4d,\qquad
 \E[R^2]\le1+8d+16d(d+1).
\]
For a tail component, these follow from the mean $k$ and second
moment $k(k+1)$ of a sum of $k\le d$ unit exponentials; the
interior component has $0\le R\le1$.
By \eqref{fx:geom-all-path-shape},
$\norm{Y_q-m}\le1+\rho R$.
The inequality $(a+b)^2\le2a^2+2b^2$ gives the first moment
bound in \eqref{up:proposal-moments}; Jensen's inequality for
the concave logarithm gives the second. Neither calculation uses
the good event or any property of the oracle's past errors.
\end{proof}

The preceding proposal bounds can now be coupled with the center
construction. We record the successful-pilot event using only the
implemented histories, so later conditioning never imposes a success
requirement on future oracle replies.
Run the center procedure with $\delta_m=\eps/8$, then run the
rounding procedure with $\delta_q=\eps/8$ at the returned center.
Use their specified fallbacks and physical-call caps on every history.
The two procedures use at most
\begin{equation}\label{up:pilot-pathwise-cost}
 T_{\rm pilot}\le C_d\bigl(1+\log(1+\kappa)+A/\eps\bigr)
\end{equation}
physical calls on every history. The output satisfies $\norm{m}\le2$;
the proposal parameters always satisfy the moment bounds of
Proposition~\ref{fx:proposal}.

Fix a potential $F$ and an admissible oracle. We define the
successful-pilot event directly from the implemented, capped procedures.

Let $\mathcal S_m$ be the event that the center procedure returns through the
certificate test \eqref{fx:center-stop}, and let $\mathcal S_q$ be the event that
the rounding procedure returns through its successful score test with
the parameters \eqref{fx:geom-returned-inner}. Both events exclude
fallback returns caused by the geometric-round or physical-call caps.

Let $\mathcal I_m$ and $\mathcal I_q$ be the finite random index sets of estimation
subroutines actually completed by the center and rounding procedures,
respectively. For $i\in \mathcal I_m$, let $\mathcal E_i^m$ be the completed-subroutine
event $\mathcal E_{\rm est}$ defined in the proof of
Lemma~\ref{fx:center-batch}. For $j\in \mathcal I_q$, let $\mathcal E_j^q$ be the event
that the actual completed subroutine satisfies both
\eqref{fx:geom-good-estimates} and \eqref{fx:geom-batch-good-cost},
with its corresponding true energy $M$. Define
\[
 \mathcal G_m=\mathcal S_m\cap\bigcap_{i\in \mathcal I_m}\mathcal E_i^m,\qquad
 \mathcal G_q=\mathcal S_q\cap\bigcap_{j\in \mathcal I_q}\mathcal E_j^q,\qquad
 \mathcal G=\mathcal G_m\cap\mathcal G_q.
\]
An intersection over an empty index set is understood to be the whole
sample space.

All these conditions refer to queries, replies, and decisions that
have actually occurred. For fixed $F$, the true gradients appearing
in the estimation events are deterministic continuous functions of
the recorded query points. The events need not be computable by the
algorithm to be measurable. Since the two pilots use no private random
draws, their states, completed subroutines, and return branches can be
reconstructed from the physical transcript. Consequently,
\[
 \mathcal G_m\in\mathcal F_{T_{\rm center}},\qquad
 \mathcal G_q,\mathcal G\in\mathcal F_{T_{\rm pilot}}.
\]
The auxiliary executions without the global physical-call cap in the
proofs of Propositions~\ref{fx:center} and \ref{fx:geometry} are used
to establish probability bounds for these actual events. Under the
common-prefix coupling, an auxiliary execution on which every estimation
event is good has total cost within the implemented budget and terminates
through the certificate branch. It therefore coincides with the
implemented execution and implies the corresponding event $\mathcal G_m$
or $\mathcal G_q$. Those proofs give
\[
 \P(\mathcal G_m^c)\le\delta_m,\qquad
 \ind_{\mathcal G_m}\P(\mathcal G_q^c\mid\mathcal F_{T_{\rm center}})
 \le\delta_q\ind_{\mathcal G_m}\quad\text{almost surely}.
\]
Here the rounding guarantee applies on $\mathcal G_m$, because the
returned center satisfies its required potential-gap promise. The tower
property therefore yields
\begin{equation}\label{up:pilot-good-probability}
 \P(\mathcal G^c)\le\delta_m+\delta_q=\eps/4.
\end{equation}
On $\mathcal G$, $F(m)-F(x_\star)\le B$, and the proposal has
ideal acceptance mass at least $p_0>0$, where $p_0$ depends only on $d$.

For clarity, the analytical quantities used for acceptance are
\begin{equation}\label{up:ideal-quantities}
 I(x)=F(x)-F(m),\qquad h(x)=I(x)-U(x)+10,
 \qquad \alpha_*(x)=e^{-h(x)}.
\end{equation}
The pilot does not evaluate these quantities. Its proved domination
$I\ge2U-B$, with $U\ge0$, implies
\begin{equation}\label{fx:combined-h-geometry}
 h\ge10-B>1,\qquad h\ge I/2+10-B/2,
 \qquad1+I\le2(1+h).
\end{equation}
Also, with $p_*:=\int q\alpha_*$,
\begin{equation}\label{up:ideal-acceptance-mass}
 p_*\ge p_0,\qquad
 \frac{q(x)\alpha_*(x)}{p_*}
 =\frac{e^{-F(x)}}{\int_{\R^d}e^{-F(y)}dy}.
\end{equation}
These assertions refer to a fixed successful pilot history; proposal
normalization and moment bounds remain valid on all histories.

\subsection{Directional envelopes and acceptance at every noise level}

Fix a proposed point $x\ne m$ and set
\begin{equation}\label{up:ray-notation}
 r=\norm{x-m},\qquad s=\frac{x-m}{r}\in\R^d,\qquad
 w(t)=\langle s,\nabla F(m+st)\rangle,\quad 0\le t\le r.
\end{equation}
The function $w$ is increasing and $\kappa$-Lipschitz, for every such unit direction
$s$.  The following bound allows its positive and negative parts to be
treated separately even when $\norm{\nabla F(m)}$ is large.

\begin{lemma}[Absolute gradient mass along a segment]\label{up:absolute-mass}
On a good pilot history, put $V_I=1+I(x)$.  Then
\begin{equation}\label{up:absolute-mass-bounds}
 \int_0^r|w(t)|\,dt\le I(x)+2B\le V_I,
 \qquad r^2\le4(I(x)+2B)\le4V_I,
 \qquad V_I\ge\frac{61}{64}.
\end{equation}
\end{lemma}

\begin{proof}
Since $w$ is increasing, its negative part occupies an initial segment,
possibly empty or all of $[0,r]$.  Its negative mass is exactly the
decrease of the potential from $m$ to the minimum along this segment:
\[
 \int_0^r\max\{-w(t),0\}\,dt
 =F(m)-\min_{0\le t\le r}F(m+st)\le B.
\]
Using $\int_0^r w=I(x)$ gives the first inequality.
Strong convexity at $x_\star$ gives
\[
 \norm{x-x_\star}^2\le2(I(x)+B),\qquad
 \norm{m-x_\star}^2\le2B.
\]
The triangle inequality and $(u+v)^2\le2u^2+2v^2$ yield the stated
bound on $r^2$.  Finally, $I\ge-B$, $2B<1$, and $B=3/64$ give the
remaining assertions.
\end{proof}

This absolute-mass bound suggests separate upper envelopes for the
positive derivative and the reflected negative derivative. The next
lemma explains why a geometric grid controls their integrals with only
a constant terminal remainder.

\begin{lemma}[A deterministic geometric envelope]\label{up:geometric-lemma}
Let $v:[0,r]\to[0,\infty)$ be nondecreasing and
$\kappa$-Lipschitz, with $r>0$.  Define
\begin{equation}\label{up:grid-definition}
 J=\max\left\{0,\left\lceil\log_2(r\sqrt\kappa)\right\rceil\right\},
 \qquad d_j=r2^{-j},\qquad t_j=r-d_j\quad(0\le j\le J).
\end{equation}
On each ordinary interval $(t_{j-1},t_j)$, $1\le j\le J$, place a
rectangle of height $v(t_j)$.  On the terminal interval $(t_J,r)$,
place a rectangle of height $v(r)$.  Their union is an upper envelope
$g_{\mathrm{exact}}$ satisfying
\begin{equation}\label{up:geometric-integral}
 \int_0^r g_{\mathrm{exact}}(t)\,dt
 \le2\int_0^r v(t)\,dt+\frac12.
\end{equation}
\end{lemma}

\begin{proof}
The definition gives $d_J\le\kappa^{-1/2}$.  Monotonicity gives the
upper-envelope property.  If $J=0$, there is only a terminal rectangle,
and Lipschitz continuity gives
\[
 rv(r)\le\int_0^r v(t)\,dt+\frac{\kappa r^2}{2}
 \le\int_0^r v(t)\,dt+\frac12.
\]
Suppose $J\ge1$.  An ordinary interval has length $d_j$, while its
successor has length $d_j/2$.  For $1\le j\le J-1$,
\[
 d_jv(t_j)\le2\int_{t_j}^{t_{j+1}}v(t)\,dt.
\]
The last ordinary rectangle satisfies
$d_Jv(t_J)\le\int_{t_J}^r v(t)\,dt$.  The terminal rectangle satisfies
\[
 d_Jv(r)\le\int_{t_J}^r v(t)\,dt+\frac{\kappa d_J^2}{2}.
\]
Summing these inequalities charges every integral on the right at most
twice and leaves an error at most $1/2$, proving
\eqref{up:geometric-integral}.
\end{proof}

Partition-point values do not affect these integrals.  To fix the
algorithm on all real-valued paths, at an interior partition point we
use the maximum of the two adjacent rectangle heights, and at $0$ or
$r$ the corresponding one-sided value.  We use the same convention for
the estimated rectangle functions below.  It preserves an upper-envelope
inequality whenever that inequality holds on the adjacent intervals.

\paragraph{Constructing the two noisy envelopes.}
Let $0<\delta<1$ be a public error budget, to be fixed below.  We apply
the grid \eqref{up:grid-definition} to the two nonnegative,
nondecreasing, $\kappa$-Lipschitz functions
\begin{equation}\label{up:two-parts}
 v_+(t)=\max\{w(t),0\},\qquad
 v_-(u)=\max\{-w(r-u),0\}.
\end{equation}
For each of the two grids and each ordinary node $j=1,\ldots,J$, use
\begin{equation}\label{up:grid-batches}
 e_j=\frac{2^{j/2}}r,\qquad
 \delta_j=\frac\delta8\,2^{-j/2},\qquad
 \beta_j=\max\left\{1,
 \left\lceil\frac{8Ar^2}{\delta}2^{-j/2}\right\rceil\right\}.
\end{equation}
For the terminal node, use
\begin{equation}\label{up:terminal-batch}
 e_E=\frac1{d_J},\qquad \delta_E=\frac\delta8,
 \qquad \beta_E=\max\left\{1,
 \left\lceil\frac{8Ad_J^2}{\delta}\right\rceil\right\}.
\end{equation}

Execute the positive grid first, in increasing node order, followed by its
terminal node.  At its ordinary node $t_j$, query $m+st_j$ exactly
$\beta_j$ times and take the inner product of its vector average with $s$.  At its terminal
node, query $m+sr=x$ exactly $\beta_E$ times and again project the
vector average onto $s$.  Denote a resulting signed average by $\widehat w$.
The height used on the corresponding ordinary or terminal rectangle is
$\max\{0,\widehat w+e\}$, with its corresponding $e_j$ or $e_E$.

Next execute the negative grid in the same order.  At its node $t_j$,
query $m+s(r-t_j)$ exactly $\beta_j$ times and project the vector average
onto $-s$; at its terminal node, query $m$ exactly $\beta_E$ times and
project its vector average onto $-s$.  This estimates the signed quantity $-w(r-u)$ before
its positive part is taken.  Use the same rectangle-height rule.
The grids use separate batches; coincident query points are still
queried and charged separately.

Let $\widehat v_+(t)$ and $\widehat v_-(u)$ be the resulting rectangle
functions, with the stated partition-point convention, and put
\begin{equation}\label{up:estimated-envelope}
 \begin{split}
 g_+(t)&=\widehat v_+(t),\qquad
 g_-(t)=\widehat v_-(r-t),\qquad g(t)=g_+(t)+g_-(t),\\
 Z_g&=\int_0^r g(t)\,dt,\qquad Z_-=\int_0^r g_-(t)\,dt.
 \end{split}
\end{equation}
The function $g$ is nonnegative, finite, and piecewise constant on every
real-valued reply history.  The algorithm computes its integral by
summing rectangle areas.

\begin{lemma}[Accuracy and cost of the noisy envelopes]\label{up:noisy-envelope}
Condition on the complete history before this envelope construction,
including the point $x\ne m$.  With conditional probability at least
$1-\delta$, all signed averages are within their assigned tolerances.
Denote this event by $\mathcal E$.  On $\mathcal E$,
\begin{equation}\label{up:noisy-envelope-bounds}
 g(t)\ge|w(t)|\quad(0\le t\le r),\qquad
 Z_g\le2\int_0^r|w(t)|\,dt+16.
\end{equation}
More precisely, $g_+\ge w_+$ and $g_-\ge w_-$, and each
envelope has integral at most twice its corresponding true mass plus
$8$. Consequently, on a good pilot history and $\mathcal E$,
\begin{equation}\label{up:separate-envelope-masses}
 Z_-\le2B+8,\qquad Z_g\le2I(x)+4B+16.
\end{equation}
On every real-valued history, the number of calls used by both grids is
bounded by
\begin{equation}\label{up:noisy-envelope-cost}
 N_{\mathrm{grid}}\le2J+2+\frac{56Ar^2}{\delta}.
\end{equation}
The grid call count depends only on $x$, the pilot values, and public
parameters, not on the observed heights.
\end{lemma}

\begin{proof}
The two types of batch size can equivalently be written as
\[
 \beta_j=\max\{1,\lceil A/(\delta_je_j^2)\rceil\},\qquad
 \beta_E=\max\{1,\lceil A/(\delta_Ee_E^2)\rceil\}.
\]
The lower bound of one makes all means well defined also when $A=0$.
Lemma~\ref{up:mean-lemma} applies
to each signed average, even when its oracle kernels depend on replies
from earlier nodes.  Therefore
\begin{equation}\label{up:envelope-error-probability}
 \mathbb P(\mathcal E^c\mid\text{incoming history})
 \le2\left(\sum_{j=1}^{\infty}\delta_j+\delta_E\right)
 =\frac\delta4(\sqrt2+2)<\delta.
\end{equation}
Here $\sum_{j\ge1}2^{-j/2}=\sqrt2+1$.

If $|\widehat w-a|\le e$, then
\[
 \max\{a,0\}\le\max\{\widehat w+e,0\}
 \le\max\{a,0\}+2e.
\]
Consequently each estimated rectangle is an upper bound for the
corresponding true function on $\mathcal E$.  In one grid the excess
area from ordinary-node estimation is at most
\[
 2\sum_{j=1}^Jd_je_j
 \le2\sum_{j=1}^{\infty}2^{-j/2}
 =2(\sqrt2+1)<5.
\]
The terminal-node excess area is $2d_Je_E=2$.
Adding the deterministic error $1/2$ from
Lemma~\ref{up:geometric-lemma} gives an excess less than $8$ for each
grid beyond twice the corresponding true integral.  Reflection preserves
the negative-part integral, proving \eqref{up:noisy-envelope-bounds}.
The negative true mass is at most $B$ by the proof of
Lemma~\ref{up:absolute-mass}, which also proves
\eqref{up:separate-envelope-masses}.

Finally, using $\max\{1,\lceil a\rceil\}\le1+a$ for $a\ge0$ gives
\[
 N_{\mathrm{grid}}
 \le2J+2+\frac{16Ar^2}{\delta}
 \left(\sum_{j=1}^{J}2^{-j/2}+2^{-2J}\right).
\]
The coefficient in parentheses is at most $\sqrt2+2$, and
$16(\sqrt2+2)<55<56$.  This proves
\eqref{up:noisy-envelope-cost} on all histories, without using
$\mathcal E$ or a good pilot.
\end{proof}

The envelope lemma controls the two parts of the target derivative.
Acceptance also subtracts the derivative of the proposal potential;
we therefore need a pointwise comparison along the same ray, not only
an integrated potential bound.

\begin{lemma}[Derivative domination along a proposed ray]
\label{up:proposal-ray-derivative}
Fix a good pilot and a unit vector $s$, and set
$u(t)=U(m+st)$ and $w(t)=\langle s,\nabla F(m+st)\rangle$
for every $t\ge0$. Give $u'$ the value zero in the flat part
and at the boundary of the outer proposal ellipsoid. Then
\begin{equation}\label{up:proposal-ray-bounds}
 u'(t)\ge0,\qquad
 w(t)\ge2u'(t)\quad\text{whenever }u(t)>0.
\end{equation}
The stated version of $u'$ agrees with the derivative almost everywhere
and is computable from the proposal parameters without potential values.
\end{lemma}
\begin{proof}
Put $a=(\rho\Shape)^{-1}(m-z)$ and $b=(\rho\Shape)^{-1}s$.
The good-pilot sandwich is
$z+\Shape\mathbb{B}_d\subset K_1\subset z+\rho\Shape\mathbb{B}_d$.
Since $I(m)=0<1$, continuity gives a neighborhood of $m$ in $K_1$;
thus $\norm a<1$. Let $t_0>0$ be the positive exit time determined by
$\norm{a+t_0b}=1$. Since $b\ne0$, this time is finite and unique.
The triangle inequality gives $t_0\norm b\le2$.
Before exit, $u'=0$. After exit, the norm of the affine ray is
increasing, and the explicit proposal formula gives
\[
 0\le u'(t)=\frac{\langle b,a+tb\rangle}{4\norm{a+tb}}
 \le\frac{\norm b}{4}\le\frac1{2t_0}.
\]
At the boundary, $I(m+st_0)\ge1$. Otherwise continuity would give a
neighborhood of that boundary point contained in $K_1$, contradicting
the outer containment. Convexity along the ray implies
$t_0w(t_0)\ge I(m+st_0)\ge1$. Monotonicity of $w$ then proves
\eqref{up:proposal-ray-bounds} for $t>t_0$.
At the boundary our convention sets $u'=0$; the single chosen value
does not change any integral. The proof also covers $t_0>r$, when the
entire segment used by the algorithm is in the flat part.

For operational purposes evaluate the displayed derivative formula
when $\norm{a+tb}>1$ and use zero otherwise. This prescription is
finite and measurable on every history, including a failed pilot.
Its correctness inequalities are used only on good pilots.
\end{proof}

The derivative comparison lets us compensate the negative part of the
target derivative while retaining a nonnegative marking intensity.
We introduce the computable position density first, and only then the
unknown intensity used to analyze its marks.

For the remainder of the acceptance construction fix $c=1/4$.
After the grids have been completed, define the known density numerator
and its normalizer by
\begin{equation}\label{up:integration-density}
 a_0(t)=g_+(t)+g_-(t)+\frac{2c}{r},\qquad
 Q=Z_g+2c=\int_0^r a_0(t)\,dt.
\end{equation}
These definitions are valid on every history: $Q\ge2c>0$ and
$a_0(t)>0$. To sample from $a_0/Q$, assign weight $2c/Q$ to the
uniform distribution on $(0,r)$ and assign each envelope rectangle
its area divided by $Q$. Reflect the negative-grid rectangles into
the $t$ coordinate. Omit zero-weight components. The weights sum to
one; overlapping rectangles simply contribute additively to the density.
There are finitely many components on every finite reply history.

For analysis, introduce the unknown intensity and its integral
\begin{equation}\label{up:ideal-mark-intensity}
 b_0(t)=w(t)+g_-(t)-u'(t)+\frac c r,
 \qquad \ell=I(x)+Z_--U(x)+c=\int_0^r b_0(t)\,dt.
\end{equation}
The last equality uses $U(m)=0$, so it is asserted on a good pilot.
Neither $b_0$ nor $\ell$ is evaluated by the algorithm.

\begin{lemma}[Margins and a uniform mark probability]
\label{up:mark-geometry}
On a good pilot and a completed good envelope history,
\begin{equation}\label{up:mark-margins}
 \frac c r\le b_0(t)\le a_0(t)-\frac c r
 \qquad(0\le t\le r).
\end{equation}
Moreover,
\begin{equation}\label{up:mark-mass-bounds}
 \ell\ge c,\qquad r^2\le8\ell,\qquad
 Q\le68\ell,
 \qquad p:=\frac\ell Q\ge\frac1{68}=:p_{\min}.
\end{equation}
\end{lemma}
\begin{proof}
Inside the outer ellipsoid, $u'=0$ and $w+g_-\ge0$.
Outside it, Lemma~\ref{up:proposal-ray-derivative} gives
$w-u'\ge0$. These facts prove the lower margin.
For the upper margin use
$a_0-b_0=g_+-w+u'+c/r\ge c/r$.
Integrating gives $\ell\ge c$.

The good-pilot inequality $U\le(I+B)/2$ yields
\[
 I\le2\ell-2Z_-+B-2c.
\]
Combine this with \eqref{up:absolute-mass-bounds} to obtain
\[
 r^2\le4(I+2B)
 \le8\ell-8Z_-+12B-8c\le8\ell,
\]
since $Z_-\ge0$ and $3B\le2c$.
By \eqref{up:separate-envelope-masses},
\[
 Q\le2I+4B+16+2c
 \le4\ell-4Z_-+6B+16-2c
 \le4\ell+6B+16-2c.
\]
As $\ell\ge1/4$, the ratio $Q/\ell$ is at most
$4+4(6B+16-2c)=67.125<68$. This proves all the claims.
\end{proof}

On all histories, set the computable prefactor probability to
\begin{equation}\label{up:prefactor-probability}
 v=\min\{1,\exp(Z_-+c-10)\}.
\end{equation}
On a good pilot and good envelope, $Z_-+c\le2B+8+1/4<10$,
so clipping is inactive. The ideal identity is then
\begin{equation}\label{up:poisson-identity}
 v e^{-\ell}=\exp[-I(x)+U(x)-10]=\alpha_*(x).
\end{equation}
Indeed, if $N$ has the Poisson distribution with mean $Q$, and a
point with density $a_0/Q$ is independently marked with probability
$b_0(t)/a_0(t)$, its mean mark probability is $p=\ell/Q$.
The probability of no marks at $N$ such independent points is
$\E[(1-p)^N]=e^{-Qp}=e^{-\ell}$. This is an analytical identity;
the implementable finite construction below does not evaluate the
unknown intensity or generate an unbounded number of nodes.
The random negative-envelope mass $Z_-$ cancels exactly in
\eqref{up:poisson-identity}.

The ideal identity reduces sampling to Bernoulli marks, but the oracle
provides only unbiased estimates of their probabilities. The two-sided
margins above are what make clipping affordable: its bias is controlled
by a variance rather than a standard deviation.

For a real number $y$, write
$\operatorname{clip}_{[0,1]}y=\min\{1,\max\{0,y\}\}$.

\begin{lemma}[Clipping with a two-sided margin]\label{up:clipping-lemma}
Let $Y$ be square integrable, with mean $z\in[\eta,1-\eta]$,
where $\eta>0$. Then
\begin{equation}\label{up:clipping-remainder}
 \bigl|\E[\operatorname{clip}_{[0,1]}(Y)]-z\bigr|
 \le\frac{\operatorname{Var}(Y)}{\eta}.
\end{equation}
\end{lemma}
\begin{proof}
Clipping changes $Y$ only on the disjoint events $Y<0$ and $Y>1$.
On either event the absolute change is at most $(Y-z)^2/\eta$:
for example, if $Y<0$, then $z-Y\ge\eta$ and
$-Y\le z-Y\le(z-Y)^2/\eta$. The upper event is identical after
reflection about $1/2$. Taking expectations and using $\E[Y]=z$
proves the assertion.
\end{proof}

For each correction node that is reached, draw a fresh position
$t$ with density $a_0/Q$. Make $n$ fresh vector-gradient calls at
$m+st$, average them, and project onto $s$ to obtain $\overline w$.
After all these replies, use a new uniform to make a rejection mark
with probability
\begin{equation}\label{up:acceptance-rule}
 \operatorname{clip}_{[0,1]}
 \left(\frac{\overline w+g_-(t)-u'(t)+c/r}{a_0(t)}\right).
\end{equation}
This is a legal probability on every history. The density $a_0/Q$
is held fixed throughout this trial's correction stage, even if
subsequent oracle kernels adapt to the entire history.

The clipping estimate applies conditionally at each fresh position.
Integrating against the fixed position density cancels its denominator,
which is the key step in the following relative-error bound.

\begin{lemma}[Conditional mark error]\label{up:conditional-mark-error}
Fix a good pilot and a completed good envelope history. Let $p_j$
be the actual mark probability conditional on any complete history
before a fresh correction position is generated. Such a history may
include the already drawn Poisson category, prior marks, and all ignored
marker replies. Then, almost surely,
\begin{equation}\label{up:relative-mark-error}
 |p_j-p|\le\eta_0:=\frac{Ar^2}{ncQ},
 \qquad \frac{\eta_0}{p}\le\frac{32A}{n}.
\end{equation}
\end{lemma}
\begin{proof}
Condition also on the fresh position $t$.
Lemma~\ref{up:mean-lemma} gives conditional mean $w(t)$ and variance
at most $A/n$ for $\overline w$, with no independence assumption
on the $n$ oracle replies. The untruncated fraction in
\eqref{up:acceptance-rule} has mean $b_0(t)/a_0(t)$, variance at
most $A/(na_0(t)^2)$, and distance at least $c/(ra_0(t))$ from
each endpoint of $[0,1]$, by \eqref{up:mark-margins}.
Lemma~\ref{up:clipping-lemma} therefore bounds its clipping bias
by $Ar/(nc a_0(t))$. The fresh marking uniform realizes the
clipped expectation as a mark probability. Integrating over $t$ with
density $a_0/Q$ bounds the unconditional node bias by
\[
 \int_0^r\frac{Ar}{nc a_0(t)}\frac{a_0(t)}Q\,dt
 =\frac{Ar^2}{ncQ}.
\]
Division by $p=\ell/Q$, followed by $r^2\le8\ell$ and $c=1/4$,
gives the second inequality.

The good envelope event concerns only completed calls. Fixing such a
history does not condition on future oracle errors. After any subsequent
marker reply, the complete-history oracle assumption and the same
bounds apply again. In particular, revealing a privately drawn category
before the node does not change this proof.
\end{proof}

The relative mark-error bound is uniform over all reached histories.
We now select its accuracy budget and cap the correction procedure.
The proof will propagate conditional survival probabilities instead of
assuming that noisy marks are independent.

Use the public parameters
\begin{equation}\label{up:sampling-parameters}
 \begin{split}
 \tau&=\frac{p_0\varepsilon}{16},\qquad
 \delta=\frac\tau2,\qquad
 n=\max\left\{1,\left\lceil\frac{256A}{\tau}\right\rceil\right\},\\
 H&=\left\lceil136\log\frac4\tau\right\rceil,
 \qquad
 K=\left\lceil\frac2{p_0}\log\frac{16}{\varepsilon}\right\rceil.
 \end{split}
\end{equation}
Here $p_0=e^{-11}/(\rho^dD_d)$ is the constant for the ideal
acceptance probability with offset $10$.
Use this $\delta$ for each envelope construction and the same
fixed $n$ for every correction node. Write
\begin{equation}\label{up:mark-relative-budget}
 \beta=\frac{32A}{n}\le\frac\tau8<\frac12.
\end{equation}
Thus, on a good envelope, every reached node has conditional mark
probability at least $(1-\beta)p\ge1/136$.

After a successful prefactor coin, generate a finite-valued category
\begin{equation}\label{up:finite-poisson-category}
 N_c\ \overset{\mathrm{law}}=\ \min\{\operatorname{Pois}(Q),H+1\}.
\end{equation}
Specifically, assign mass $e^{-Q}Q^j/j!$ to each $j=0,\ldots,H$
and the remaining mass to $H+1$, and use one fresh uniform to select
a category. Only these finitely many weights are computed.
If the prefactor fails, reject immediately. If it succeeds, category
$0$ accepts immediately and category $H+1$ rejects immediately.
For $1\le N_c\le H$, process at most $N_c$ nodes using
\eqref{up:acceptance-rule}; reject on the first mark, and accept
if every node is unmarked. The charged marker calls specified below
are made before a trial is left, including in the immediate branches.

If $x=m$, omit the grids, category generation, and every formula
involving $1/r$. Use a fresh uniform to accept with probability
$\min\{1,\exp(U(m)-10)\}$. On a good pilot this equals
$e^{-10}=\alpha_*(m)$.

\begin{lemma}[Conditional acceptance error]\label{up:acceptance-error-lemma}
Fix a potential $F$ and an admissible oracle. For $1\le k\le K$, let
$R_k$ be the event that trial $k$ is reached. On $R_k$, let $H_k$
be the complete history immediately before its fresh proposal $Y_k$ is
generated, and let $B_k$ be its acceptance indicator. Thus $H_k$
includes all internal randomness already generated, but no future random
draws.

For an incoming history $h$, write $m_h,q_h,U_h$ for the center,
proposal density, and proposal potential determined by its completed
pilot, and define
\[
 \alpha_{*,h}(x)=\exp[-F(x)+F(m_h)+U_h(x)-10].
\]
When $\P(R_k)>0$, choose a jointly measurable version of the regular
conditional acceptance probability
\[
 \alpha_k(h,x)=\P(B_k=1\mid R_k,H_k=h,Y_k=x).
\]
This conditional probability averages over the initial marker reply,
the envelopes, the prefactor and category, correction nodes, all
remaining marker replies, and future internal random draws. Define the
subprobability measure on incoming histories by
\[
 \lambda_k(D)=\P\bigl(\mathcal G\cap R_k\cap\{H_k\in D\}\bigr).
\]
Then, for $\lambda_k$-almost every $h$ and for $q_h(x)\,dx$-almost every $x$,
\begin{equation}\label{up:acceptance-error}
 |\alpha_k(h,x)-\alpha_{*,h}(x)|\le\tau.
\end{equation}
If $\P(R_k)=0$, the assertion is vacuous. The exceptional history sets
may depend on $F$, the oracle, and $k$; the exceptional proposal set may
also depend on $h$.

At $x=m_h$, we may choose the version
\[
 \alpha_k(h,m_h)=\min\{1,\exp(U_h(m_h)-10)\},
\]
consistently with the specified zero-radius branch. Since $q_h$ has a
Lebesgue density, this changes the conditional function only on a set
of zero joint history--proposal measure.

When the trial index and incoming history are understood, we retain
the notation $\alpha_{\mathcal H}(x)=\alpha_k(h,x)$ and
$\alpha_*(x)=\alpha_{*,h}(x)$, and write $q=q_h$.
\end{lemma}
\begin{proof}
The finite histories of this algorithm can be encoded in a countable
disjoint union of finite-dimensional Borel spaces, so regular conditional
probabilities may be chosen. All conditional estimates below are
interpreted under the law of the reached $k$th trial and disintegrated
with respect to $(H_k,Y_k)$.

On $R_k$, the event $\mathcal G$ is determined by $H_k$.
Restricting to good incoming histories therefore imposes no additional
success condition on future oracle errors. The conditional oracle bounds
are used only almost surely at actually reached calls. Conditional
expectation and disintegration transfer the ensuing estimates to
$\lambda_k$-almost every incoming history and $q_h(x)\,dx$-almost every
proposal. No pointwise oracle guarantee on exceptional histories is
required.

For $x\ne m$, condition first on the initial marker reply and on a
completed good envelope history. The parameters $p$, $Q$, and $v$
are then fixed for analysis. Condition also on successful prefactor
selection, a category $N_c=j\le H$, and the reply to its control
marker. Lemma~\ref{up:conditional-mark-error} applies to every
subsequent complete history. Let $s_l$ be the probability that the
first $l$ nodes are all unmarked, with $s_0=1$. On the event of
surviving $l-1$ nodes, the next conditional mark probability differs
from $p$ by at most $\eta_0$. Therefore
\[
 |s_l-(1-p)s_{l-1}|\le\eta_0 s_{l-1}\le\eta_0.
\]
Iterating this scalar inequality gives
\begin{equation}\label{up:finite-survival-error}
 |s_j-(1-p)^j|
 \le\eta_0\sum_{l=0}^{j-1}(1-p)^l
 \le\frac{\eta_0}{p}\le\beta\le\frac\tau8.
\end{equation}
The formula includes $j=0$, with zero error. It uses neither independent
oracle replies nor hypothetical replies after the algorithm has stopped.
Averaging over the control reply and the categories $j\le H$ preserves
the error bound $\tau/8$. The ideal Poisson survival mass omitted
by the overflow category is at most
\begin{equation}\label{up:poisson-tail-error}
 \sum_{j>H}e^{-Q}\frac{Q^j}{j!}(1-p)^j
 \le(1-p)^H\le e^{-H/68}\le\frac\tau4.
\end{equation}
The full ideal sum is $e^{-\ell}$. Multiplying by $v\le1$ and
using \eqref{up:poisson-identity} proves an acceptance error at most
$3\tau/8$ on this completed good envelope history.

Conditional on the history before the grids, their failure probability
is at most $\delta=\tau/2$ by
\eqref{up:envelope-error-probability}. Both actual and ideal acceptance
probabilities are in $[0,1]$ on the entire fixed good pilot, so failed
envelopes contribute at most $\delta$. Averaging the completed grid
histories and the initial marker reply gives an error at most
$3\tau/8+\tau/2<\tau$. At $x=m$, the explicit branch has zero
error. Final marker replies cannot alter the previously selected
acceptance bit or proposal. This proves the assertion.
\end{proof}

The preceding lemma controls one trial after any good incoming pilot
history. We now specify how trials are concatenated and how every
private decision is recorded in the physical transcript; the next
lemma then converts the one-trial bound into a bound on the final law.

Execute the two pilots and fix their $m,z,\Shape,q,U$. Attempt at most
$K$ trials, each with a fresh proposal from $q$. In normalized
coordinates, implement a reached trial as follows. All marker replies
are charged and ignored, but belong to the complete history conditioning
of every later oracle call.
\begin{enumerate}
 \item Generate $x\sim q$ by the finite radial mixture of
 Proposition~\ref{fx:proposal}, and make one initial marker call at $x$.
 \item If $x=m$, draw the zero-radius acceptance bit specified above
 and proceed directly to the final marker in step~5.
 Otherwise construct the two grids in their prescribed order, with
 \eqref{up:grid-batches} and \eqref{up:terminal-batch}.
 \item Generate the prefactor coin with probability
 \eqref{up:prefactor-probability}. If it fails, set the trial's
 acceptance bit to zero and make a control marker call at $0e_1$.
 If it succeeds, generate $N_c$ by
 \eqref{up:finite-poisson-category}, and make the control marker call
 at $(N_c+1)e_1$. Category $0$ sets the acceptance bit to one;
 category $H+1$ sets it to zero. These branches then proceed to step~5.
 \item For a passed prefactor with $1\le N_c\le H$, generate each
 reached node position freshly from $a_0/Q$, make exactly $n$ calls
 there, and draw its rejection mark by \eqref{up:acceptance-rule}.
 Make one node marker call at $be_1$, where $b=1$ denotes a mark and
 $b=0$ its absence. A mark ends the node stage and sets the trial's
 acceptance bit to zero. After $N_c$ unmarked nodes, set that bit to one.
 \item Make one final marker call at $ae_1$, where $a\in\{0,1\}$ is
 the trial's acceptance bit. Only after its reply, if $a=1$, stop and
 return $Y=x/\sqrt\mu$. If $a=0$ on trial $K$, stop and return
 $Y=m/\sqrt\mu$. Otherwise begin the next trial.
\end{enumerate}
Thus an immediate decision in the probabilistic description always
includes the control and final markers prescribed by its branch.
The algorithm queries no potential values, mode, Hessian, or target
sample. Section~\ref{up:operational-section} proves that its stopping
rule is measurable in the smaller physical filtration.

Although trial acceptance probabilities may depend on earlier replies,
the accepted law after each reached good-pilot history is close to the
same target. The argument below retains both normalizing constants and
then mixes over the first accepting trial.

\begin{lemma}[Total variation under adaptive trials]\label{up:tv-lemma}
The complete sampler satisfies
\begin{equation}\label{up:tv-result}
 \norm{\mathcal L(Y)-\nu_f}_{\TV}\le\frac{7\varepsilon}{16}.
\end{equation}
\end{lemma}
\begin{proof}
Throughout this proof, $Y$ denotes the algorithm's output in the
original coordinates, and
\[
 Y_{\mathrm{norm}}:=\sqrt{\mu}\,Y
\]
denotes its output before inverse scaling. Thus $Y_{\mathrm{norm}}$
is the first accepted proposal, or $m$ if all $K$ trials reject.
Write
\[
 \nu_F(dz)
 =
 \frac{e^{-F(z)}}{\int_{\R^d}e^{-F(u)}\,du}\,dz
\]
for the target distribution in normalized coordinates. We first
bound the total variation distance between the law of
$Y_{\mathrm{norm}}$ and $\nu_F$, and then transfer this bound to
the original coordinates. All statements below conditional on a completed pilot history or
a reached trial history are understood almost surely under the
corresponding disintegrated laws. For the finitely many trials,
we discard the union of the exceptional null sets supplied by
Lemma \ref{up:acceptance-error-lemma}.
Fix a good completed pilot history. For any subsequent complete history
$\mathcal H$ before a reached trial's proposal, put
$p_{\mathcal H}=\int q(x)\alpha_{\mathcal H}(x)\,dx$.
The ideal acceptance mass \eqref{up:ideal-acceptance-mass} and
Lemma~\ref{up:acceptance-error-lemma} give
\begin{equation}\label{up:actual-acceptance-mass}
 |p_{\mathcal H}-p_*|\le\tau,
 \qquad p_{\mathcal H}\ge p_0-\tau\ge p_0/2.
\end{equation}
Conditional on this trial accepting, its proposal density is
$q\alpha_{\mathcal H}/p_{\mathcal H}$. Retaining both normalizers,
its total variation distance from $\nu_F$ is at most
\begin{align}
 &\frac12\int_{\R^d}
 \left|\frac{q(x)\alpha_{\mathcal H}(x)}{p_{\mathcal H}}
       -\frac{q(x)\alpha_*(x)}{p_*}\right|\,dx\notag\\
 &\quad\le\frac1{2p_{\mathcal H}}
 \left(\int_{\R^d}q(x)|\alpha_{\mathcal H}(x)-\alpha_*(x)|\,dx
       +|p_*-p_{\mathcal H}|\right)\notag\\
 &\quad\le\frac\tau{p_0-\tau}
 =\frac{\varepsilon}{16-\varepsilon}
 \le\frac\varepsilon8.\label{up:accepted-law-tv}
\end{align}
All preceding rejection decisions and ignored replies belong to
$\mathcal H$. Thus this bound holds after any such history, even if
the oracle kernels changed adaptively.

Disintegrate over the index of the first accepting trial and its
incoming history. Earlier rejections are already determined by that
history, and conditioning its current trial on acceptance gives exactly
the law bounded in \eqref{up:accepted-law-tv}. The final marker reply
changes neither the selected bit nor the recorded proposal.
Consequently, conditional on the fixed good pilot history and on at
least one acceptance, the law of $Y_{\mathrm{norm}}$ is a mixture
of laws each within $\varepsilon/8$ of $\nu_F$.
Repeated conditional expectation in
\eqref{up:actual-acceptance-mass} also yields
\begin{equation}\label{up:all-reject-probability}
 \P(\text{all }K\text{ trials reject}\mid\text{good pilot history})
 \le(1-p_0/2)^K\le e^{-p_0K/2}\le\varepsilon/16.
\end{equation}
If all trials reject, the normalized output is $Y_{\mathrm{norm}}=m$.
This fallback contributes at most the probability above to total
variation. Therefore, conditional on almost every good pilot history,
the law of $Y_{\mathrm{norm}}$ is within $3\varepsilon/16$ of $\nu_F$.

The good-pilot event $\mathcal G$ is measurable at the completion of
the pilot and imposes no condition on future noise. Mixing over good
pilot histories preserves the preceding bound, while failed pilot
histories contribute at most $\P(\mathcal G^c)\le\varepsilon/4$.
Hence
\[
 \norm{\mathcal L(Y_{\mathrm{norm}})-\nu_F}_{\TV}
 \le \frac{3\varepsilon}{16}+\P(\mathcal G^c)
 \le \frac{7\varepsilon}{16}.
\]
Finally, since $F(z)=f(z/\sqrt{\mu})$, the map
$z\mapsto z/\sqrt{\mu}$ sends $\nu_F$ to $\nu_f$.
This map is a measurable bijection with a measurable inverse, so
total variation is invariant under it. Since
$Y=Y_{\mathrm{norm}}/\sqrt{\mu}$, we obtain
\[
 \norm{\mathcal L(Y)-\nu_f}_{\TV}
 =
 \norm{\mathcal L(Y_{\mathrm{norm}})-\nu_F}_{\TV}
 \le \frac{7\varepsilon}{16}.
\]
This proves the assertion.
\end{proof}

\begin{remark}
The failure probability of the envelopes is included in each trial's
conditional acceptance error. It is not union-bounded over all $K$
trials. Random integration positions are used directly to generate marks;
they create no approximation error when $A=0$. Only the clipping bias
from the oracle noise is bounded in
Lemma~\ref{up:conditional-mark-error}.
\end{remark}

\subsection{Operational completion and endpoint cases}
\label{up:operational-section}

\begin{proposition}[Upper bound and operational admissibility]
\label{prop:upper}
For every $A\ge0$, the complete sampler is measurable, has a
physical-transcript stopping time, terminates on every specified
finite-real-reply path with open-interval uniforms, and obeys
\[
 \norm{\mathcal L(Y)-\nu_f}_{\TV}\le7\eps/16,
 \qquad \E[T]\le C_d S.
\]
The cost includes failed pilot histories and every ignored reply.
\end{proposition}
\begin{proof}
Lemma~\ref{up:tv-lemma} already proves the accuracy assertion.
We verify execution, physical stopping, and the uniform expected cost.

\paragraph{Finite execution and measurability.}
Both pilots have deterministic limits on geometric updates and physical
calls. Every matrix update is nonsingular, and every fallback preserves
the stated bounded, nonsingular proposal geometry. The pilots' averages,
free cuts, and budget decisions depend only on public parameters and
physical replies. Proposal generation uses a finite mixture and finitely
many open-interval uniforms and Gaussian coordinates. Its specified
fixed-direction convention on the zero-Gaussian-vector event completes
the definition on every path.

For a finite proposal with $r>0$, the grid length $J$ and all positive
batch sizes are finite. Every envelope height and its integral are
finite on a finite-real-reply path. The correction normalizer satisfies
$Q\ge2c>0$, and the position density $a_0/Q$ is everywhere positive.
The category generation in \eqref{up:finite-poisson-category} uses
exactly $H+2$ categories and a finite list of explicitly given weights.
It does not require an unbounded Poisson simulation or the generation
of any unused future correction nodes. Every reached node has a fixed
finite batch of $n$ calls; there are at most $H$ such nodes and at most
$K$ trials. The zero-radius branch contains no division by zero.
The clipped probabilities remain defined on all bad histories.
Thus every specified path terminates after finitely many calls.

All operations are measurable: matrix arithmetic and norms on
nonsingular matrices, finite sums, exponential and logarithmic functions,
comparisons with fixed tie-breaking, ceilings, and finite mixtures.
Countably many possible finite grid and batch lengths give countably
many measurable branches. Free exact real arithmetic is understood in
the sense of the model; no numerical floating-point claim is needed.

\paragraph{Parsing the physical transcript.}
First reconstruct the two pilots from their deterministic rules,
including free ellipsoid updates and budget fallbacks. After their last
reply, the parser knows $m,z,\Shape$ and all public acceptance parameters.
The initial marker of a trial reveals $x$, and hence whether $x=m$.
In the zero-radius branch, the next call is the final acceptance marker.
Otherwise the proposal fixes $r,J$ and every grid batch size, so the
parser recognizes the end of the two grids without any private uniforms.

The next call is the control marker. Its location is $0e_1$ after a
failed prefactor, or $(N_c+1)e_1$ after a successful prefactor. Thus it
announces the category as well as whether the prefactor passed.
A failed prefactor, category $0$, or category $H+1$ has no correction
nodes and is followed immediately by the final marker. In a category
$1\le N_c\le H$, each node consists of exactly $n$ gradient calls
followed by one mark-bit marker. A marked node ends the node stage;
otherwise it continues until $N_c$ nodes have been completed. The next
call is then the final acceptance-bit marker. The parser can identify
all these positions and count trials from the transcript alone.

Every stop occurs after a final marker reply, when its acceptance bit
is one or the public trial cap has been reached. Thus $\{T=t\}$ is
measurable in the physical transcript after reply $t$. Spatial
coincidences between data points and marker locations cause no ambiguity:
the current stage and position specify the role of a call. No proposal,
category, or acceptance uniform needs to be adjoined to the physical
filtration. Ignored replies cannot change a bit already selected or the
recorded output. They may influence later oracle kernels; the conditional
proofs explicitly include them.

\paragraph{Trial counts and grid costs.}
Write $N_{\rm tr}$ for the number of attempted trials and set
\[
 L_\kappa=1+\log(1+\kappa),\qquad
 R_\eps=1+\log(1/\eps).
\]
Since $p_0$ depends only on $d$, the public choices imply
\begin{equation}\label{up:cost-parameter-bounds}
 n+1\le C_d(1+X),\qquad H\le C_dR_\eps,
 \qquad K\le C_dR_\eps,
 \qquad A/\delta=C_dX.
\end{equation}
Here the final equality denotes the explicit positive dimension-dependent
factor $32/p_0$ multiplying $X$ and remains valid when $A=X=0$.
There are at most three non-node markers per trial: initial, control,
and final. Every correction node costs $n+1$ calls including its marker.
Consequently, on every history, a reached trial has the bound
\begin{equation}\label{fx:trial-cost}
 C_{\rm trial}\le N_{\rm grid}+3+(n+1)H,
\end{equation}
where $N_{\rm grid}=0$ if $x=m$.
We do not use this crude bound on all good trials; that would retain
an unnecessary factor $H$.

For $r>0$ the grid lemma gives
$N_{\rm grid}\le2J+2+56Ar^2/\delta$, and
\[
 J\le1+\frac{\log\kappa}{2\log2}
          +\frac{\log(1+r)}{\log2}.
\]
Set $J=0$ when $r=0$. Conditional on any history before a reached
proposal, the new proposal has the fixed pilot's density $q$.
Its second radius moment and logarithmic radius moment are bounded by
constants depending only on $d$, on every pilot history, including
fallback histories. Therefore
\begin{equation}\label{fx:mean-trial-cost}
 \E[N_{\rm grid}\mid\text{incoming trial history}]
 \le C_d(L_\kappa+X)=C_dS.
\end{equation}

On any good completed pilot history,
\eqref{up:actual-acceptance-mass} gives conditional acceptance
probability at least $p_0/2$ in every reached trial. Repeated conditioning
implies $\E[N_{\rm tr}\mid\text{good pilot history}]\le2/p_0$.
On a bad pilot use only $N_{\rm tr}\le K$. It follows that
\begin{equation}\label{up:expected-trial-count}
 \E[N_{\rm tr}]\le\frac2{p_0}+\P(\mathcal G^c)K
 \le\frac2{p_0}+\frac\eps4K\le C_d.
\end{equation}
The event of reaching a trial is determined before its new proposal.
Multiply \eqref{fx:mean-trial-cost} by this event's indicator, take
expectations, and sum over the trial index. The total expected grid
cost is at most $C_dS\E[N_{\rm tr}]\le C_dS$.
The total cost of non-node markers is at most $3\E[N_{\rm tr}]\le C_d$.

\paragraph{Correction nodes after a good pilot.}
Fix a good completed pilot and any history before a reached trial.
On a completed good envelope, condition on a passed prefactor,
any category $N_c=j\le H$, and its control marker reply.
By \eqref{up:relative-mark-error} and
\eqref{up:mark-relative-budget}, every reached node has conditional
mark probability at least $1/136$. Hence the expected number of
processed nodes is at most
\[
 \sum_{l\ge0}(1-1/136)^l=136.
\]
This bound holds uniformly in $j$; category $H+1$ uses no nodes.
It also holds when the prefactor is not conditioned to pass, since a
failed prefactor has no nodes. On failed envelope histories use the
bound $H$ instead. Their conditional probability, before the envelope
construction, is at most $\delta$, independently of any estimate of
its realized heights. Averaging gives the conditional node-count bound
\begin{equation}\label{up:good-pilot-node-count}
 \E[N_{\rm nodes}\mid\text{incoming good-pilot trial history}]
 \le136+\delta H\le C_d.
\end{equation}
The same bound covers the zero-radius branch. The inequality
$\delta H\le C_d$ follows from boundedness of
$\tau\log(4/\tau)$ over the stipulated range of $\tau$.
Multiply by $n+1$ and sum over reached trials as above. Conditional
on any good pilot the total expected node cost is at most
$C_d(n+1)\le C_d(1+X)$.

\paragraph{Correction nodes after a bad pilot.}
On a bad pilot there need not be a positive lower bound on the mark
probability. Both public caps must be retained in this part of the
argument. If $T_{\rm nodes}$ counts all correction calls and their
node markers, then
\begin{align}
 \E[\mathbf1_{\mathcal G^c}T_{\rm nodes}]
 &\le\P(\mathcal G^c)KH(n+1)\notag\\
 &\le C_d\eps R_\eps^2(1+X)
 \le C_d(1+X).\label{up:bad-pilot-node-cost}
\end{align}
The last step uses the uniform boundedness of
$\eps[1+\log(1/\eps)]^2$ on $(0,1/10]$.
The event $\mathcal G^c$ is already measurable at the end of the pilot;
this estimate makes no independence assumption about later replies.
In particular, it does not treat bad-pilot trials as if they enjoyed
the good-pilot mark bound.

Adding the grid and marker bounds, the two node-cost contributions,
and the pathwise pilot bound \eqref{up:pilot-pathwise-cost} proves
$\E[T]\le C_dS$. All grid ceilings, node means, ignored replies,
zero-radius branches, and failure histories have been charged.
No bound on the expectation of a noisy envelope height was used.
Normalization preserves both call count and TV, completing the proof.
\end{proof}

The complete construction now proves the main upper-bound theorem;
its two endpoint consequences are treated separately below.

\begin{proof}[Completion of the proof of Theorem~\ref{thm:upper}]
We now specify the parameters left symbolic in
Algorithms~\ref{alg:full-sampler} and~\ref{alg:one-trial}.
Use $B=3/64$, $b_{\rm acc}=10$, and the dimension-dependent constants
\begin{equation}\label{alg:dimension-parameters}
 \gamma=\frac1{8d},\qquad
 \rho=\frac{20\sqrt d}{\gamma},\qquad
 D_d=1+d\sum_{j=0}^{d-1}\binom{d-1}{j}4^{j+1}j!,\qquad
 p_0=\frac{e^{-11}}{\rho^dD_d}.
\end{equation}
Choose the failure budgets, correction batch size, and caps as
\begin{equation}\label{alg:public-parameters}
 \begin{gathered}
 \delta_m=\delta_q=\eps/8,\qquad c=1/4,\qquad
 \tau=p_0\eps/16,\qquad \delta=\tau/2,\\
 n=\max\{1,\lceil256A/\tau\rceil\},\qquad
 H=\lceil136\log(4/\tau)\rceil,\\
 N_{\rm overflow}=H+1,\qquad
 K=\left\lceil\frac2{p_0}\log\frac{16}{\eps}\right\rceil.
 \end{gathered}
\end{equation}
These agree with \eqref{up:sampling-parameters}; $\tau$ and $p_0$ are
calibration quantities, not additional unknown inputs.
Implement the two pilots by Propositions~\ref{fx:center} and~\ref{fx:geometry},
including their caps, fallbacks, and tie-breaking. The grid sizes,
padding, and batches are \eqref{up:grid-definition}--\eqref{up:estimated-envelope}.
Use the finite radial mixture in Proposition~\ref{fx:proposal} and the
rectangle mixture following \eqref{up:integration-density}.
For each fixed input $Q$, $\mathsf{FiniteCategory}$ uses the following weights
\begin{equation}\label{alg:category-weights}
 \P(N_c=j)=e^{-Q}\frac{Q^j}{j!}\quad(0\le j\le H),\qquad
 \P(N_c=N_{\rm overflow})=1-\sum_{j=0}^He^{-Q}\frac{Q^j}{j!}.
\end{equation}
Thus its output has precisely the law \eqref{up:finite-poisson-category}
without simulating an uncapped Poisson count. All choices depend only on
public parameters and completed replies. With $b_{\rm acc}=10$, the
zero-radius probability, the prefactor, and the correction marks in the
main-body pseudocode coincide with the ones analyzed above.
The algorithms therefore implement exactly that capped sampler,
including every marker call and its timing. Proposition~\ref{prop:upper} proves its
accuracy, physical stopping, termination, and unconditional expected
cost. It is therefore admissible in \eqref{eq:minimax}, and taking
the infimum over admissible algorithms gives the asserted upper bound.
\end{proof}

The finite-accuracy construction is now complete. At zero oracle
variance the marking bias disappears, so removing the two rejection
caps gives a separate exact algorithm. Its proof must establish
almost-sure termination and expected cost anew.

\paragraph{Proof sketch and intuition for Corollary~\ref{up:noiseless-exact}.}
At zero variance every queried gradient is exact almost surely.
The pilots and directional envelopes then succeed without a statistical
failure event; each batch has size one. Remove the outer trial cap and
Poisson overflow cap and use the exact marking probability $b_0(t)/a_0(t)$ in
\eqref{up:acceptance-rule}. The Poisson acceptance identity
\eqref{main:poisson-identity} then yields exact acceptance.
Every accepted sample has law $\nu_f$. Acceptance mass is at least $p_0$,
so the trial count is geometrically bounded, while early rejection keeps
the expected number of processed nodes constant even for large Poisson
intensity. Expected grid and pilot costs are $O(\log(1+\kappa))$.
This is a separate construction, not substitution of $\eps=0$ into an
approximation theorem. Almost-sure termination and bounded expected cost
do not assert an accuracy-independent deterministic query budget.

\paragraph{Proof of Corollary~\ref{up:noiseless-exact}.}
\begin{proof}
Run the modified center and rounding procedures with fixed budgets
$\delta_m=\delta_q=1/8$. Since the conditional oracle variance is zero,
every reply equals the exact gradient almost surely. This statement
holds simultaneously for all calls that are made: there are only
countably many possible physical call indices, so the union of their
null error events is null. Every reached pilot estimation is therefore
accurate, its good cost promise holds, and its cap cannot preempt
successful completion. Both pilots are correct almost surely, and their
combined call bound is $C_d\log(1+\kappa)$. Their thresholds have
lower bound one and their noise-dependent parts are zero.

Use the same proposal and both envelopes, now with the fixed budget
$\delta=1/2$. Every grid batch has size one and every signed estimate
is exact almost surely; the nonnegative padding remains in the
construction. The conclusions of
Lemmas~\ref{up:noisy-envelope}, \ref{up:proposal-ray-derivative},
and~\ref{up:mark-geometry} consequently hold almost surely in every
reached trial. Use one exact gradient call at each correction node,
so the mark probability in \eqref{up:acceptance-rule} is exactly
$b_0(t)/a_0(t)$. It already belongs to $[0,1]$, by
\eqref{up:mark-margins}.

Replace the finite category by a genuine Poisson variable
$N\sim\operatorname{Pois}(Q)$, and omit the overflow branch.
For definiteness, generate $N$ by inverse CDF from a fresh uniform
on $(0,1)$; for finite $Q$ this returns a finite nonnegative integer
for every such uniform. Following a successful prefactor, process
at most $N$ nodes, with immediate rejection after the first mark.
If all $N$ nodes are unmarked, accept. Category $N=0$ accepts
without any correction nodes. Remove the outer trial cap and keep
trying until an acceptance occurs. All clipped-probability and
zero-radius conventions are retained on exceptional histories.

Conditional on every exact completed grid, the fresh positions and
marking uniforms give
\[
 \P(\text{accept}\mid x,\text{completed grids})
 =v\E[(1-p)^N]=v e^{-\ell}=\alpha_*(x).
\]
The expression is independent of the realized envelope mass because
of \eqref{up:poisson-identity}. Conditional on each incoming trial
history, the accepted proposal therefore has exactly the target law
by \eqref{up:ideal-acceptance-mass}, and the trial acceptance mass
is $p_*\ge p_0$. Repeated conditioning gives
$\P(N_{\rm tr}>k)\le(1-p_0)^k$ and
$\E[N_{\rm tr}]\le1/p_0$. Thus an acceptance occurs almost surely,
and disintegrating over its index and preceding history proves exactness
of the returned sample.

For each fixed $N$, the expected number of processed nodes is at most
$\sum_{l\ge0}(1-p)^l=1/p\le68$. This is uniform in $N$, so
averaging the Poisson count preserves the bound; no bound on
$\E[Q]$ is needed. Each node has one data call and one mark-bit
marker. The grid count is at most $2J+2$. The proposal's uniform
logarithmic moment gives conditional expected grid cost at most
$C_d\log(1+\kappa)$ in every reached trial. The constantly many
other markers and the bounded node cost have the same upper order.
Summing over reached trials using their geometric tail, and adding
the pilots, gives the claimed uniform expected cost.

The physical implementation is the earlier parser with the genuine
integer $N$ announced by the control marker at $(N+1)e_1$ after a
passed prefactor, and $0e_1$ after a failed one. The parser processes
exactly one known-size node block at a time, ending the node stage
after a marked node or after $N$ unmarked nodes. A final acceptance
marker is always charged before leaving the trial. The zero-radius
branch still has its initial and final markers. Stop only after a
successful final marker. Thus the query count remains a stopping time
for the physical filtration. Every Poisson count and each individual
trial is finite almost surely, and the outer geometric tail proves
almost-sure termination of the whole procedure.

An infinite sequence of rejected trials can be specified as an
exceptional internal-randomness path. No claim of termination on that
path is made for this uncapped exact algorithm. For every positive
$\eps$, the capped algorithm of Proposition~\ref{prop:upper} instead
retains termination on every specified real-valued path and has the
same expected asymptotic cost when $A=0$.
\end{proof}

The general proof is not needed at $\kappa=1$: the target is a Gaussian
with an unknown shift. Directly averaging that shift gives the sharper
query bound in Corollary~\ref{lo:quadratic-case}.

\paragraph{Proof sketch and intuition for Corollary~\ref{lo:quadratic-case}.}
In normalized coordinates $F(x)=F(0)+\langle b,x\rangle+\norm{x}^2/2$.
Average $n=\max\{1,\lceil A/(2\eps)\rceil\}$ replies at zero and return
$Z-\widehat b$, with an independent $Z\sim N(0,I_d)$.
The mean error has zero expectation and second moment at most $A/n$.
In a Taylor expansion of the Gaussian density, the first-order shift
cancels after averaging. The remaining TV error is at most $A/(2n)$,
not merely proportional to the standard deviation. This explains the
inverse-linear accuracy order already for quadratic potentials.
The full proof below gives the density calculation.
The same construction works without any localization assumption when
$L=\mu$.

\paragraph{Proof of Corollary~\ref{lo:quadratic-case}.}
\begin{proof}
When $\kappa=1$, the two Bregman inequalities coincide. Setting their
second argument to zero gives
\[
 F(x)=F(0)+\langle b,x\rangle+\tfrac12\norm{x}^2,
 \qquad b=\nabla F(0),\quad\norm{b}\le1.
\]
The target is therefore $N(-b,I_d)$. Query zero a deterministic number
$n=\max\{1,\lceil A/(2\eps)\rceil\}$ of times and let $\widehat b$
be the mean reply. Lemma~\ref{up:mean-lemma} gives
\[
 \E[\widehat b]=b,\qquad
 \E[\norm{\widehat b-b}^2]\le A/n.
\]
Draw an independent $Z\sim N(0,I_d)$ using private randomness and
return $Z-\widehat b$. The stopping count is the deterministic number
$n$, so it is a physical stopping time.

To bound the error, write $e=\widehat b-b$ and let $\phi_d$ be the
standard Gaussian density. For fixed $e$, Taylor's formula gives
\[
 \phi_d(y+e)-\phi_d(y)-\langle\nabla\phi_d(y),e\rangle
 =\int_0^1(1-t)e^\top\nabla^2\phi_d(y+te)e\,dt.
\]
For each $t$, rotation and translation invariance imply
\[
 \int_{\R^d}\bigl|e^\top\nabla^2\phi_d(y+te)e\bigr|\,dy
 =\norm{e}^2\E[|Z_1^2-1|]\le2\norm{e}^2.
\]
The integrated absolute Taylor remainder is at most $\norm{e}^2$.
Upon averaging, the linear term vanishes because $\E[e]=0$.
Tonelli's theorem and the finite second moment consequently give
\[
 \norm{\mathcal L(Z-e)-N(0,I_d)}_{\TV}
 \le\tfrac12\E[\norm{e}^2]\le\frac{A}{2n}\le\eps.
\]
Translation proves the required target guarantee. This argument uses
no independence assumption on successive replies; their conditional
mean and variance suffice. The lower bounds follow from Theorem~\ref{thm:lower},
proved independently in Section~\ref{sec:lower}.
Finally, if $A=0$, the single reply equals $b$ almost surely, so the
same construction is exact and the lower bound one is attained.
\end{proof}

\section{Proof of the Lower Bound}\label{sec:lower}
We separate the noise, curvature, and constant obstructions before
combining them at the same public parameter tuple. In this appendix,
$N^\star$ abbreviates the minimax quantity~\eqref{eq:minimax};
normalization preserves that quantity and gives the parameters
$\kappa,A,X,S$ defined in Section~\ref{sec:model}.
\subsection{Hard-family embedding for vector algorithms}\label{fx:embedding}
For each scalar potential $h$, use the completed potential and oracle
\[
 f_h(x)=h(x_1)+\tfrac12\norm{x_{2:d}}^2,\qquad
 \mathcal G_h(x)=(G_h(x_1),x_2,\ldots,x_d).
\]
The Bregman remainder is the scalar remainder plus
$\norm{x_{2:d}-y_{2:d}}^2/2$, so the curvature bounds and localization
hold in $\R^d$. The total vector variance equals the scalar variance.
The target is the scalar law times a fixed standard Gaussian, so its
first marginal is exactly the scalar law and projection contracts TV.

In the two proofs below, $u$ denotes the first coordinate of an arbitrary
vector query; $\nu_j,Q_j,Y_j$ denote first-coordinate target laws,
output laws, and outputs. The algorithm and its stopping count remain
fully vector-valued. All extra reply coordinates are deterministic
functions of the current query, identical across instances. Therefore
the rare-message coupling preserves the entire vector transcript until
its first informative reply. Fixing the complete private tape makes
the vector transcript tree have at most the same three replies per
node. The proofs consequently establish lower bounds on the full
quantity \eqref{eq:minimax}, including algorithms that query outside
the first coordinate axis. No assertion about preserving a projected
physical stopping filtration is needed.

\subsection{The stochastic-noise lower bound}\label{sec:noise-lower}

We work in the normalized model throughout this subsection. Thus the curvature
bounds are $1$ and $\kappa$, the first coordinate of the mode belongs to $[-1,1]$ (the others are zero), and the
advertised conditional variance ceiling is $A$. Recall that $A\ge0$,
$0<\eps\le1/10$, and $X=A/\eps$. All instances below have the same public
parameters. In particular, the algorithm receives no instance label.

\begin{proposition}[Full-variance stochastic-noise lower bound]
\label{prop:noise-lower}
For every public tuple in the stated regime,
\begin{equation}
N^\star\ge \frac{A}{24\eps}=\frac{X}{24}.
\label{lo:eq:noise-bound}
\end{equation}
For $A>0$, the witnesses attain the advertised curvature endpoints and use fresh,
conditionally unbiased gradient replies whose conditional variance is exactly
$A$. The bound allows arbitrary adaptive vector-valued queries and random
stopping with expected query cost.
\end{proposition}

\begin{proof}
If $A=0$, the asserted lower bound is zero and is immediate. Assume
$A>0$ for the rest of this proof; no lower bound on its positive value
will be used.
We first construct the potentials and separate their normalized target laws.
We then couple two valid stochastic oracles up to their first informative
reply, without imposing a deterministic query horizon.

\paragraph{Potentials and class membership.}
Set
\begin{equation}
\tau_0=4\eps,\qquad
H_{\mathrm a}(u)=\frac{u^2}{2}
 +\frac{\kappa-1}{2}(|u|-8)_+^2,
\qquad F_\pm(u)=H_{\mathrm a}(u)\pm\tau_0 u,
\label{lo:eq:noise-potentials}
\end{equation}
where $v_+=\max\{v,0\}$. The function
\[
r_0(u)=\operatorname{sgn}(u)(|u|-8)_+,
\qquad \operatorname{sgn}(0)=0,
\]
is continuous, nondecreasing, and one-Lipschitz. Indeed, it equals $u+8$
on the left tail, zero in the core, and $u-8$ on the right tail. Consequently
\[
H_{\mathrm a}'(u)=u+(\kappa-1)r_0(u),
\]
and, for $u>v$,
\begin{equation}
u-v\le H_{\mathrm a}'(u)-H_{\mathrm a}'(v)
\le\kappa(u-v).
\label{lo:eq:anchor-slopes}
\end{equation}
Integrating along the segment between two points, in either orientation,
gives the strong-convexity and smoothness inequalities with constants
$1,\kappa$. The linear tilts in \eqref{lo:eq:noise-potentials} leave their
Bregman remainders unchanged. Thus $F_\pm$ belong to the required $C^1$
class. On a nonempty interval in the core the curvature is exactly $1$,
and on an interval strictly outside the core it is exactly $\kappa$.
When $\kappa=1$, the two advertised endpoints coincide.

The unique minimizers of $F_-$ and $F_+$ are $\tau_0$ and $-\tau_0$,
respectively. They lie in the core, solve the derivative equations, and
obey $\tau_0=4\eps\le2/5<1$. Both instances therefore satisfy the
localization promise. Their gradient gap is independent of the query:
\begin{equation}
|F_+'(u)-F_-'(u)|=\Delta:=2\tau_0=8\eps
\qquad (u\in\R).
\label{lo:eq:constant-gap}
\end{equation}
In particular, arbitrarily remote queries cannot amplify this gap.

\paragraph{Separation after normalization.}
Let $\nu_\pm$ be the probability measures with densities proportional to
$e^{-F_\pm}$. Their normalizers are finite and positive because
$H_{\mathrm a}(u)\ge u^2/2$ and the remaining terms are linear.
Evenness of $H_{\mathrm a}$ makes the two normalizers equal to
\[
2\int_0^\infty e^{-H_{\mathrm a}(u)}\cosh(\tau_0 u)\,du.
\]
It follows, with the normalizers retained, that
\begin{equation}
\nu_+(( -\infty,0])-\nu_-(( -\infty,0])
=\frac{\displaystyle\int_0^\infty
 e^{-H_{\mathrm a}(u)}\sinh(\tau_0 u)\,du}
 {\displaystyle\int_0^\infty
 e^{-H_{\mathrm a}(u)}\cosh(\tau_0 u)\,du}.
\label{lo:eq:normalized-separation}
\end{equation}
The numerator is nonnegative. Restrict it to $[0,8]$ and use
$\sinh(\tau_0 u)\ge\tau_0 u$. Enlarge the denominator by removing the
nonnegative additional anchor term. This bounds the ratio below by
\begin{equation}
\frac{\displaystyle\tau_0\int_0^8u e^{-u^2/2}\,du}
 {\displaystyle\int_0^\infty e^{-u^2/2}\cosh(\tau_0 u)\,du}
=\tau_0(1-e^{-32})e^{-\tau_0^2/2}\sqrt{\frac2\pi}.
\label{lo:eq:separation-integral}
\end{equation}
For the denominator identity, expand the hyperbolic cosine, complete
the squares, and reflect one integral. The Gaussian integral itself
follows from Tonelli's theorem and polar coordinates:
\[
\left(\int_\R e^{-u^2/2}\,du\right)^2
=\int_0^{2\pi}\int_0^\infty e^{-r^2/2}r\,dr\,d\theta=2\pi.
\]

We record elementary numerical bounds to keep the constant explicit.
For $0\le u\le1$,
$(1-u^2+u^4)(1+u^2)=1+u^6\ge1$, so
\[
\pi=4\int_0^1\frac{du}{1+u^2}
\le4\int_0^1(1-u^2+u^4)\,du
=\frac{52}{15}<\frac{32}{9}.
\]
The first identity follows from $u=\tan\theta$ on $[0,\pi/4]$.
Therefore $\sqrt{2/\pi}>3/4$. Since $\tau_0\le2/5$,
\[
e^{-\tau_0^2/2}\ge1-\tau_0^2/2\ge\frac{23}{25},
\qquad 1-e^{-32}>\frac{99}{100}.
\]
The last inequality follows from $e>2$ and $2^{32}>100$.
Combining these estimates with \eqref{lo:eq:normalized-separation} and
\eqref{lo:eq:separation-integral} yields
\begin{equation}
\|\nu_+-\nu_-\|_{\TV}
>\frac{99}{100}\frac34\frac{23}{25}(4\eps)
=\frac{6831}{2500}\eps>\frac83\eps.
\label{lo:eq:target-separation}
\end{equation}
The last comparison is $20493>20000$. None of these estimates deteriorates
when $\kappa-1$ increases.

\paragraph{Fresh kernels using the entire variance budget.}
Define
\begin{equation}
p_{\mathrm r}=\frac{\Delta^2}{4A+\Delta^2}\in(0,1).
\label{lo:eq:rare-probability}
\end{equation}
On every physical query draw a new Bernoulli variable $\xi$ with success
probability $p_{\mathrm r}$, independently of the complete incoming history.
On instance $F_\pm$, return only
\begin{equation}
G_\pm(u)=H_{\mathrm a}'(u)\pm\frac{\tau_0}{p_{\mathrm r}}\xi.
\label{lo:eq:rare-oracle}
\end{equation}
The Bernoulli flag is not a separate observation. These are measurable
conditional kernels at every real query. For a query measurable in the
complete incoming history,
\[
\E[G_\pm(u)\mid\cH^{\rm in}]
=H_{\mathrm a}'(u)\pm\tau_0=F_\pm'(u),
\]
and their conditional variance is exactly
\begin{equation}
\tau_0^2\frac{1-p_{\mathrm r}}{p_{\mathrm r}}
=\frac{\Delta^2}{4}\frac{4A}{\Delta^2}=A.
\label{lo:eq:full-variance}
\end{equation}
Thus these kernels satisfy the stipulated oracle contract, including at
adaptive and unbounded query points.

\paragraph{Coupling with arbitrary expected stopping.}
Fix any algorithm uniformly $\eps$-accurate over the entire vector class and all
valid oracles. If its expected count is infinite on either constructed
instance, the claimed lower bound already holds. Otherwise couple the
two runs with the same complete private random tape, including randomness
used for the output, and the same Bernoulli sequence
$\xi_1,\xi_2,\ldots$, independent of that tape.

Until the first success on a queried round, both replies equal
$H_{\mathrm a}'(u)$. Inductively, their queries, physical transcripts,
and stopping decisions agree throughout this common execution. If they
stop before any such success, their outputs agree as well. Let
$E_{\rm split}$ be the event of a success during the common queried
execution. Its first success is a queried call in both marginals.

For marginal $i\in\{-,+\}$ let $T_i$ be its physical call count.
The event $\{T_i\ge j\}$ is determined before the fresh $j$th innovation,
also after adjoining the independent complete private tape for this
coupling calculation. Nonnegative summation and conditional expectation
give
\begin{align}
\E_i\biggl[\sum_{j=1}^{T_i}\xi_j\biggr]
&=\sum_{j\ge1}\E_i\bigl[\mathbf1_{\{T_i\ge j\}}\xi_j\bigr]
\notag\\
&=p_{\mathrm r}\sum_{j\ge1}\P_i(T_i\ge j)
=p_{\mathrm r}\E_i[T_i].
\label{lo:eq:stopped-message-count}
\end{align}
This calculation enlarges the information used in the proof, not the
physical stopping information available to the algorithm.

Let $Q_\pm$ be the output laws and $Y_\pm$ the coupled outputs. On
$E_{\rm split}$ the stopped sum in \eqref{lo:eq:stopped-message-count}
is at least one in each marginal. The coupling inequality therefore gives,
for either $i$,
\[
\|Q_+-Q_-\|_{\TV}
\le\P(Y_+\ne Y_-)
\le\P(E_{\rm split})
\le p_{\mathrm r}\E_i[T_i].
\]
Uniform accuracy and \eqref{lo:eq:target-separation} imply
\[
p_{\mathrm r}\E_i[T_i]
\ge\|\nu_+-\nu_-\|_{\TV}-2\eps>\frac23\eps.
\]
Consequently the algorithm's worst-case expected cost is at least
\[
\frac{2\eps}{3p_{\mathrm r}}
=\frac23\eps+\frac{8A\eps}{3\Delta^2}
=\frac23\eps+\frac{A}{24\eps}
\ge\frac{X}{24}.
\]
Taking the infimum over all full-class successful algorithms proves
\eqref{lo:eq:noise-bound}. The argument includes zero-call branches,
random output kernels, and rare arbitrarily long executions.
\end{proof}

\subsection{The scalar curvature lower bound}\label{lo:sec:curvature-lower}

The second obstruction is present even with exact gradients. Such a
kernel is admissible under every advertised variance ceiling, so the
construction applies for all $A\geq0$. Its family encodes an unknown
mode inside the prescribed localization ball. The finite-branch
argument controls gradient replies, not joint value-and-gradient
replies, and is not asserted to survive revelation of the exact mode.
The result proves the curvature dependence in our stated localized
model. Appendix~\ref{app:localization-lower} gives a different family
that varies the location scale while keeping $\mu=1$ and $L=2$;
that argument isolates why unrestricted initialization cannot have a
uniformly finite bound in general.

\begin{proposition}[Scalar curvature lower bound]
\label{prop:curvature-lower}
For $\kappa\ge640^2$, put
\begin{equation}
M=\left\lfloor\frac{\sqrt\kappa}{80}\right\rfloor,
\qquad
B_\kappa=\frac14\left[\log_3\!\left(\frac{16M}{25}\right)-1\right].
\label{lo:eq:curvature-constant}
\end{equation}
Then $N^\star\ge B_\kappa$. The first coordinates of the witness modes
belong to $(-1/4,1/4)$, and all other coordinates are zero. The witnesses
attain the curvature endpoints $1,\kappa$ and use exact gradients.
The bound allows arbitrary adaptive vector queries and expected stopping.
\end{proposition}

\begin{proof}
We construct the family, prove concentration on disjoint intervals, and
then apply a finite-branch decoding argument.

\paragraph{A localized threshold-ramp family.}
Let
\begin{equation}
s_0=\sqrt\kappa,\qquad
w=\frac{s_0}{s_0^2-1},\qquad
\theta_j=-\frac14+\frac{40(j-1)}{s_0}
\quad(1\le j\le M).
\label{lo:eq:ramp-parameters}
\end{equation}
Then $s_0\ge640$, $M\ge8$, $w<2/s_0$, and
\begin{equation}
-\frac14\le\theta_j\le\frac14-\frac{40}{s_0}<\frac14.
\label{lo:eq:threshold-range}
\end{equation}
Define the continuous derivative
\begin{equation}
p_j(u)=
\begin{cases}
u-s_0/2,&u\le\theta_j,\\
u-s_0/2+(s_0^2-1)(u-\theta_j),&\theta_j\le u\le\theta_j+w,\\
u+s_0/2,&u\ge\theta_j+w.
\end{cases}
\label{lo:eq:ramp-derivative}
\end{equation}
The formulas agree at both joins, since $(s_0^2-1)w=s_0$. Let $F_j$ be any
primitive of $p_j$. The function $p_j$ is absolutely continuous, with
derivative $1$ off the ramp and $s_0^2$ in its interior almost everywhere.
For $u>v$ this gives
\[
u-v\le p_j(u)-p_j(v)\le s_0^2(u-v).
\]
Integrating once more, with the orientation reversed when necessary,
proves
\begin{equation}
\frac12(u-v)^2
\le F_j(u)-F_j(v)-p_j(v)(u-v)
\le\frac{s_0^2}{2}(u-v)^2.
\label{lo:eq:ramp-class}
\end{equation}
Hence the functions are $C^1$, $1$-strongly convex and $\kappa$-smooth.
Pairs of distinct points in an outer interval attain the first curvature
endpoint, while pairs in the nonempty ramp interior attain the second.

Let $a_j=s_0/2-\theta_j$. The left endpoint derivative is negative and
the right endpoint derivative is positive. The unique root is consequently
inside the ramp and equals
\begin{equation}
z_j=\theta_j+\frac{a_j}{s_0^2}.
\label{lo:eq:ramp-minimizer}
\end{equation}
On the ramp, $p_j(u)=s_0^2(u-z_j)$. By
\eqref{lo:eq:threshold-range} and $w<2/s_0$,
\[
-\frac14<z_j<\frac14-\frac{40}{s_0}+\frac2{s_0}<\frac14.
\]
The localization promise holds. Moreover,
$F_j(u)\ge F_j(z_j)+(u-z_j)^2/2$, so each target normalizer is finite
and positive.

\paragraph{Disjoint events carrying almost all target mass.}
Define
\begin{equation}
D_j=\left[\theta_j-\frac{16}{s_0},\,
                 \theta_j+w+\frac{16}{s_0}\right].
\label{lo:eq:decoding-events}
\end{equation}
For adjacent indices,
\[
\inf D_{j+1}=\theta_j+\frac{24}{s_0},
\qquad \sup D_j=\theta_j+w+\frac{16}{s_0}.
\]
Since $w<2/s_0<8/s_0$, these closed intervals are pairwise disjoint with
strict gaps.

Write
\[
V_j(u)=F_j(u)-F_j(z_j),\qquad
b_j=s_0/2+\theta_j+w,\qquad c=s_0/2-1/4.
\]
Both $a_j$ and $b_j$ are at least $c$. On the ramp,
\[
V_j(u)=\frac{s_0^2}{2}(u-z_j)^2.
\]
The distances from $z_j$ to the left and right ramp endpoints are
$a_j/s_0^2$ and $b_j/s_0^2$, respectively. Since
$c/s_0^2>1/(4s_0)$, the interval $[z_j-1/(4s_0),z_j+1/(4s_0)]$ lies in the
ramp. Therefore the normalizer satisfies
\begin{equation}
Z_j:=\int_\R e^{-F_j(u)}\,du
\ge e^{-F_j(z_j)}\frac1{2s_0}e^{-1/32}.
\label{lo:eq:ramp-normalizer}
\end{equation}

Direct integration of the outer affine derivatives gives, for $t\ge0$,
\begin{align*}
V_j(\theta_j-t)&=\frac{a_j^2}{2s_0^2}+a_jt+\frac{t^2}{2},\\
V_j(\theta_j+w+t)&=\frac{b_j^2}{2s_0^2}+b_jt+\frac{t^2}{2}.
\end{align*}
Dropping the nonnegative square terms and integrating for $t\ge16/s_0$
bounds the unnormalized mass outside $D_j$ by
\[
e^{-F_j(z_j)}
\left(\frac{e^{-16a_j/s_0}}{a_j}
     +\frac{e^{-16b_j/s_0}}{b_j}\right)
\le e^{-F_j(z_j)}\frac{2e^{-16c/s_0}}c.
\]
Dividing by \eqref{lo:eq:ramp-normalizer}, the normalized tail probability
is at most
\begin{equation}
\nu_j(D_j^c)\le\frac{4s_0}c
       \exp\!\left(\frac1{32}-\frac{16c}{s_0}\right),
\label{lo:eq:ramp-tail}
\end{equation}
where $\nu_j$ denotes the target corresponding to $F_j$.
Here
\[
\frac{4s_0}c=\frac{16s_0}{2s_0-1}<9,
\qquad
\frac1{32}-\frac{16c}{s_0}
=-8+\frac1{32}+\frac4{s_0}
\le-\frac{637}{80}<-7.
\]
Since $e>1+1+1/2+1/6=8/3$ and $(8/3)^7>900$, the right side of
\eqref{lo:eq:ramp-tail} is smaller than $9e^{-7}<1/100$. We have thus
proved the normalized probability statement
\begin{equation}
\nu_j(D_j)>\frac{99}{100}
\qquad(1\le j\le M).
\label{lo:eq:ramp-mass}
\end{equation}

\paragraph{At most three replies per arbitrary query.}
For any fixed real query $u$, all indices satisfying $u\le\theta_j$
return the same derivative $u-s_0/2$. All indices satisfying
$u\ge\theta_j+w$ return the same derivative $u+s_0/2$. At most one
index has $\theta_j<u<\theta_j+w$, because the ramps are disjoint.
Such an index supplies at most one additional value, strictly between
the two outer values. At either ramp endpoint its value equals the
corresponding outer value by continuity. Thus every real query, including
an endpoint or an arbitrarily remote point, has at most three exact
replies across this family.

Fixing the complete private random tape of an adaptive algorithm makes
its query a deterministic function of prior replies. The query itself
does not create additional branches. Its exact-gradient transcript tree
therefore has branching factor at most three.

\paragraph{TV decoding with expected query cost.}
Consider any algorithm uniformly $\eps$-accurate on the full class and
all valid oracles. Let $Q_j$ be its output law on the exact-gradient
instance $F_j$. By \eqref{lo:eq:ramp-mass},
\[
Q_j(D_j)>\frac{99}{100}-\frac1{10}=\frac{89}{100}.
\]
Define a decoder $\widehat J(y)=j$ when $y\in D_j$, and assign a failure
symbol outside their union. Let the true index $J_0$ be uniform on
$\{1,\ldots,M\}$, independently of the algorithm's private tape. Then
\begin{equation}
\P(\widehat J(Y)=J_0)>\frac{89}{100}.
\label{lo:eq:decoding-success}
\end{equation}
Write $T_j$ for the physical query count on instance $j$ and set
\[
\bar n=\frac1M\sum_{j=1}^M\E[T_j].
\]
If $\bar n=\infty$, the desired lower bound already holds. If
$\bar n=0$, there are almost surely no queries on each member of this
finite family. The common depth-zero output can correctly decode at
most one index for any fixed private tape, giving success probability
at most $1/M\le1/8$, contrary to \eqref{lo:eq:decoding-success}.
It remains to consider $0<\bar n<\infty$.

Let $T$ be the query count in the uniform experiment and take
$H_{\rm tree}=\lceil4\bar n\rceil$. Markov's inequality gives
\begin{equation}
\P(T>H_{\rm tree})\le\frac{\bar n}{H_{\rm tree}}\le\frac14.
\label{lo:eq:tree-truncation}
\end{equation}
Fix the entire private tape, including the randomness used for the output.
Truncate the resulting deterministic transcript tree at depth
$H_{\rm tree}$. Pad an earlier stopping leaf with one dummy edge per
remaining level, retaining its original output. Each active node has
at most three children and each padded node has one, so at most
$3^{H_{\rm tree}}$ leaves occur at the common depth.
A leaf corresponding to stopping by the horizon has a single output
and hence a single decoder label. Among the indices reaching that leaf,
at most one is decoded correctly. Averaging over the private tape gives
\begin{equation}
\P(\widehat J(Y)=J_0,\ T\le H_{\rm tree})
\le\frac{3^{H_{\rm tree}}}{M}.
\label{lo:eq:leaf-count}
\end{equation}
On the other hand, \eqref{lo:eq:decoding-success} and
\eqref{lo:eq:tree-truncation} make this probability greater than
$89/100-1/4=16/25$. Therefore
\[
H_{\rm tree}>\log_3\!\left(\frac{16M}{25}\right).
\]
Since $H_{\rm tree}<4\bar n+1$,
\[
\bar n>\frac14\left[\log_3\!\left(\frac{16M}{25}\right)-1\right]
=B_\kappa.
\]
The worst-case expected cost on the full class is at least this prior
average. Infimizing over all admissible uniformly accurate algorithms
proves $N^\star\ge B_\kappa$. The exact-gradient kernel used here is
valid at the same advertised variance ceiling as every other instance
in the minimax definition.
\end{proof}

\subsection{The constant lower bound and completion}\label{sec:completion}
The ramp family yields useful information only for large curvature,
whereas the noise lower bound vanishes at $A=0$. To cover the remaining
parameter range, we use two separated Gaussian targets and the stated
physical stopping convention.

\begin{proposition}[Constant lower bound under physical stopping]
\label{prop:one-call}
For every public tuple in the stated regime, including $A=0$,
\begin{equation}\label{lo:eq:one-call}
 N^\star\ge1.
\end{equation}
\end{proposition}
\begin{proof}
The event $\{T=0\}$ belongs to the trivial physical sigma-field
$\mathcal F_0$ in \eqref{eq:physical}, so it has probability zero or
one. The decision whether to make a first query uses no
instance-dependent information. On a common private random tape,
therefore, this zero-call decision and any resulting output are the
same for all instances sharing the public parameters.

In normalized coordinates consider
\[
 F_\pm(x)=\tfrac12\norm{x\mp e_1}^2,
\]
where $e_1$ is the first coordinate vector. Both potentials satisfy
the curvature bounds for every $\kappa\ge1$, and their modes obey
the localization bound. Use their exact gradients, which are admissible for every
$A\ge0$. Their targets are $N(e_1,I_d)$ and $N(-e_1,I_d)$. Their
probability difference on the half-space $\{x:x_1\ge0\}$ equals
\[
 2\Phi(1)-1
 =\sqrt{\frac2\pi}\int_0^1e^{-t^2/2}\,dt
 >\frac34\cdot\frac12=\frac38>2\eps,
\]
where $\Phi$ is the standard normal distribution function. The two
elementary bounds used here are $\sqrt{2/\pi}>3/4$ and
$e^{-1/2}>1/2$.

A common zero-query output law cannot have TV distance at most $\eps$
from both targets, by the triangle inequality. Uniform accuracy
therefore rules out the probability-one zero-call decision. It follows
that $T\ge1$ almost surely for every admissible uniformly successful
algorithm, so its worst-case expected query count is at least one.
Taking the infimum proves \eqref{lo:eq:one-call}.
\end{proof}

All three obstructions are now available. Their witnesses need not
coincide: the minimax cost dominates each witness family separately.
It remains to compare their maximum with the common benchmark $S$.

\begin{proof}[Proof of Theorem~\ref{thm:lower}]

If $N^\star=\infty$, the lower bound is immediate. Otherwise, combine
the three lower bounds at the same public tuple. No assertion that
their costs add on a single instance is needed.

By Propositions~\ref{prop:noise-lower} and~\ref{prop:one-call},
\begin{equation}
X\le24N^\star,
\qquad N^\star\ge1.
\label{lo:eq:completion-baseline}
\end{equation}
These statements also hold when $A=X=0$.
If $1\le\kappa<640^2$, then
$1+\kappa<409601<2^{21}$ and $\log2<1$, so
\[
S<X+22\le24N^\star+22N^\star=46N^\star.
\]
This branch includes $\kappa=1$ and does not invoke the ramp construction
outside its stated range.

If $\kappa\ge640^2$, the argument of the floor in
\eqref{lo:eq:curvature-constant} is at least eight. Hence
\[
M\ge\frac{\sqrt\kappa}{160}.
\]
Proposition~\ref{prop:curvature-lower} gives
\begin{equation}
N^\star\ge\frac{\log\kappa}{8\log3}
 -\frac{\log250}{4\log3}-\frac14.
\label{lo:eq:log-curvature-bound}
\end{equation}
Since $\log3>0$, this can be rearranged. Using $\log3<2$,
$\log250<8$, and $N^\star\ge0$, we obtain
\[
\log\kappa
\le8\log3\,N^\star+2\log250+2\log3
<16N^\star+20.
\]
Also $1+\kappa\le2\kappa$. It follows that
\[
S<16N^\star+22+X
\le40N^\star+22
\le62N^\star,
\]
where the last step uses $N^\star\ge1$ from
\eqref{lo:eq:completion-baseline}. Both curvature ranges therefore yield
$N^\star\ge S/62$.

Together with the stochastic-noise and constant lower bounds, this proves
\begin{equation}
N^\star\ge\max\!\left\{1,\frac{X}{24},\frac{S}{62}\right\}.
\label{lo:eq:completed-minimax}
\end{equation}
This proves Theorem~\ref{thm:lower}. Together with
Theorem~\ref{thm:upper}, it gives the joint minimax rate in
Section~\ref{sec:results}. The ratio of the displayed upper
benchmark $C_dS$ to the lower benchmark $S/62$ is at most
$62C_d$, independently of $L,\mu,\sigma,\eps$ for each fixed $d$. Restoring
$A=\sigma^2/\mu$ and $\kappa=L/\mu$ through the query-preserving
normalization yields the asserted constant-factor characterization in
the original physical coordinates, throughout $\sigma^2\ge0$.
\end{proof}

\section{Proofs of the initialization extensions}
\label{app:initialization}

The main theorem uses $\mathcal C_d(L,\mu)$ as defined in Section~\ref{sec:model}.
We now derive extensions without changing that theorem or its pilot
algorithms. All query bounds below count physical replies, including
markers, and retain the physical-transcript stopping convention.

\paragraph{Proof sketch and intuition for Corollary~\ref{cor:known-radius}.}
Translate by $x_0$ and advertise the conservative strong-convexity bound
$\mu_R=\mu/\Gamma_R$. Then $R\leq\mu_R^{-1/2}$, so the translated
potential belongs to the class of the main theorem with condition
number $\kappa\Gamma_R$. Only the public curvature bound changes;
the target potential is not multiplied or tempered. Applying
Theorem~\ref{thm:upper} or Corollary~\ref{up:noiseless-exact} proves the
claims. The factor $\Gamma_R$ in the noisy term is a sufficient cost of
this reduction, not a claimed matching radius-noise lower bound.

\paragraph{Proof of Corollary~\ref{cor:known-radius}.}

\begin{proof}
Put $\Gamma_R=\max\{1,\mu R^2\}$ and $\mu_R=\mu/\Gamma_R$. Let $h(u)=f(x_0+u)$. Translation preserves the curvature bounds,
and $\mu_R\leq\mu$, so $h$ is also $\mu_R$-strongly convex.
Its minimizer satisfies
\[
\|u_h^\star\|\leq R\leq\mu_R^{-1/2}.
\]
Thus $h\in\mathcal C_d(L,\mu_R)$. Simulate a gradient query at $u$
by one physical query at $x_0+u$, returning the same vector.
The conditional variance ceiling remains $\sigma^2$. Apply
Theorem~\ref{thm:upper} with public parameters $(L,\mu_R,\sigma,\varepsilon)$
and translate its output back by $x_0$. Its stronger error bound
$7\varepsilon/16$ in particular implies $\varepsilon$-accuracy.
Its condition number and normalized variance are
$\kappa\Gamma_R$ and $\Gamma_R\sigma^2/\mu$, proving
\eqref{eq:known-radius-noisy}. The additive constant in its bound is
absorbed because $\kappa\Gamma_R\geq1$.

For $\sigma=0$, instead invoke Corollary~\ref{up:noiseless-exact} on $h$ with the same
public curvature bounds. This gives exact sampling and the first
bound in~\eqref{eq:known-radius-exact}. If $D_R=\sqrt{\mu}\,R$, then
$\Gamma_R\leq(1+D_R)^2$ and
\[
1+\kappa\Gamma_R\leq(1+\kappa)(1+D_R)^2,
\]
which proves the second bound. Translation and the sampler's internal
normalization are invertible transformations known from the inputs,
so they preserve the query count, the physical stopping rule, and TV.
Only a more conservative lower curvature bound is supplied to the
sampler; the potential values and target temperature are not rescaled.
\end{proof}

\paragraph{Scope of the noisy-radius bound.}
Equation~\eqref{eq:known-radius-noisy} is a direct upper-bound
consequence, not a claim of optimal joint dependence on $R$ and
$\sigma$. In particular, no matching lower bound for its factor
$\Gamma_R\sigma^2/(\mu\varepsilon)$ is asserted.

\paragraph{Proof sketch and intuition for Corollary~\ref{cor:no-radius-exact}.}
A single exact reply $g_0=\nabla f(x_0)$ certifies
$\norm{x_0-x_f^\star}\leq\norm{g_0}/\mu$ by strong monotonicity.
Use this radius in Corollary~\ref{cor:known-radius}. Smoothness gives
$\norm{g_0}\leq L\norm{x_0-x_f^\star}$ and converts the resulting cost
$1+C_d\log(1+\kappa\max\{1,\norm{g_0}^2/\mu\})$ into
\eqref{eq:unlocalized-exact-cost}. The exact first reply thus supplies its own
initialization certificate. This certificate is not valid for a single
noisy reply. The full proof below includes the initialization cost and physical stopping.

\paragraph{Proof of Corollary~\ref{cor:no-radius-exact}.}

\begin{proof}
Put $g_0=\nabla f(x_0)$ and
$\widehat\Gamma=\max\{1,\norm{g_0}^2/\mu\}$.
Make one physical gradient query at $x_0$. Its reply equals $g_0$
almost surely. Strong monotonicity of the gradient and
$\nabla f(x_f^\star)=0$ imply
\[
\mu\|x_0-x_f^\star\|^2
\leq\langle g_0,x_0-x_f^\star\rangle
\leq\|g_0\|\,\|x_0-x_f^\star\|.
\]
Hence
\[
\|x_0-x_f^\star\|\leq\widehat R:=\frac{\|g_0\|}{\mu}.
\]
This also holds when $g_0=0$, in which case $x_0=x_f^\star$.
Compute $\widehat\mu=\mu/\widehat\Gamma$ and apply the exact
sampler of Corollary~\ref{cor:known-radius} with center $x_0$ and
radius $\widehat R$. Its required input promise is certified by the
first reply. The additional expected cost is at most
$C_d\log(1+\kappa\widehat\Gamma)$. Including the initial reply gives
\begin{equation}\label{eq:gradient-certified-cost}
 \E_f[T]\leq 1+C_d\log(1+\kappa\widehat\Gamma),
 \qquad \widehat\Gamma=\max\{1,\norm{g_0}^2/\mu\},
\end{equation}
and the returned sample is exact.

To express the bound through the actual distance, put
$D_f=\sqrt{\mu}\,\|x_0-x_f^\star\|$. Smoothness gives
\[
\|g_0\|\leq L\|x_0-x_f^\star\|,
\qquad \widehat\Gamma\leq\max\{1,\kappa^2D_f^2\}.
\]
Since $\kappa\geq1$,
\[
1+\kappa\widehat\Gamma
\leq(1+\kappa)^3(1+D_f)^2.
\]
Taking logarithms and absorbing the initial one-call cost proves
\eqref{eq:unlocalized-exact-cost}.

For a fixed $f$, the first exact reply fixes all subsequent parameters
almost surely. The new parameter choices are measurable functions
of that physical reply, and the translated, normalized query locations
can be reconstructed from it and the later transcript. Thus
composition with the first call preserves physical stopping and
charges every marker. Almost-sure termination follows from
Corollary~\ref{up:noiseless-exact}. No estimate of $D_f$ is supplied as an input, and no
potential-value query is used. A noisy reply would not provide the
deterministic certificate above; this corollary is stated only for
zero variance.
\end{proof}

\paragraph{The quadratic endpoint without localization.}
At $\kappa=1$, even the initialization term can be avoided.
For every $f\in\mathcal U_d(\mu,\mu)$,
\[
f(x)=\frac{\mu}{2}\|x\|^2+\langle b,x\rangle+c,
\qquad \nu_f=N(-b/\mu,\mu^{-1}I_d),
\]
where $b$ is unrestricted. The argument in Corollary~\ref{lo:quadratic-case} never uses
the bound on $b$ in its error estimate. In original coordinates,
query zero
\[
n=\max\left\{1,\left\lceil\frac{\sigma^2}{2\mu\varepsilon}
\right\rceil\right\}
\]
times, average the replies as $\widehat b$, and return
$-\widehat b/\mu+Z/\sqrt{\mu}$ for independent $Z\sim N(0,I_d)$.
Translation of the Taylor-remainder argument gives TV error at most
$\sigma^2/(2\mu n)$. At zero noise one query is sufficient for exact
sampling, and the constant lower bound follows from the localized
subclass. Thus the location lower bound in
Appendix~\ref{app:localization-lower}, proved at $\kappa=2$, is a
general obstruction, not a claim about the quadratic endpoint.

\section{A value-and-gradient benchmark for the noiseless comparison}
\label{app:value-gradient-benchmark}

This appendix records an elementary comparison argument, not an
additional gradient-only algorithm. We work with a $1$-strongly convex,
$\kappa$-smooth potential $V$, a known mode at the origin, and exact
access to both $V$ and $\nabla V$. Normalize
$V(0)=\min V=0$ by one value query and subtraction when necessary.
For fixed $d\geq 2$, the ellipsoidal approximation in
\citet[arXiv v2, Corollary~48]{chewi2024lower}
supplies, in $O(\log(1+\kappa))$ queries,
\[
E=z+M\mathbb{B}_d\ \subseteq\ K:=\{x:V(x)\leq 1\}
\ \subseteq\ z+\rho M\mathbb{B}_d,
\]
where $M$ is invertible and $1\leq\rho=O(1)$.
The following argument explains why this geometry already permits
accuracy-independent expected cost when values are available.

\begin{proposition}[Value-based rejection benchmark]\label{prop:value-benchmark}
Given the displayed sandwich, there is an exact sampler for
$\pi(dx)=Z^{-1}e^{-V(x)}\,dx$ with at most
$e^2d!(2\rho)^d$ expected additional value queries under the usual
randomized stopping convention. Hence preprocessing and sampling
together cost $O(\log(1+\kappa))$ expected queries.
For $0<\varepsilon<1$, truncation gives an $\varepsilon$-TV sampler
with a deterministic budget
$O(\log(1+\kappa)+\log(1/\varepsilon))$.
Imposing physical-transcript stopping changes these bounds by only a
constant factor.

\end{proposition}

\begin{proof}
Because $0\in K\subseteq z+\rho M\mathbb{B}_d$, we have
$\|M^{-1}z\|\leq\rho$. Consequently
\[
K\subseteq E_0:=2\rho M\mathbb{B}_d,
\qquad
\operatorname{vol}(E_0)=(2\rho)^d\operatorname{vol}(E).
\]
Define
\[
r(x)=\frac{\|M^{-1}x\|}{2\rho}.
\]
We claim that $V(x)\geq r(x)-1$ for all $x$. If $r(x)\leq1$, this
follows from $V(x)\geq0$. If $r(x)>1$, then $x\notin K$, so
$V(x)>1$. Convexity and $V(0)=0$ give
\[
V\!\left(\frac{x}{V(x)}\right)\leq1.
\]
Thus $x/V(x)\in K\subseteq E_0$, which implies $r(x)\leq V(x)$
and proves the claim.

Use the full-support proposal
\[
q(x)=\frac{e^{-r(x)}}{d!\operatorname{vol}(E_0)}.
\]
Polar integration verifies
\[
\int_{\mathbb R^d}e^{-r(x)}\,dx
=d\operatorname{vol}(E_0)
  \int_0^\infty e^{-s}s^{d-1}\,ds
=d!\operatorname{vol}(E_0).
\]
To sample $q$, draw a uniform direction $\Theta$ on the unit sphere
and, independently, $R\sim\operatorname{Gamma}(d,1)$, and return
$Y_q=2\rho MR\Theta$. Here $R$ can be generated as a sum of $d$
independent unit exponentials, using only internal randomness.

Given $Y_q=x$, query $V(x)$ and accept with probability
\[
a(x)=\exp\bigl(r(x)-1-V(x)\bigr).
\]
The proved inequality ensures $0<a(x)\leq1$, and
\[
q(x)a(x)=\frac{e^{-V(x)}}{e\,d!\operatorname{vol}(E_0)}.
\]
Hence the accepted law is exactly $\pi$. Since $V\leq1$ on $E$,
\[
Z\geq e^{-1}\operatorname{vol}(E),
\qquad
p_{\mathrm{acc}}
=\frac{Z}{e\,d!\operatorname{vol}(E_0)}
\geq p_d:=\frac{1}{e^2d!(2\rho)^d}>0.
\]
Independent repetitions therefore terminate almost surely and use
at most $p_d^{-1}$ expected value queries. Adding the geometric
preprocessing and the optional normalization query proves the
expected-cost bound.

For a deterministic budget, attempt at most
\[
m_\varepsilon=\left\lceil
p_d^{-1}\log\frac1\varepsilon
\right\rceil
\]
trials and return a fixed point if all reject. Couple this algorithm
with the uncapped exact sampler. Their outputs can differ only if
the first $m_\varepsilon$ trials all reject, an event of probability
at most
\[
(1-p_{\mathrm{acc}})^{m_\varepsilon}
\leq e^{-p_dm_\varepsilon}\leq\varepsilon.
\]
The coupling inequality gives the TV guarantee and the stated
deterministic query budget.

Finally, to enforce the physical-transcript stopping rule, follow
each trial's privately drawn acceptance bit $b\in\{0,1\}$ by a
charged dummy oracle call at $b e_1$ and ignore its reply. Stop only
after the marker announcing acceptance, or after the marker in the
last capped trial. The query locations then reveal every stopping
decision. This at most doubles the additional query cost and does
not change either complexity order.
\end{proof}

The acceptance test above explicitly uses $V(x)$ and therefore does
not implement the oracle model of the main paper. Its role is to
separate the distinction between expected and deterministic cost
from the separate issue of implementing exact acceptance without
potential values. Corollary~\ref{up:noiseless-exact} addresses the latter issue while also
allowing an unknown, coarsely localized mode.

\section{A fixed-curvature lower bound for unrestricted localization}
\label{app:localization-lower}

This appendix isolates location hardness while keeping curvature
fixed. It uses the unrestricted class $\mathcal U_d$, defined in
Section~\ref{sec:model}, and exact gradient queries only.

\begin{proposition}[Logarithmic location cost]
\label{prop:location-lower}
Fix $d\geq1$ and $0<\varepsilon\leq1/10$. For every integer $m\geq1$,
put $R_m=8^{m+1}$. There is a finite family
\[
\{F_s:s\in\{-1,+1\}^m\}
\subseteq\mathcal C_d(2,1;0,R_m)
\]
such that every randomized gradient-query sampler with TV error at
most $\varepsilon$ on each family member satisfies
\[
\sup_s\E_s[T]\geq\frac{27m}{80}.
\]
The bound permits arbitrary real full-vector queries and adaptive
random stopping. Consequently, over $\mathcal U_d(2,1)$ with no
restriction on the minimizer,
\[
\inf_{\mathcal A\ {\rm uniformly}\ \varepsilon{\rm\mbox{-}accurate}}
\sup_{f\in\mathcal U_d(2,1)}\E_f[T]=+\infty.
\]
\end{proposition}

\paragraph{Hard instance: idea and proof sketch.}
In one dimension, construct a depth-$m$ binary family on an interval of
radius $8^{m+1}$.
At each level the active interval has two possible children, each of
one-eighth the radius and with disjoint interiors. Affine derivative
connectors make the alternatives agree outside the parent while keeping
all derivative slopes in $[1,2]$. Terminal modes are at least $16$ apart;
each target puts more than $99/100$ of its mass within distance four of
its mode. A TV-accurate sample therefore decodes the entire hidden string.

The children of $[c-r,c+r]$ are $[c-r/4,c]$ and $[c,c+r/4]$.
The connecting derivative slopes are $21/16$ and $7/4$, exchanged between
the two choices. Thus each hidden sign moves the possible zero without
changing external replies or the curvature bounds. The final target width
stays constant as the enclosing radius grows.

A crucial distinction from the curvature family is that one query
\emph{can} reveal more than one sign: it may land inside several nested
active intervals. An argument asserting one bit per real query would
therefore be invalid. The proof instead charges the expected number of
newly exposed signs. Long revelations are possible, but geometrically
unlikely under the hidden-sign prior, even when the query location is
chosen adaptively from all previous replies.

To control arbitrary real replies, augment the oracle by revealing the
hidden prefix needed to evaluate a query. At most one child interior
contains that query. Under uniform unrevealed signs, proceeding to each
further level has conditional probability at most one half, so a physical
query reveals at most two new signs in expectation. The suffix remains
uniform after every adaptive stopping history. Decoding with constant
success forces $\Omega(m)$ exposed signs and hence $\Omega(m)$ expected
queries. More precisely, if $K_T$ is the exposed-prefix length at
stopping, deferred decisions and nonnegative summation yield
\[
 \E[K_T]\leq2\E[T],\qquad
 \P(\text{correct decoding})
 \leq\E[2^{-(m-K_T)}]\leq\frac12+\frac{\E[K_T]}{2m}.
\]
The constant decoding probability therefore forces expected cost
linear in $m$. Unqueried signs cannot be inferred from the stopping
event itself: terminality depends only on the exposed prefix and
private randomness, not on the still-uniform suffix. This is why the
same lower bound covers algorithms that occasionally make very long
runs, rather than only deterministic-budget algorithms.
For the supplied-radius conclusion in Proposition~\ref{main:location},
let $R\geq64$ and take $m=\lfloor\log_8R\rfloor-1\geq1$; then
$8^{m+1}\leq R$, so all modes satisfy the supplied-ball promise, including
at $R=64$. Adding independent standard Gaussian coordinates preserves
the scalar decoding argument and the curvature bounds in every fixed
dimension. The full proof of Proposition~\ref{prop:location-lower} below supplies
the construction and random-stopping calculation. Combining this choice of
$m$ with Corollary~\ref{cor:known-radius} yields the matching radius order. This is a worst-case-over-locations statement, not a
per-instance lower bound, and it does not apply to the one-query
quadratic endpoint $L=\mu$.

\begin{proof}
It suffices to construct scalar witnesses and then use the Gaussian
completion in Section~\ref{fx:embedding}. We give a nested family whose
curvature bounds do not depend on its depth. An augmented oracle will
make the information calculation explicit; it only strengthens the
algorithm relative to exact-gradient access.

\paragraph{A binary family of uniformly curved derivatives.}
Write $R=R_m=8^{m+1}$ and $r_k=R/8^k$ for $0\leq k\leq m$.
For a sign string $s=(s_1,\ldots,s_m)$, put
\[
 c_0=0,\qquad c_k=\sum_{i=1}^k s_i r_i,
 \qquad I_k=[c_k-r_k,c_k+r_k].
\]
The two possible children of an interval $[c-r,c+r]$ have radius
$r/8$ and centers $c\pm r/8$. Their interiors are disjoint:
\[
 I_-=[c-r/4,c],\qquad I_+=[c,c+r/4].
\]
Define a continuous piecewise-affine derivative $p_s$ as follows.
Outside $I_0$, set $p_s(u)=3u/2$. On each of the two components of
$I_{k-1}\setminus\operatorname{int}I_k$, interpolate affinely between
its prescribed endpoint values
\[
 p_s(c_{k-1}\pm r_{k-1})=\pm\tfrac32 r_{k-1},
 \qquad p_s(c_k\pm r_k)=\pm\tfrac32 r_k.
\]
Finally, on $I_m$ set $p_s(u)=3(u-c_m)/2$.
When the chosen child is the right child, the left and right connector
slopes are respectively $21/16$ and $7/4$; for the left child these
slopes are exchanged. The exterior and terminal slopes are $3/2$.
Consequently the definitions agree at every joint and all slopes lie
in $[1,2]$. With $F_s(u)=\int_{c_m}^u p_s(v)\,dv$, integration gives
\[
 \tfrac12(u-v)^2
 \leq F_s(u)-F_s(v)-p_s(v)(u-v)
 \leq (u-v)^2 .
\]
Thus $F_s\in C^1$, with the advertised curvature bounds $1,2$.
These are bounds, not a requirement that either endpoint be attained.
Its unique minimizer is $c_m$, and
$|c_m|\leq\sum_{k=1}^m r_k<R/7<R$.

\paragraph{A sample identifies the entire string.}
Distinct centers are separated by at least $16$. Indeed, if $k$ is
the first differing sign, their separation is at least
$2r_k-2\sum_{j>k}r_j$; it equals $2r_m=16$ when $k=m$, and exceeds
$12r_k/7\geq12\cdot64/7>16$ when $k<m$.
Therefore the intervals
$D_s=[c_m-4,c_m+4]$ are pairwise disjoint.
The curvature bounds and $F_s(c_m)=0$ imply
\[
 Z_s\geq\int_{c_m-1}^{c_m+1}e^{-F_s(u)}\,du\geq2e^{-1},
 \qquad
 \int_{D_s^c}e^{-F_s(u)}\,du
 \leq2\int_4^\infty e^{-t^2/2}\,dt\leq\tfrac12 e^{-8}.
\]
For the final inequality use $1\leq t/4$ on $[4,\infty)$.
Hence $\nu_s(D_s^c)\leq e^{-7}/4<1/100$.
Any uniformly $\varepsilon$-accurate sampler has
$\P_s(Y\in D_s)>89/100$.
Under a uniform prior on the $m$ signs, decoding by the disjoint
intervals consequently succeeds with probability greater than $89/100$.

\paragraph{One arbitrary query reveals at most two new bits in expectation.}
Expose the signs only as needed to answer a query. Suppose a prefix
of length $k$ has been revealed. At $k=m$, the entire derivative
is known and no more signs are exposed. If $k<m$ and the queried scalar coordinate
$u$ is outside $\operatorname{int}I_k$, its derivative is already fixed
by that prefix, so reveal no further sign. Otherwise reveal $s_{k+1}$.
If $u$ is outside the interior of the selected child, the connector
formula determines its derivative and the query is answered. If it
is inside, continue recursively. Stop also at depth $m$, where the
terminal affine formula determines the reply. At a child boundary the
derivative is fixed and no deeper sign is needed.
Return both the derivative and the newly exposed signs. This is an
augmented oracle, and an exact-gradient algorithm can be simulated by
ignoring the exposed signs. Previously answered queries impose no
restriction on the unexposed suffix: their replies are functions only
of the exposed prefix. Thus, conditionally on the augmented past and
all private randomness used to choose the next query, the remaining
signs are still independent and uniform.

At most one child interior contains a fixed $u$. Every additional
level after the first therefore requires an independent sign of
conditional probability at most $1/2$. If $\Delta_j$ is the number
of newly exposed signs at the $j$th query, then, on a reached query,
\begin{align*}
 \P(\Delta_j\geq a\mid\hbox{incoming augmented history})
 &\leq2^{-(a-1)}\quad(a\geq1),\\
 \E[\Delta_j\mid\hbox{incoming augmented history}]&\leq2.
\end{align*}
These inequalities hold for arbitrary real query locations, including
points outside $I_0$ and points at any connector boundary.
Let $K_T$ be the exposed-prefix length when the algorithm stops.
If the prior-average expected call count is infinite, there is nothing
to prove. Otherwise nonnegative summation and conditioning before each
reached query give
\begin{equation}\label{loc:exposure-cost}
 \E[K_T]
 =\sum_{j\geq1}\E[\mathbf1_{\{T\geq j\}}\Delta_j]
 \leq2\sum_{j\geq1}\P(T\geq j)=2\E[T].
\end{equation}
This argument permits arbitrary adaptive random stopping; it uses no
fixed-horizon reduction.

\paragraph{Decoding forces a linear number of exposures.}
Conditionally on the terminal augmented transcript, the unrevealed
suffix remains uniform. This deferred-decisions fact also holds at a
stopping time: whether that transcript is terminal is determined by
its exposed information and private randomness, not by unused signs.
Adjoin any remaining private randomness used for the output; it is
independent of the unexposed suffix. A fixed output can decode at
most one of the $2^{m-K_T}$ completions. Hence
\[
 \P(\widehat s(Y)=s)
 \leq\E[2^{-(m-K_T)}]
 \leq\frac12+\frac{\E[K_T]}{2m}.
\]
The last inequality follows from convexity: for $0\leq k\leq m$,
$2^{k-m}$ lies below the chord joining $(0,2^{-m})$ and $(m,1)$,
which in turn lies below $1/2+k/(2m)$.
Success greater than $89/100$ implies
$\E[K_T]>39m/50$. By~\eqref{loc:exposure-cost},
\[
 \sup_s\E_s[T]\geq\E[T]>\frac{39m}{100}\geq\frac{27m}{80}.
\]
For $d>1$, append $\tfrac12\|x_{2:d}\|^2$ and return the exact known
gradients in those coordinates. The same argument applies directly
to the first coordinate of every full-vector query and to the first
coordinate of the output. Finally, $m$ is arbitrary, so a uniformly
accurate algorithm on $\mathcal U_d(2,1)$ cannot have a finite
worst-case expected query count. Taking the infimum proves the claim.
\end{proof}

\begin{proof}[Proof of Proposition~\ref{main:location}]
For a supplied radius $R\geq64$, choose
$m=\lfloor\log_8R\rfloor-1$ in the construction. Its minimizers lie
inside $\mathbb{B}(0,R)$. Writing $t=\log_8R\ge2$, we have
$m=\lfloor t\rfloor-1\ge t/3$: for $2\le t<3$ use $m=1$;
for $t\ge3$ use $m>t-2\ge t/3$.
Combined with Corollary~\ref{cor:known-radius},
this gives $\Theta(\log(1+R))$ noiseless expected minimax cost
at $\mu=1$, $L=2$, uniformly for $0<\varepsilon\leq1/10$.
This is a worst-case-over-locations statement, not a per-instance
lower bound. It does not assert joint optimality for all curvature
ratios or extend the lower bound to a value oracle. In particular,
it does not apply to the quadratic endpoint $L=\mu$.
\end{proof}

\endgroup
\end{document}